\documentclass[reqno,11pt]{amsart}
\usepackage{amsmath,amssymb,latexsym,soul,cite,mathrsfs}
\usepackage{color,enumitem,graphicx}
\usepackage[colorlinks=true,urlcolor=blue,
citecolor=red,linkcolor=blue,linktocpage,pdfpagelabels,
bookmarksnumbered,bookmarksopen]{hyperref}
\usepackage[english]{babel}
\usepackage[left=2.6cm,right=2.6cm,top=2.9cm,bottom=2.9cm]{geometry}
\usepackage[hyperpageref]{backref}
\usepackage[toc]{appendix}

\newtheorem{theorem}{Theorem}
\newtheorem{lemma}[theorem]{Lemma}
\newtheorem{corollary}[theorem]{Corollary}
\newtheorem{proposition}[theorem]{Proposition}
\newtheorem{remark}[theorem]{Remark}
\newtheorem{definition}[theorem]{Definition}

\newtheorem{theoremletter}{Theorem}
\newtheorem{propositionletter}{Proposition}

\newtheorem{lemmaletter}{Lemma}

\newtheoremstyle{tttheorem}
{}                
{}                
{\slshape}        
{}                
{\bfseries}       
{'}               
{ }               
{}                
\theoremstyle{tttheorem}
\newtheorem{theoremtio}{Theorem}

\newcommand{\capt}{{\mathrm{cap}}}

\newcommand{\ud}{\mathrm{d}}
\newcommand{\loc}{\mathrm{loc}}

\DeclareMathOperator{\sgn}{sgn}

\title[Qualitative properties for solutions to fourth order critical systems]{Qualitative properties for solutions to conformally invariant fourth order critical systems} \thanks{Research supported in part by Fulbright Commission in Brazil grant G-1-00001, CNPq grant 305726/2017-0, and the Coordena\c c\~ao de Aperfei\c coamento de Pessoal de N\'ivel Superior - Brasil (CAPES) grant 88882.440505/2019-01}

\author[J.H. Andrade]{Jo\~{a}o Henrique\ Andrade}
\author[J.M. do \'O]{Jo\~ao Marcos do \'O*}

\address[J.H. Andrade]{Department of Mathematics, 
	Federal University of Para\'{\i}ba
	\newline\indent 
	58051-900, Jo\~ao Pessoa-PB, Brazil}
\email{\href{mailto:andradejh@mat.ufpb.br}{andradejh@mat.ufpb.br}}

\address[J.M. do \'O]{Department of Mathematics,
	Federal University of Para\'{\i}ba
	\newline\indent 
	58051-900, Jo\~ao Pessoa-PB, Brazil}
\email{\href{mailto:jmbo@pq.cnpq.br}{jmbo@pq.cnpq.br}}

\thanks{* Corresponding author.}

\begin{document}
	
	\begin{abstract}
		We study qualitative properties for nonnegative solutions to a conformally invariant coupled system of fourth order equations involving critical exponents.  For solutions defined in the punctured space, there exist essentially two cases to analyze.  If the origin is a removable singularity, we prove that non-singular solutions are rotationally invariant and weakly positive. More precisely, they are the product of a fourth order spherical solution by a unit vector with nonnegative coordinates.  If the origin is a non-removable singularity, we show that the solutions are radially symmetric and strongly positive. Furthermore, using a Pohozaev-type invariant, we prove the non-existence of semi-singular solutions, that is, all components equally blow-up in the neighborhood of origin. Namely, they are classified as multiples of the Emden--Fowler solution. Our results are natural generalizations of the famous classification due to [L. A. Caffarelli, B. Gidas and J. Spruck, Comm. Pure Appl. Math. (1989)] on the classical singular Yamabe equation.
	\end{abstract}
	
	\maketitle
	
	
	\begin{center}
	\footnotesize
	\tableofcontents
	\end{center}
	
	\section{Description of the results}
	
	We study qualitative properties for nonnegative {\it $p$-map} solutions $\mathcal{U}=(u_1,\dots,u_p): \mathbb{R}^{n}\setminus\{0\}\rightarrow \mathbb{R}^p$ to the following fourth order system in the {\it punctured space},
	\begin{equation}\label{oursystem}\tag{$\mathcal{S}_p$}
		\Delta^{2} u_{i}=c(n)|\mathcal{U}|^{2^{**}-2}u_{i} \quad {\rm in} \quad \mathbb{R}^{n}\setminus\{0\},
	\end{equation}
	where  $n\geqslant5$, 
	$\Delta^{2}$ is the bi-Laplacian and $|\mathcal{U}|$ is the Euclidean norm, that is, $|\mathcal{U}|=(\sum_{i=1}^{p} u_i^2)^{1/2}$.
	System \eqref{oursystem} is strongly coupled by the {\it Gross--Pitaevskii nonlinearity}  $f_i(\mathcal{U})=c(n)|\mathcal{U}|^{2^{**}-2}u_i$ with associated potential $F(\mathcal{U})=(f_1(\mathcal{U}),\dots,f_p(\mathcal{U}))$, where $s\in(1,2^{**}-1)$ with $2^{**}=2n/(n-4)$ the {\it critical Sobolev exponent}, and $c(n)=[n(n-4)(n^2-4)]/{16}$ a normalizing constant.
	
	By a {\it $($classical$)$ solution} to System \eqref{oursystem}, we mean a $p$-map $\mathcal{U}$ such that each component $u_i \in C^{4,\zeta}(\mathbb{R}^{n}\setminus\{0\})$, for some $\zeta\in(0,1)$, and it satisfies \eqref{oursystem} in the classical sense.
	A solution may develop an isolated singularity when $x=0$, that is, some components may have a non-removable singularity at the origin. 
	More accurately, a solution to \eqref{oursystem} is said to be {\it singular}, if there exists $i\in I:=\{1,\dots,p\}$ such that the origin is a {\it non-removable singularity} for $u_{i}$. Otherwise, if the origin is a {\it removable singularity} for all components, this solution is called {\it non-singular}, and it can be extended continuously to the whole domain. 
	
	Let us notice that when $p=1$, \eqref{oursystem} becomes the following fourth order equation,
	\begin{equation}\label{scalarsystem}\tag{$\mathcal{S}_1$}
		\Delta^{2}u=c(n)u^{2^{**}-1} \quad {\rm in} \quad \mathbb{R}^n\setminus\{0\}.
	\end{equation}
	In this sense, the Gross--Pitaevskii nonlinearity is the more natural coupling term such that \eqref{oursystem} generalizes \eqref{scalarsystem}. Our objective is to present classification results for both non-singular and singular solutions to our conformally invariant system \eqref{oursystem}. 
	
	Our first main result is motivated by the fundamental classification theorem due to C. S. Lin \cite[Theorem~1.3]{MR1611691} (see also X. Xu \cite[Theorem~1.1]{MR1769247}) for positive solutions to \eqref{scalarsystem} with a removable singularity at the origin, which can be stated as follows
	\begin{theoremletter}\label{theoremA}
		Let $u$ be a nonnegative non-singular solution to \eqref{scalarsystem}. Then, there exist $x_0\in\mathbb{R}^n$ and $\mu>0$ such that $u$ is radially symmetric about $x_0$ and 
		\begin{equation}\label{sphericalfunctions}
			u_{x_0,\mu}(x)=\left(\frac{2\mu}{1+\mu^{2}|x-x_0|^{2}}\right)^{\frac{n-4}{2}}. 
		\end{equation}
		Let us call $u_{x_0,\mu}$ a fourth order spherical solution.
	\end{theoremletter}
	This $(n+1)$-parameter family of solutions can also be regarded as maximizers for the Sobolev embedding theorem $\mathcal{D}^{2,2}(\mathbb{R}^n)\hookrightarrow L^{2^{**}}(\mathbb{R}^n)$, that is,
	\begin{equation*}
		\|u_{x_0,\mu}\|_{L^{2^{**}}(\mathbb{R}^n)}= S(n)\|u_{x_0,\mu}\|_{\mathcal{D}^{2,2}(\mathbb{R}^n)} \quad \mbox{with} \quad S(n)=\left(c(n)\omega_n^{{4}/{n}}\right)^{-1/2},
	\end{equation*}
	where $\omega_n$ is the volume of the unit ball in the $n$-dimensional Euclidean space.
	The existence of extremal functions for the last identity was obtained by P.-L. Lions \cite[Section~V.3]{MR778970}. 
	Besides, these optimizers were found in a more general setting by E. Lieb \cite[Theorem~3.1]{MR717827}, using an equivalent dual formulation.
	Let us also mention the related result in D. E. Edmunds et al. \cite[Theorem~2.1]{MR1076074}. 
	
	Our second main result yields a classification theorem for nonnegative singular solutions to  \eqref{oursystem}. 
	On this subject, we should mention that when the origin is a non-removable singularity, C. S. Lin \cite[Theorem~1.4]{MR1611691} obtained radial symmetry for solutions to \eqref{scalarsystem} using the asymptotic moving planes technique. Recently, Z. Guo, Huang, Wang and J. Wei. \cite[Theorem~1.3]{MR4094467} proved the existence of periodic solutions applying a mountain pass theorem and conjectured that all solutions should be periodic.
	Later on, R. L. Frank and T. K\"onig \cite[Theorem~2]{MR3869387}  answered this conjecture, obtaining more accurate results concerning the classification for global singular solutions to \eqref{scalarsystem}. More precisely, they used the Emden--Fowler change coordinates (see Section~\ref{sec:cylindricaltransform}) to transform \eqref{scalarsystem} into the following fourth order Cauchy problem,
	\begin{equation}\label{fowler4order}
		\begin{cases}
			v^{(4)}-K_2v^{(2)}+K_0v=c(n)v^{2^{**}-1} \quad \mbox{in} \quad \mathbb{R},\\
			v(0)=a,\ v^{(1)}(0)=0,\ v^{(2)}(0)=b(a),\ v^{(3)}(0)=0,
		\end{cases}
	\end{equation}    
	where $K_2,K_0$ are constants depending on the dimension (see \eqref{coefficients}).
	In this work, positive periodic solutions $v_{a,T}$ to \eqref{fowler4order} are proved to exist using a topological shooting method based on the parameter $b(a)$. 
	One needs to be restricted to the situation $a\in (0,a_0]$, where $a_0=[n(n-4)/(n^2-4)]^{(n-4)/8}$ and $T\in (0,T_{a}]$ is the period with $T_{a}\in\mathbb{R}$ is the fundamental period of $v_a$.
	
	\begin{theoremletter}\label{theoremB}
		Let $u$ be a positive singular solution to \eqref{scalarsystem}. Then, $u$ is radially symmetric about the origin.
		Moreover, there exist $a \in (0,a_0]$ and $T\in (0,T_{a}]$ such that
		\begin{equation}\label{emden-folwersolution}
			u_{a,T}(x)=|x|^{\frac{4-n}{2}}v_{a}(\ln|x|+T),
		\end{equation}
		where $v_{a}$ is the unique $T$-periodic bounded solution to \eqref{fowler4order} and $T_a\in\mathbb{R}$ its fundamental period. Let us call both $u_{a,T}$ and $v_{a,T}$ Emden--Fowler $($or Delaunay-type$)$ solutions.
	\end{theoremletter}
	
	Let us remark that differently from Theorem~\ref{theoremA} where the solution can be classified by $(n+1)$-parameter family, in Theorem~\ref{theoremB} we have a two-parameter family of solutions. 
	However, we should mention that it is possible to construct a $n$-parameter family of deformations for \eqref{emden-folwersolution}, which are called the {\it deformed Emden--Fowler solutions} \cite[page 241]{MR1666838}, which are constructed by composing three conformal transformations. 
	In this sense, the {\it necksize} $a\in (0,a_0]$ of a singular solution to \eqref{scalarsystem} plays a similar role as the parameter $\mu>0$ for the non-singular solutions to \eqref{scalarsystem}. 
	
	In the light of Theorems~\ref{theoremA} and \ref{theoremB}, we present our main results.
	
	\begin{theorem}[Liouville--type]\label{theorem1}
		Let $\mathcal{U}$ be a nonnegative non-singular solution to \eqref{oursystem}. Then, there exists $\Lambda \in\mathbb{S}^{p-1}_{+}=\{ x \in \mathbb{S}^{p-1} : x_i \geqslant 0 \}$ and a fourth order spherical solution given by \eqref{sphericalfunctions} such that
		\begin{equation*}
			\mathcal{U}=\Lambda u_{x_0,\mu}.
		\end{equation*}
	\end{theorem}
	
	As an application, we show that the non-singular solutions classified above are the extremal maps for a higher order Sobolev-type inequality. Moreover, the best constant associated with this embedding coincides with the one when $p=1$ \cite{MR2221095,MR2852264}. 
	
	\begin{theorem}[Classification]\label{theorem2}
		Let $\mathcal{U}$ be a nonnegative singular solution to \eqref{oursystem}. Then, there exists $\Lambda^*\in\mathbb{S}^{p-1}_{+,*}=\{ x \in \mathbb{S}^{p-1} : x_i > 0 \}$ and an Emden--Fowler solution given by \eqref{emden-folwersolution} such that 
		\begin{equation*}
			\mathcal{U}=\Lambda^* u_{a,T}.
		\end{equation*}
	\end{theorem}
	Since singular solutions to the blow-up limit equation \eqref{oursystem} are the natural candidates for asymptotic models of the same system in the punctured ball, the last theorem is the first step in describing the local asymptotic behavior for positive singular solutions to 
	\begin{equation*}
		\Delta^{2} u_{i}=c(n)|\mathcal{U}|^{2^{**}-2}u_{i} \quad {\rm in} \quad B^{n}_{1}\setminus\{0\}.
	\end{equation*}
	This asymptotic analysis would be a version of the celebrated results due to L. A. Caffarelli, B. Gidas and J. Spruck \cite{MR982351} and N. Korevaar, R. Mazzeo, F. Pacard and R. Schoen \cite{MR1666838} for the context of fourth order strongly coupled systems. 
	When $p=1$, the subcritical cases of \eqref{scalarsystem} were addressed in \cite{MR1436822,MR4123335}.
	However, the problem of describing the local behavior for singular solutions to the critical equation \eqref{scalarsystem} near the isolated singularity remains unsolved; this question was posed by R. L. Frank and T. K\"{o}nig \cite[page 1103]{MR3869387}.
	
	\begin{remark}\label{existence} 
		The existence of non-singular $($singular$)$ solutions to \eqref{oursystem} follows directly from Theorem~\ref{theoremA} $($Theorem~\ref{theoremB}$)$. 
		In fact, for any $\Lambda\in\mathbb{S}^{p-1}_{+}$ $($$\Lambda^*\in\mathbb{S}^{p-1}_{+,*}$$)$, we observe that $\mathcal{U}=\Lambda u_{x_0,\mu}$ $($$\mathcal{U}=\Lambda^* u_{a,T}$$)$ is a non-singular $($singular$)$ solution to \eqref{oursystem}. 
		Roughly speaking, our results classify these solutions as the only possible expressions for nontrivial solutions to \eqref{oursystem}. 
	\end{remark}
	
	Now we will compare our results with their second order counterparts.
	One of the first results on the classification for solutions to second order critical equations dates back to the seminal work of L. A. Caffarelli et al. \cite{MR982351}. This challenging analysis for singular PDEs was motivated by the classical {\it Lane--Emden--Fowler equation} 
	\begin{equation}\label{lane-emden}
		-\Delta u=u^{s} \quad {\rm in} \quad \mathbb{R}^n\setminus\{0\},
	\end{equation}
	for $n\geqslant3$ and $s>1$, which models the distribution of mass density in spherical polytropic star in hydrostatic equilibrium \cite{MR0092663}. In addition, when $s=2^{*}-1$, where $2^{*}:=2n/(n-2)$ is the critical Sobolev exponent, \eqref{lane-emden} corresponds, up to a normalizing constant, to the conformally flat singular scalar curvature equation, a famous problem in differential geometry, which can be set as 
	\begin{equation}\label{secscalarsystem}
		-\Delta u =\frac{n(n-2)}{4}u^{2^{*}-1} \quad {\rm in} \quad \mathbb{R}^n\setminus\{0\}.
	\end{equation}
	It is well known that \eqref{secscalarsystem} is a particular case of the Yamabe problem on a non-compact complete Riemannian manifold $(M^n,g)$ with {\it simple structure at infinity}, that is, there exists $\widetilde{M}^n$ containing $M^n$ such that $M^n=\widetilde{M}\setminus \mathcal{Z}$, where $\mathcal{Z}$ is a closed subset called the {\it singular set} of $M^n$. Thus, this problem can be reduced to obtaining positive solution to the {\it singular Yamabe equation}
	\begin{equation}\label{yamabe}
		\begin{cases}
			-\Delta_g u+\frac{(n-2)}{4(n-1)}R_gu=\frac{n(n-2)}{4}u^{2^{*}-1} \quad {\rm on} \quad \widetilde{M}\setminus \mathcal{Z}\\
			\displaystyle\lim_{d_g(x,\mathcal{Z})\rightarrow0}u(x)=\infty,
		\end{cases}
	\end{equation}
	where $-\Delta_g$ is the Laplace--Beltrami operator and $R_g$ is the scalar curvature. 
	In this way, \eqref{secscalarsystem} is related to \eqref{yamabe} when $\widetilde{M}^n=(\mathbb{S}^{n-1},g_0)$ is the standard sphere with the round metric, and $\mathcal{Z}$ is a unique point.
	The geometric operator $L_g:=-\Delta_g+\frac{(n-2)}{4(n-1)}R_g$ on the left-hand side of \eqref{yamabe} is the so-called {\it conformal Laplacian}.
	
	The study of singular solutions to geometric equations like \eqref{yamabe} is related to the characterization of the size of the limit set of the image domain in the round sphere ($\mathbb{S}^{n},g_0)$ of the developing map for a locally conformally flat $n$-dimensional manifold \cite{MR929283,MR931204}. 
	Notice that conformal metrics $\bar{g}=u^{4/(n-2)}g_0$ with constant scalar curvature, are generated by using a positive solutions to \eqref{yamabe} as conformal factor. These conformal metrics are complete, whenever singular solutions to \eqref{yamabe} have a suitable blow-up rate close to the singular set. Then, for the geometrical point of view, to study the local behavior for singular solutions to \eqref{yamabe} near the singular set is essential to understand the asymptotic behavior of conformal metrics near the singular set.
	
	In \cite{MR982351}, using ODE methods, the whole set of solutions to \eqref{secscalarsystem} was classified. 
	More accurately, it was proved that if $u$ is a non-singular solution to \eqref{secscalarsystem}, then there exist $x_0\in\mathbb{R}^n$ and $\mu>0$ such that 
	\begin{equation}\label{secaubintalenti}
		u(x)=\left(\frac{2\mu}{1+\mu^2|x-x_0|^2}\right)^{\frac{n-2}{2}}.
	\end{equation}
	This classification result can be seen as a complement to the works of T. Aubin \cite{MR0448404} and G. Talenti \cite{MR0463908}. Moreover, they as well dealt with the case when the origin is a non-removable singularity, proving that if $u$ is a singular solution to \eqref{secscalarsystem}, there exist $a \in (0,[(n-2)/n]^{(n-2)/4}]$ and $T\in (0,T_a]$ such that
	\begin{equation}\label{secfowler}
		u(x)=|x|^{\frac{2-n}{2}}v_{a}(-\ln|x|+T),
	\end{equation}
	where $v_{a,T}$ is the unique $T$-periodic bounded solution to the following second order problem 
	\begin{equation}\label{secondorderfowler}
		\begin{cases}
			v^{(2)}-\displaystyle\frac{(n-2)^2}{4}v+\frac{n(n-2)}{4}v^{2^{*}-1}=0 \quad \mbox{in} \quad \mathbb{R},\\
			v(0)=a,\ v^{(1)}(0)=0,
		\end{cases}
	\end{equation}
	where $T_a\in\mathbb{R}$ is the fundamental period of $v_a$. In this situation, asymptotic properties for global solutions to \eqref{secondorderfowler} can be inferred using standard ODE methods, such as conservation of energy, phase-plane analysis and Floquet theory (or Bloch wave theory).
	
	Now let us consider nonnegative $p$-map solutions $\mathcal{U}=(u_1,\dots,u_p): \mathbb{R}^{n}\setminus\{0\}\rightarrow \mathbb{R}^p$ to the following critical second order Gross--Pitaevskii system extending \eqref{secscalarsystem},
	\begin{equation}\label{secondordersystem}
		-\Delta u_{i}=\frac{n(n-2)}{4}|\mathcal{U}|^{2^{*}-1}u_{i} \quad {\rm in} \quad \mathbb{R}^{n}\setminus\{0\}.
	\end{equation} 
	As in Remark~\ref{existence}, we observe that the existence of non-singular (singular) solutions to \eqref{secondordersystem} is a direct consequence of the existence results due to P.-L. Lions \cite{MR778970} (R. Fowler \cite{fowler}). Indeed, for every $\Lambda\in \mathbb{S}^{p-1}_+$ ($\Lambda^*\in\mathbb{S}^{p-1}_{+,*}$) unit vector with nonnegative (positive) coordinates and $u$ a non-singular (singular) solution to \eqref{secscalarsystem}, we have that $\mathcal{U}=\Lambda u$ is a non-singular (singular) solution to \eqref{secondordersystem}. Moreover, O. Druet, E. Hebey and J. V\'etois \cite[Proposition~1.1]{MR2558186} on System \eqref{secondordersystem} proved the Liouville-type theorem stated below. We also refer to \cite[Theorem~1]{MR2510000} for related results on integral systems with critical exponents.
	
	\begin{theoremletter}\label{theoremC}
		Let $\mathcal{U}$ be a nonnegative non-singular solution to \eqref{secondordersystem}. Then, $\mathcal{U}=\Lambda u$ for some $\Lambda\in\mathbb{S}^{p-1}_{+}$, where $u$ is given by \eqref{secaubintalenti}.
	\end{theoremletter}
	
	A natural question that arises is whether Theorem~\ref{theoremC} still holds in the singular case. Recently, R. Caju, J. M do \'O and A. Santos \cite[Theorem~1.2]{MR4002167} gave an affirmative answer for this.
	
	\begin{theoremletter}
		Let $\mathcal{U}$ be a nonnegative singular solution to \eqref{secondordersystem}. Then, $\mathcal{U}=\Lambda^* u$ for some $\Lambda^*\in\mathbb{S}^{p-1}_{+,*}$, where $u$ is given by \eqref{secfowler}.
	\end{theoremletter}
	
	Strongly coupled fourth order systems appear in several important branches of mathematical physic. For instance, in hydrodynamics, for modeling the behavior of deep-water and Rogue waves in the ocean \cite{lo-mei}. 
	Also, in the Hartree--Fock theory for Bose--Einstein double condensates \cite{MR2040621}.
	Moreover, in conformal geometry, \eqref{scalarsystem} is the limit equation of the conformally constant $Q$-curvature problem. 
	Hence, in the same way of the singular Yamabe problem, solutions to \eqref{scalarsystem} give rise to complete conformal metrics with a constant $Q$-curvature. 
	For more details on the $Q$-curvature problem and some applications, see, for instance, \cite{MR3618119}.
	Motivated by its applications in nonlinear analysis, minimal surface theory, and differential geometry, classification for singular solutions to PDEs has been a topic of intense study in recent years. 
	There exists a vast literature for problems this arising in conformal geometry. 
	For instance, in prescribing different types of curvature in differential geometry, such as the higher order $Q$-curvature, the fractional curvature, and the $\sigma_k$-curvature.
	
	The primary sources of difficulties in seeking qualitative properties for fourth order systems like \eqref{oursystem} are the lack of maximum principle and the failure of truncation methods provoked by the fourth order operator on the left-hand side of \eqref{oursystem}, the coupled setting caused by the Gross--Pitaevskii nonlinearity in the right-hand side of \eqref{oursystem}.
	In both theorems, we study the PDE or ODE satisfied by the quotient of any two strictly positive components.
	This requires significant technical manipulations with higher derivatives of quotients.
	In the non-singular case, we obtain a strong Liouville-type for solutions of the linear fourth order equation satisfied by the quotient of components.
	
	The proof of Theorems~\ref{theorem1} uses the moving sphere technique for each component. 
	The main difficulty is that due to the Gross--Pitaevskii nonlinearity, it may occur that the process does not hold for some components, which we prove is not the case.
	Our technique relies on recovering regularity and superharmonicity properties for each component solution, based on a comparison with the norm of the vectorial solution.
	Another way to prove this theorem is to prove that the classification holds for weak solutions and that classical solutions satisfy an estimate of the $L^2$-norm of its Laplacian.
	Theorem \ref{theorem2} is proved using the moving planes technique, which shows that all components solutions are rotationally invariant and radially monotonically decreasing.
	The first step in our argument uses the fact that any component solution cannot vanish unless it is identically zero.
	The second step is the analysis of the Pohozaev invariant, which provides a removable singularity classification theorem.
	In this case, we can prove that all components are strictly positive and blow-up at the origin with the same prescribed asymptotic rate.
	
	Here is a brief description of our plan for the remaining part of this paper.
	In Section~\ref{section2}, we summarize some basic definitions. 
	In Section~\ref{section3}, we  prove that solutions to \eqref{oursystem} are non-singular and weakly positive. 
	Also, we show that Theorem~\ref{theorem1} holds for weak solutions to \eqref{oursystem}. 
	Hence, we apply a moving spheres method to prove the classification for each component.
	Using the classification form the norm of a vectorial solution we
	show that classical solutions are weak solutions as well. 
	Hence, a direct integral method is used to given an alternative proof of the classification in Theorem~\ref{theorem1} for weak solutions.
	We also prove that solutions from Theorem~\ref{theorem1} are extremal functions for a Sobolev embedding theorem. 
	In Section~\ref{section4}, we obtain that singular solutions are as well classical. Thus, we employ an asymptotic moving planes method to show they are rotationally invariant about the origin. Therefore, on the singular case \eqref{oursystem} is equivalent to a fourth order ODE system in the real line. 
	In this direction, we use its Hamiltonian energy to define a suitable Pohozaev-type invariant. Finally, we perform a delicate ODE analysis to prove a removable-singularity classification for solutions to \eqref{oursystem} based on the Pohozaev invariant sign. 
	Then, as a direct consequence, we give the proof of Theorem~\ref{theorem2}.

Here is a brief description of our plan for the remaining part of this paper.
In Section~\ref{section2}, we summarize some basic definitions. 
In Section~\ref{section3}, we  prove that solutions to \eqref{oursystem} are non-singular and weakly positive. 
In addition, we show that Theorem~\ref{theorem1} holds for weak solutions to \eqref{oursystem}. 
Hence, we apply a moving spheres method to prove the classification for each component.
Using the classification form the norm of a vectorial solution we
show that classical solutions are weak solutions as well. 
Hence, a direct integral method is used to given an alternative proof of the classification in Theorem~\ref{theorem1} for weak solutions.
Besides, we also prove that solutions from Theorem~\ref{theorem1} are extremal functions for a Sobolev embedding theorem. 
In Section~\ref{section4}, we obtain that singular solutions are as well classical. Thus, we employ an asymptotic moving planes method to show they are rotationally invariant about the origin. Therefore, on the singular case \eqref{oursystem} is equivalent to a fourth order ODE system in the real line. 
In this direction, we use its Hamiltonian energy to define a suitable Pohozaev-type invariant. Finally, we perform a delicate ODE analysis to prove a removable-singularity classification for solutions to \eqref{oursystem} based on the sign of the Pohozaev invariant. 
Then, as a direct consequence, we give the proof of Theorem~\ref{theorem2}.

\section{Preliminaries}\label{section2}
We set some background definitions and tools that will be used later in this text. 
First, we define some standard concepts for elliptic systems.
Second, we introduce the vectorial fourth order versions of the Kelvin transform and the cylindrical transformation, which will be used in the sequel to run the sliding techniques.

\subsection{Basic definitions}
Let us introduce some basic definitions used in the remaining part of the text. 
Here and subsequently, we always deal with non-trivial nonnegative solutions $\mathcal{U}$ of \eqref{oursystem}, that is, $u_i \geqslant 0$ for all $i\in I$ and $|\mathcal{U}|\not\equiv0$, where we recall the notation $I=\{1,\dots,p\}$.
We split the index set $I$ into two parts $I_0=\{i\in I : u_i\equiv0\}$ and $I_{+}=\{i\in I : u_i>0\}$. 
Then, following standard notation for elliptic systems, we divide solutions to \eqref{oursystem} into two types.
\begin{definition}
	Let $\mathcal{U}$ be a nonnegative solution to \eqref{oursystem}. We call $\mathcal{U}$ {\it strongly positive} if $I_+=I$. On the other hand, when $I_0\neq\emptyset$, we say that $\mathcal{U}$ is {\it weakly positive}. 
\end{definition}

\begin{remark}
	For the proof of Theorems~\ref{theorem1} and \ref{theorem2}, it is crucial to show that solutions to \eqref{oursystem} are weakly positive. 
	We need to guarantee that nontrivial solutions to \eqref{oursystem} do not develop zeros in the domain. 
	Namely, our strategy is to prove that the so-called {\it quotient function} $q_{ij}=u_i/u_j$ is constant for all $i,j\in I_+$. First, for the quotient to be well defined, the denominator must be strictly positive. 
	Notice that contrary to the case $p=1$, nonnegative solutions to some inhomogeneous elliptic coupled systems are not necessarily weakly positive, and thus not strongly positive as well. 
\end{remark}

When $\liminf_{|x|\rightarrow0}|\mathcal{U}(x)|=\infty$, we call $\mathcal{U}$ a {\it singular solution}. 
In this case, some components might develop a non-removable singularity at the origin.  
We will divide singular solutions into two classes.
Namely, a solution to \eqref{oursystem}  is called {\it fully-singular}, if the origin is a {\it non-removable singularity} for all component solution $u_i$. Otherwise, we say that $\mathcal{U}$ is {\it semi-singular}. 
More precisely, we present the following definitions.

\begin{definition}
	For $\mathcal{U}$ a nonnegative singular solution to \eqref{oursystem}, let us define its blow-up set by $I_{\infty}=\{i\in I: \liminf_{|x| \rightarrow 0 }u_i(|x|)=\infty\}$.
\end{definition}

It is easy to observe that $\mathcal{U}$ being a singular solution to \eqref{oursystem} is equivalent to $I_{\infty}\neq \emptyset$. 
Hence, in terms of the blow-up set's cardinality, we divide singular solutions to \eqref{oursystem} as follows.

\begin{definition}\label{singset}
	Let $\mathcal{U}$ be a nonnegative singular solution to \eqref{oursystem}. 
	We say that $\mathcal{U}$ is {\it fully-singular} if $I_{\infty}=I$. Otherwise, if $ I_{\infty} \neq I $ we call $\mathcal{U}$ {\it semi-singular}.
\end{definition}

\begin{definition}
	Let $\Omega=\mathbb{R}^n$ $($$\Omega=\mathbb{R}^n\setminus\{0\}$ be the punctured space$)$ be the whole space, and $\mathcal{U}$ be a nonnegative non-singular $($singular$)$ solution to \eqref{oursystem}. We say that $\mathcal{U}$ is a {\it weak solution}, if it belongs to $\mathcal{D}^{2,2}(\Omega,\mathbb{R}^p)$ and satisfies \eqref{oursystem} in the weak sense, that is, for all nonnegative $\Phi\in  C^{\infty}_c(\Omega,\mathbb{R}^p)$, one has 
	\begin{equation}\label{weakformulattion}
		\displaystyle\int_{\mathbb{R}^n}\Delta u_i\Delta \phi_i \ud x=\displaystyle c(n)\int_{\mathbb{R}^n}|\mathcal{U}|^{2^{**}-2}u_{i}\phi_i \ud x.
	\end{equation}
	Here $\mathcal{D}^{2,2}(\Omega,\mathbb{R}^p)$ is the classical Beppo--Levi space, defined as the completion of the space of compactly supported smooth $p$-maps, denoted by $C^{\infty}_{c}(\Omega,\mathbb{R}^p)$, under the Dirichlet norm $\|\mathcal{U}\|_{\mathcal{D}^{2,2}(\Omega,\mathbb{R}^p)}^2=\sum_{i=1}^p\|\Delta u_i\|_{L^{2}(\Omega)}^2$.
\end{definition}

\begin{remark}\label{classical/weak}
	In what follows, we use classical regularity theory to prove that any weak non-singular $($singular$)$ solution to \eqref{oursystem} is also a classical non-singular $($singular$)$ solution.
	Since we are working on unbounded domains, it is not direct, though, to verify that classical solutions to \eqref{oursystem} are also weak. 
	In general, it is true that, by the Green identity, classical solutions $\mathcal{U}\in C^{4,\zeta}(\Omega,\mathbb{R}^p)$ also satisfy \eqref{weakformulattion}.  Nevertheless, to show that $\mathcal{U}\in\mathcal{D}^{2,2}(\Omega,\mathbb{R}^p)$ is an entire solution to \eqref{oursystem}, one needs to prove some suitable decay at both the origin and infinity.
\end{remark}

\subsection{Kelvin transform}\label{subsec:kelvintransfform}
We define some type of transform suitable to explore the symmetries of \eqref{oursystem}, which is called the {\it fourth order Kelvin transform} of a $p$-map. 
The Kelvin transform is a device to extend the concept of harmonic (superharmonic or subharmonic) functions by allowing the definition of a function which is harmonic (superharmonic or subharmonic) at infinity. 
This map is a key ingredient for developing a sliding method, namely the moving spheres or the moving planes techniques. 

For $\Omega=\mathbb{R}^n$ or $\Omega=\mathbb{R}^n\setminus\{0\}$, we define the Kelvin transform. 
To this end, for given $x_0\in\mathbb{R}^n$ and $\mu>0$, we need to establish the concept of {\it inversion about a sphere} $\partial B_{\mu}(x_0)$, which is a map $\mathcal{I}_{x_0,\mu}:\Omega\rightarrow\Omega_{x_0,\mu}$ given by $\mathcal{I}_{x_0,\mu}(x)=x_0+K_{x_0,\mu}(x)^2(x-x_0)$, where $K_{x_0,\mu}(x)=\mu/|x-x_0|$ and $\Omega_{x_0,\mu}:=\mathcal{I}_{x_0,\mu}(\Omega)$ is the domain of the Kelvin transform.
In particular, when $x_0=0$ and $\mu=1$, we denote it simply by $\mathcal{I}_{0,1}(x)=x^{*}$ and $K_{0,1}(x)=x|x|^{-2}$.

\begin{definition}
	For any $\mathcal{U}:\Omega\rightarrow\mathbb{R}^p$, let us consider the fourth order Kelvin transform about the sphere with center at $x_0\in\mathbb{R}^n$ and radius $\mu>0$ defined on $\mathcal{U}_{x_0,\mu}:\Omega_{x_0,\mu}\rightarrow\mathbb{R}^p$ by 
	\begin{equation*}
		\mathcal{U}_{x_0,\mu}(x)=K_{x_0,\mu}(x)^{n-4}\mathcal{U}\left(\mathcal{I}_{x_0,\mu}(x)\right).
	\end{equation*}
	In particular, when $p=1$ we set the notation $u_{x_0,\mu}$.
\end{definition}

Now we need to understand how \eqref{oursystem} behaves under the Kelvin transform's action.

\begin{proposition}\label{confinva}
	System \eqref{oursystem} is conformally invariant, in the sense that it is invariant under the action of Kelvin transform, {\it i.e.}, if $\mathcal{U}$ is a non-singular solution to \eqref{oursystem}, then  $\mathcal{U}_{x_0,\mu}$ is a solution to 
	\begin{equation}\label{revised}
		\Delta^{2}(u_{i})_{x_0,\mu}=c(n)|\mathcal{U}_{x_0,\mu}|^{2^{**}-2}(u_{i})_{x_0,\mu} \quad {\rm in} \quad \mathbb{R}^{n}\setminus\{x_0\},
	\end{equation}
	where $\mathcal{U}_{x_0,\mu}=((u_1)_{x_0,\mu},\dots,(u_p)_{x_0,\mu})$.
\end{proposition}

\begin{proof}
	For all $x\in\mathbb{R}^n\setminus\{x_0\}$, let us recall the formulas below
	\begin{equation}\label{laplaciankelvin}
		\Delta u_{x_0,\mu}(x)=K_{x_0,\mu}(x)^{n+2}\Delta u\left(\mathcal{I}_{x_0,\mu}(x)\right)=K_{x_0,\mu}(x)^{4}(\Delta u)_{x_0,\mu}(x)
	\end{equation}
	and
	\begin{equation}\label{biharmonickelvin}
		\Delta^2 u_{x_0,\mu}(x)=K_{x_0,\mu}(x)^{n+4}\Delta^{2}u\left(\mathcal{I}_{x_0,\mu}(x)\right)=K_{x_0,\mu}(x)^{8}(\Delta^2 u)_{x_0,\mu}(x).
	\end{equation}
	Next, expanding the right-hand side of \eqref{revised}, we observe
	\begin{equation}\label{honduras}
		|\mathcal{U}_{x_0,\mu}(x)|^{2^{**}-2}(u_{i})_{x_0,\mu}=K_{x_0,\mu}(x)^{n+4}|\mathcal{U}(x)|^{2^{**}-2}u_i(x).
	\end{equation}
	Therefore, the proof of the proposition follows by a combination of \eqref{biharmonickelvin} and \eqref{honduras}.
\end{proof}

\begin{remark}
	Proposition~\ref{confinva} is not a surprising conclusion since the Gross--Pitaevskii-type nonlinearity preserves the same {\it conformal invariance} enjoyed by the scalar case. Namely, in the case $p=1$, \eqref{scalarsystem} is invariant under the conformal euclidean group's action.
\end{remark}

\subsection{Cylindrical transformation}\label{sec:cylindricaltransform}
Let us introduce the so-called {\it cylindrical transformation} \cite{MR1794994}. 
Using this device, we convert singular solutions to \eqref{oursystem} in the punctured space into non-singular solutions in a cylinder. In fact, considering {\it spherical coordinates} denoted by $(r,\sigma)$, we can rewrite \eqref{oursystem} as the nonautonomous nonlinear system,
\begin{equation*}
	\Delta^2_{\rm sph}u_i=c(n)|\mathcal{U}|^{2^{**}-2}u_i \quad {\rm in} \quad {\mathcal{C}}_0.
\end{equation*}
Here ${\mathcal{C}}_0:=(0,\infty)\times\mathbb{S}^{n-1}$ is the cylinder and $\Delta^2_{\rm sph}$ is the bi-Laplacian in spherical coordinates given by
\begin{align}\label{Pdespherical}
	\Delta^2_{\rm sph}&=\partial_r^{(4)}+ \frac{2(n-1)}{r}\partial_r^{(3)}+\frac{(n-1)(n-3)}{r^2}\partial_r^{(2)}-\frac{(n-1)(n-3)}{r^3}\partial_r&\\\nonumber
	&+\frac{1}{r^4}\Delta_{\sigma}^2+\frac{2}{r^2}\partial^{(2)}_r\Delta_{\sigma}+\frac{2(n-3)}{r^3}\partial_r\Delta_{\sigma}-\frac{2(n-4)}{r^4}\Delta_{\sigma},&
\end{align}
where $\Delta_{\sigma}$ denotes the Laplace--Beltrami operator in $\mathbb{S}^{n-1}$.
Moreover, the vectorial Emden--Fowler change of variables (or logarithm coordinates) given by $\mathcal{V}(t,\theta)=r^{\gamma}\mathcal{U}(r,\sigma)$, where $r=|x|$, $t=-\ln r$, $\sigma=\theta=x/|x|$, and $\gamma=({n-4})/{2}$ is the {\it Fowler rescaling exponent}, sends the problem to the entire cylinder $\mathcal{C}_{\infty}=\mathbb{R}\times \mathbb{S}^{n-1}$. 

In the geometric setting, this change of variables corresponds to the conformal diffeomorphism between the cylinder $\mathcal{C}_{\infty}$ and the punctured space $\varphi:(\mathcal{C}_{\infty},g_{{\rm cyl}})\rightarrow(\mathbb{R}^n\setminus\{0\},\delta_0)$ defined by $\varphi(t,\sigma)=e^{-t}\sigma$. Here $g_{{\rm cyl}}=\ud t^2+\ud\sigma^2$ stands for the cylindrical metric with ${\ud\theta}=e^{-2t}(\ud t^2+\ud\sigma^2)$ its volume element obtained via the pullback $\varphi^{*}\delta_0$, where $\delta_0$ is the standard flat metric. Using this coordinate system, and performing a lengthy computation, we arrive at the following fourth order nonlinear PDE on the cylinder,
\begin{equation}\label{sphevectfowlersystem}
	\Delta^2_{\rm cyl}v_i=c(n)|\mathcal{V}|^{2^{**}-2}v_i \quad {\rm on} \quad {\mathcal{C}_{\infty}}.
\end{equation}
Here $\mathcal{V}=(v_1,\dots,v_p)$ and $\Delta^2_{\rm cyl}$ is the bi-Laplacian in cylindrical coordinates given by
\begin{equation*}
	\Delta^2_{\rm cyl}=\partial_t^{(4)}-K_2\partial_t^{(2)}+K_0+\Delta_{\theta}^{2}+2\partial_t^{(2)}\Delta_{\theta}-J_0\Delta_{\theta},
\end{equation*}
where $K_0,K_2,J_0$ are constants depending only in the dimension defined by
\begin{equation}\label{coefficients}
	K_0=\frac{n^2(n-4)^2}{16}, \quad K_2=\frac{n^2-4n+8}{2} \quad {\rm and} \quad J_0=\frac{n(n-4)}{4}.
\end{equation}
Along this lines let us consider the cylindrical transformation of a $p$-map as follows
\begin{equation*}
	\mathfrak{F}:C_c^{\infty}(\mathbb{R}^n\setminus\{0\},\mathbb{R}^p)\rightarrow C_c^{\infty}(\mathcal{C}_{\infty},\mathbb{R}^p)\quad \mbox{given by} \quad
	\mathfrak{F}(\mathcal{U})=r^{\gamma}\mathcal{U}(r,\sigma).
\end{equation*}

\begin{remark}
	The transformation $\mathfrak{F}$ is a continuous bijection with respect to the Sobolev norms $\|\cdot\|_{\mathcal{D}^{2,2}(\mathbb{R}^n\setminus\{0\},\mathbb{R}^p)} \ {\rm and} \ \|\cdot\|_{H^{2}(\mathcal{C}_{\infty},\mathbb{R}^p)}$, respectively. Furthermore, this transformation sends singular solutions to \eqref{oursystem} into solutions to \eqref{sphevectfowlersystem}, and, by density, $\mathfrak{F}:\mathcal{D}^{2,2}(\mathbb{R}^n\setminus\{0\},\mathbb{R}^p)\rightarrow H^{2}(\mathcal{C}_{\infty},\mathbb{R}^p)$.
\end{remark}

\begin{remark}
	Our choice for the symbol $\Delta^2_{\rm cyl}=\Delta^2_{\rm sph}\circ\mathfrak{F}^{-1}$ is an abuse of notation, since the cylindrical background metric is not flat, we should have $P_{\rm cyl}=\Delta_{\rm sph}^2\circ\mathfrak{F}^{-1}$ $($resp. $\widetilde{P}_{\rm cyl}=\Delta_{\rm sph}^2\circ\widetilde{\mathfrak{F}}^{-1}$$)$, where $P_{\rm cyl}$ stands for the Paneitz--Branson operator of this metric in the new logarithmic cylindrical coordinate system.
\end{remark}

\section{Liouville-type theorem for non-singular solutions}\label{section3}
This section is devoted to present the proof of Theorem~\ref{theorem1}. 
Using the regularity lifting theorem based on \cite{MR1338474}, we aim to obtain regularity results for solutions to \eqref{oursystem} with a removable singularity at the origin. 
Hence, employing an iteration argument from \cite{MR1769247}, we show that non-singular solutions to \eqref{oursystem} are weakly positive. 
Then, we perform a moving spheres technique from O. Druet et al. \cite{MR2558186} and Y. Li and L. Zhang \cite{MR2001065} to obtain that non-singular solutions to \eqref{oursystem} are rotationally invariant about some point. 
This argument provides as a by-product an estimate for the Sobolev norm of solutions to \eqref{oursystem}, yielding that classical solutions to \eqref{oursystem} are also weak (see Remark~\ref{classical/weak}). 
Adopting a variational technique from O. Druet and E. Hebey \cite{MR2603801}, we prove that the Liouville-type result holds for weak solutions to \eqref{oursystem}. 
Finally, as an application of our main result, we show that non-singular solutions to \eqref{oursystem} are indeed extremal maps for the Sobolev embedding of the space $\mathcal{D}^{2,2}(\mathbb{R}^n,\mathbb{R}^p)$ into $L^{2^{**}}(\mathbb{R}^n,\mathbb{R}^p)$.  

Since the origin is a removable singularity, System \eqref{oursystem} can be modeled in the entire space, in the sense that solutions can be smoothly extended to be defined in $\mathbb{R}^n$. In this situation, \eqref{oursystem} is reduced to 
\begin{equation}\label{regularsystem}
\Delta^{2}u_i=c(n)|\mathcal{U}|^{2^{**}-2}u_i \quad {\rm in} \quad \mathbb{R}^n.
\end{equation}
Subsequently, the idea is to provide some properties for solutions to \eqref{regularsystem} by writing this system as a nonlinear fourth order Schr\"{o}dinger equation
with potential $V: \mathbb{R}^n\rightarrow\mathbb{R}$ defined by
\begin{equation*}
V(x)=c(n)|\mathcal{U}(x)|^{2^{**}-2}.
\end{equation*} 

\subsection{Regularity}
We prove that weak solutions to \eqref{oursystem} are as well as classical solutions. 
We should mention that De Giorgi--Nash--Moser bootstrap techniques combined with the Br\'{e}zis-Kato method are standard strategies to produce regularity results for second order elliptic PDEs involving critical growth. 
Unfortunately, this tool does not work in our critical fourth order setting. 
More precisely, the nonlinearity on the right-hand side of \eqref{oursystem} has critical growth, so $|\mathcal{U}|^{2^{**}-2}u_i\in L^{{2n}/{(n+4)}}(\mathbb{R}^n)$. 
Notice that we cannot conclude, using the Sobolev embedding theorem, that $|\mathcal{U}|^{2^{**}-2}u_i$ belongs to $ L^{q}(\mathbb{R}^n)$ for some $q>{2n}/{(n+4)}$ and any $i\in I$. 
We can overcome this lack of integrability by applying the lifting method due to W. Chen and C. Li \cite{MR1338474}.

\begin{propositionletter}\label{reglift}
	Let $Z$ be a Hausdorff topological space, $\|\cdot\|_{X},\|\cdot\|_{Y}:Z\rightarrow[0,\infty]$ be extended norms in $Z$ and $X,Y$ be subspaces defined by $X=\{z \in Z: \|z\|_{X}<\infty\}$ and $Y=\{z \in Z: \|z\|_{Y}<\infty\}$. 
	Suppose that $T$ is a contraction map from $X$ into itself and from $Y$ into itself, and that for $u\in X$, there exists $\widetilde{u}\in X\cap Y$ such that $u=Tu+\widetilde{u}$. 
	Then, $u\in X\cap Y$.
\end{propositionletter}

In the next step, we apply Proposition~\ref{reglift} to show that it is possible to improve the Lebesgue class in which solutions to \eqref{regularsystem} lie. 
Here our strategy is to prove that they indeed belong to the Lebesgue space $L^{s}(\mathbb{R}^n,\mathbb{R}^p)$ for any $s>2^{**}$. 

\begin{proposition}\label{prop:bootstrap}
	Let $\mathcal{U}\in\mathcal{D}^{2,2}(\mathbb{R}^n,\mathbb{R}^p)$ be a nonnegative weak non-singular solution to \eqref{regularsystem}. 
	Then, $\mathcal{U}\in L^{s}(\mathbb{R}^n,\mathbb{R}^p)$ for all $s>2^{**}$.
\end{proposition}

\begin{proof}
	Let us consider the spaces $Z=C^{\infty}_c(\mathbb{R}^n)$, $X=L^{{2n}/{(n-4)}}(\mathbb{R}^n)$ and $Y=L^{q}(\mathbb{R}^n)$ for $q>2n/(n-4)$. Let $\Gamma_2(x,y)=C(n)|x-y|^{4-n}$ be the fundamental solution to $\Delta^2$ in $\mathbb{R}^n$, where $C(n)=[(n-4)(n-2)\omega_{n-1}]^{-1}$. Thus, it is well-defined the inverse operator
	\begin{equation*}
		(Tu)(x)=\int_{\mathbb{R}^n}\Gamma_2(x,y)u(y)\ud y.
	\end{equation*}
	Hence, using the Hardy--Littlewood--Sobolev inequality (see \cite{MR717827}), we get that for any $q\in(1,n/4)$, there exists $C>0$ such that
	\begin{equation*}
		\|Tu\|_{L^{\frac{nq}{n-4q}}(\mathbb{R}^n)}=\|\Gamma_2\ast u\|_{L^{\frac{nq}{n-4q}}(\mathbb{R}^n)}\leqslant C\|u\|_{L^{q}(\mathbb{R}^n)}.
	\end{equation*}
	For $M>0$, let us define  $\widetilde{V}_M(x)=V(x)-V_M(x)$, where
	\begin{equation*}
		V_M(x)=
		\begin{cases}
			V(x), \ {\rm if} \ |V(x)|\geqslant M,\\
			0,\ {\rm otherwise}.
		\end{cases}
	\end{equation*}
	Applying the integral operator $T_Mu:=\Gamma_2\ast V_Mu$ on  \eqref{regularsystem}, we obtain that $u_i=T_Mu_i+\widetilde{T}_Mu_i$, where 
	\begin{equation*}
		(T_M u_i)(x)=\int_{\mathbb{R}^n}\Gamma_2(x,y)V_{M}(y)u_i(y)\ud y \quad \text{and} \quad \widetilde{T}_Mu_i(x)=\int_{\mathbb{R}^n}\Gamma_2(x,y)\widetilde{V}_M(y)u_i(y)\ud y.
	\end{equation*}
	
	\noindent{\bf Claim 1:} For $n/(n-4)<q<\infty$, there exists $M\gg1$ large such that $T_M:L^{q}(\mathbb{R}^n)\rightarrow L^{q}(\mathbb{R}^n)$ is a contraction.
	
	\noindent In fact, for any $q\in(n/(n-4),\infty)$, there exists $m\in (1,n/4)$ such that $q=nm/(n-4m)$. Then, by the H\"{o}lder inequality, for any $u\in L^{q}(\mathbb{R}^n)$, we get that there exists $C>0$ satisfying
	\begin{equation*}
		\|T_Mu\|_{L^{q}(\mathbb{R}^n)}\leqslant\|\Gamma_2\ast V_{M}u\|_{L^{q}(\mathbb{R}^n)}\leqslant C\|V_M\|_{L^{{n}/{4}}(\mathbb{R}^n)}\|u\|_{L^{q}(\mathbb{R}^n)}.
	\end{equation*}
	Since $V_{M}\in L^{n/4}(\mathbb{R}^n)$ it is possible to choose $M\gg1$ such that $\|V_M\|_{L^{{n}/{4}}(\mathbb{R}^n)}<{1}/{2C}$. Therefore, we arrive at $\|T_Mu\|_{L^{q}(\mathbb{R}^n)}\leqslant{1}/{2}\|u\|_{L^{q}(\mathbb{R}^n)}$, which yields $T_M$ is a contraction.
	
	\noindent{\bf Claim 2:} For any $n/(n-4)<q<\infty$, it follows that $\widetilde{T}_Mu_i\in L^{q}(\mathbb{R}^n)$.
	
	\noindent Indeed, for any $n/(n-4)<q<\infty$, choose $1<m<n/4$, satisfying $q=nm/(n-4m)$. Since $\widetilde{V}_M$ is bounded, we obtain
	\begin{equation*}
		\|\widetilde{T}_Mu_i\|_{L^{q}(\mathbb{R}^n)}=\|\Gamma_2\ast\widetilde{V}_Mu_i\|_{L^{q}(\mathbb{R}^n)}\leqslant C_1\|\widetilde{V}_Mu_i\|_{L^{m}(\mathbb{R}^n)}\leqslant C_2\|u_i\|_{L^{m}(\mathbb{R}^n)}.
	\end{equation*}
	However, using the Sobolev embedding theorem, we have that $u_i\in L^{m}(\mathbb{R}^n)$ when $m=2n/(n-4)$, which implies $q=2n/(n-8)$. Thus, we find that $u_i\in L^{q}(\mathbb{R}^n)$ when
	\begin{equation*}
		\begin{cases}
			1<q<\infty,& {\rm if} \ 5\leqslant n\leqslant 8\\
			1<q\leqslant\frac{2n}{n-8},& {\rm if} \ n\geqslant9.
		\end{cases}
	\end{equation*}
	Now we can repeat the argument for $m=2n/(n-8)$ to obtain that $u_i\in L^{q}(\mathbb{R}^n)$ for
	\begin{equation*}
		\begin{cases}
			1<q<\infty,& {\rm if} \ 5\leqslant n\leqslant 12\\
			1<q\leqslant\frac{2n}{n-12},& {\rm if} \ n\geqslant13.
		\end{cases}
	\end{equation*}
	Therefore proceeding inductively as in the last argument, the proof of the claim follows.
	
	Combining Claims 1 and 2, we can apply Proposition~\ref{reglift} to show that $u_i \in L^{q}(\mathbb{R}^n)$ for all $q>2^{**}$ and $i\in I$. In particular, the proof of the proposition is concluded.
\end{proof}
	
\begin{corollary}\label{regregularity}
	Let $\mathcal{U}\in\mathcal{D}^{2,2}(\mathbb{R}^n,\mathbb{R}^p)$ be a nonnegative weak non-singular solution to \eqref{oursystem}. 
	Then, $\mathcal{U}\in C^{4,\zeta}(\mathbb{R}^n,\mathbb{R}^p)$ is a classical non-singular solution to \eqref{oursystem}.
\end{corollary}

\begin{proof}
	Using the Proposition~\ref{prop:bootstrap}, we can apply Morrey embedding theorem to get $u_i\in C^{0,\zeta}(\mathbb{R}^n)$ for some $\zeta\in(0,1)$. Finally using Schauder estimates, one concludes $u_i\in C^{4,\zeta}(\mathbb{R}^n)$, which provides that $\mathcal{U}\in C^{4,\zeta}(\mathbb{R}^n,\mathbb{R}^p)$.
\end{proof}

\begin{remark} In \cite[Proposition~3.1]{MR1809291} using a different approach, K. Uhlenbeck and J. Viaclovski proved regularity for solutions to a class of general geometric fourth order PDEs, which could also be used to prove that for some $q>2^{**}$ it holds $\mathcal{U}\in L^{q}(\mathbb{R}^n,\mathbb{R}^p)$. 
\end{remark}

\subsection{Superharmonicity}
We aim to obtain a strong maximum principle for nonnegative solutions to \eqref{oursystem}. In this direction, we prove that any component solution to \eqref{oursystem} is superharmonic. We are inspired in \cite[Theorem~2.1]{MR1769247}. The main difference in our approach is the appearance of the strong coupling term on the right-hand side of \eqref{regularsystem}. This coupled nonlinearity could imply the failure of the method for some components. However, we can overcome this issue thanks to an inequality involving the norm of the $p$-map solution. Before proving the superharmonicity result, we need to establish two technical lemmas, which proofs are merely calculus argument and can be found in \cite[Lemma~2.2 and Lemma~2.3]{MR1769247}, respectively.

\begin{lemmaletter}\label{lemmaxu1}
	Suppose that $l_0=2$ and $\{l_k\}_{k\in\mathbb{N}}$ given by the formula $l_{k+1}=sl_k+4$ for some $s>1$. Then, for all $k\in\mathbb{N}$,
	
	\noindent {{\rm (i)} Recursion formula:} $l_{k+1}=\frac{2s^{k+2}+2s^{k+1}-4}{s-1}$;\\
	\noindent {{\rm (ii)} Upper estimate:} $(n+sl_k)(2+sl_k)(n+2+sl_k)(4+sl_k)\leqslant(n+2+2s)^{4(s+1)}$.
\end{lemmaletter}

\begin{lemmaletter}\label{lemmaxu2}
	Suppose that $b_0=0$ and define $\{b_k\}_{k\in\mathbb{N}}$ by $b_{k+1}=sb_k+4(k+1)$. Then, for all $k\in\mathbb{N}$, 
	\begin{equation*}
	b_{k+1}=4\left[\frac{s^{k+2}-(k+2)s+k+1}{s^2}\right].
	\end{equation*}
\end{lemmaletter}

The superharmonicity result can be stated as follows.

\begin{proposition}\label{superharmonicity}
	Let $\mathcal{U}$ be a nonnegative non-singular solution to \eqref{oursystem}. Then, $-\Delta u_i\geqslant0$ in $\mathbb{R}^n$ for all $i\in I$.
\end{proposition}

\begin{proof} 
	Supposing by contradiction that the proposition does not hold, there exists $i\in I$ and $x_0\in\mathbb{R}^n$ satisfying $-\Delta u_{i}(x_0)<0$. Since the Laplacian is invariant under translations, we may suppose without loss of generality that $x_0=0$. Let us reformulate \eqref{oursystem} as the following system in the whole space
	\begin{align}\label{lane-emden-system}
	\begin{cases}
	-\Delta {u}_i&={h}_i\\
	-\Delta {h}_i&=c(n)|{\mathcal{U}}|^{2^{**}-2}{u}_i.
	\end{cases}
	\end{align}
	
		Let $B_r\subseteq\mathbb{R}^n$ be the ball of radius $r>0$, and $\omega_{n-1}$ be 
		the $(n-1)$-dimensional surface measure of the unit sphere, we consider 
		\begin{equation*}
		\overline{u}_i=\frac{1}{n\omega_{n-1}r^{n-1}}\displaystyle\int_{\partial{B_r}}u_i\ud \sigma_r \quad {\rm and} \quad \overline{h}_i=\frac{1}{n\omega_{n-1}r^{n-1}}\displaystyle\int_{\partial{B_r}}h_i\ud \sigma_r, 
		\end{equation*}
		the spherical averages of $u_i$ and $h_i$, respectively. Now taking the spherical average on the first line of \eqref{lane-emden-system}, and using that $\overline{\Delta u_i}=\Delta\overline{u}_i$, implies
		\begin{equation}\label{portugal}
		\Delta\overline{u}_i+\overline{h}_i=0.
		\end{equation}
		Furthermore, we rewrite the second equality of \eqref{lane-emden-system} to get $\Delta h_i+c(n)|\mathcal{U}|^{2^{**}-2}u_i=0$, from which, by taking again the spherical average in both sides, provides
		\begin{equation*}
		0=\frac{1}{n\omega_{n-1}r^{n-1}}\displaystyle\int_{\partial{B_r}}\left(\Delta h_i+c(n)|\mathcal{U}|^{2^{**}-2}u_i\right)\ud\sigma_r
		=\Delta\overline{h}_i+\frac{c(n)}{n\omega_{n-1}r^{n-1}}\displaystyle\int_{\partial{B_r}}|\mathcal{U}|^{{2^{**}-2}}{u}_i \ud \sigma_r.
		\end{equation*}
		Hence,
		\begin{equation}\label{greece}
		\Delta\overline{h}_i=-\frac{c(n)}{n\omega_{n-1}r^{n-1}}\displaystyle\int_{\partial{B_r}}|\mathcal{U}(x)|^{{2^{**}-2}}{u}_i(x) \ud \sigma_r,
		\end{equation}
		which, by using that $0\leqslant u_i(x)\leqslant|\mathcal{U}(x)|$ for any $x\in\mathbb{R}^n$, implies
		\begin{align}\label{revision}
		-\frac{c(n)}{n\omega_{n-1}r^{n-1}}\displaystyle\int_{\partial{B_r}}|\mathcal{U}(x)|^{{2^{**}-2}}{u}_i(x) \ud \sigma_r
		&\leqslant-\frac{c(n)}{n\omega_{n-1}r^{n-1}}\displaystyle\int_{\partial{B_r}}|u_i(x)|^{{2^{**}-1}}\ud \sigma_r&\\\nonumber
		&\leqslant-c(n)\left(\frac{1}{n\omega_{n-1}r^{n-1}}\displaystyle\int_{\partial{B_r}}|u_i(x)|\ud \sigma_r\right)^{{2^{**}-1}}&\\\nonumber
		&=-c(n)\overline{u}_i^{2^{**}-1},&
		\end{align} 
		where on the second inequality, we used the Jensen inequality for the convex function $t\mapsto t^{2^{**}-1}$. Finally, combining \eqref{greece} and \eqref{revision}, we get
		\begin{equation}\label{correctinequality}
		\Delta \overline{h}_i+c(n)\overline{u}_i^{2^{**}-1}\leqslant0.
		\end{equation}
	
	By the definition of spherical average, we have that $\overline{h}_i(0)=h_i(0)<0$. In addition, by \eqref{correctinequality}, we find
	\begin{equation}\label{spain}
	\Delta \overline{h}_i\leqslant0.
	\end{equation}
	Then, multiplying equation \eqref{spain} by $r^{n-1}$, and integrating, we arrive at
	\begin{equation*}
	r^{n-1}\frac{\ud}{\ud r}\overline{h}_i\leqslant0.
	\end{equation*}
	It clearly implies that $\overline{h}_i$ is monotonically decreasing for all $r>0$, we obtain
	\begin{equation}\label{colombia}
	\overline{h}_i(r)\leqslant h_i(0).
	\end{equation}
	Substituting \eqref{colombia} into \eqref{portugal}, and integrating, it follows
	\begin{equation}\label{belgium}
	\overline{u}_i(r)\geqslant-\frac{h_i(0)}{2n}r^2.
	\end{equation}
	Putting \eqref{belgium} in \eqref{correctinequality}, multiplying both side of inequality by $r^{n-1}$, and integrating, we have
	\begin{equation}\label{poland}
	\overline{h}_i(r)\leqslant-\frac{c_0^{s^2}r^{2s+2}}{(n+2+2s)(2s+4)},
	\end{equation}
	where $c_0=-{h_i(0)}/{2n}>0$ and $s=(n+4)/(n-4)$. Then, combining \eqref{poland} with \eqref{portugal}, and repeating the same procedure, it provides
	\begin{equation}\label{ukraine}
	\overline{u}_i(r)\geqslant\frac{c_0^{s^2}r^{2s+4}}{(n+2s)(s+2)(n+2+2s)(2s+4)}.
	\end{equation}
	Based on \eqref{ukraine} and thanks to Lemma~\ref{lemmaxu2}, we may assume that for some $k\in\mathbb{Z}$ and $l_k,b_k \in \mathbb{R}$, it holds
	\begin{equation}\label{netherlands}
	\overline{u}_i(r)\geqslant\frac{c_0^{s^k}r^{l_k}}{(n+2+2s)^{b_k}}.
	\end{equation}
	Again, we can use estimate \eqref{netherlands} combined with \eqref{portugal} and \eqref{correctinequality} to obtain
	\begin{equation*}
	\overline{h}_i(r)\leqslant-\frac{c_0^{s^{k+1}}r^{sl_k+2}}{(n+2+2s)^{sb_k}(n+sl_k)(sl_k+2)}
	\end{equation*}
	and
	\begin{equation}\label{israel}
	\overline{u}_i(r)\geqslant\frac{c_0^{s^{k+1}}r^{sl_k+4}}{(n+2+2s)^{pb_k}(n+sl_k)(sl_k+2)(n+2+sl_k)(sl_k+4)}.
	\end{equation}
	Setting $l_{k+1}=sl_k+4$, we have by (ii) of Lemma~\ref{lemmaxu1} that \eqref{israel} remains true for $k+1$ with $b_{k+1}=sb_k+4(k+1)$. In other words, it follows
	\begin{equation*}
	\overline{u}_i(r)\geqslant\frac{c_{0}^{s^{k+1}}r^{l_k+1}}{(n+2s+2)^{b_{k+1}}}.
	\end{equation*}
	Assuming that $c_0\geqslant1$, we can choose $r_0=(n+2s+2)^{4/(s-1)}$ and, by Lemmas~\ref{lemmaxu1} and \ref{lemmaxu2}, the following estimates holds
	\begin{equation}\label{wales}
	\overline{u}_i(r_0)\geqslant c_{0}^{s^{k+1}}\left[(n+2s+2)^\frac{4}{(s-1)^2}\right]^{s^{k+2}+2s^{k+1}+(k+2)s-k-5}.
	\end{equation}
	Taking the limit as $k\rightarrow\infty$ in \eqref{wales}, we find a contradiction since the right-hand side blows-up. Therefore, $\Delta u_i\leqslant0$ for all $i\in I$. When $c_0<1$, choosing $r_0=c_0^{-1}(n+2s+2)^{4/(s-1)}$ the same argument can be applied. 
\end{proof}

As a consequence of the last result, we can prove that solutions to \eqref{oursystem} are weakly positive. 

\begin{corollary}\label{positivity}
	Let $\mathcal{U}$ be a non-singular solution to \eqref{oursystem}. Then, for any $i\in I$ we have that either $u_i\equiv 0$ or $u_i>0$. In other terms, $I=I_0\cup I_+$ is a disjoint union.
\end{corollary}

\subsection{Lower bound estimates}
The last subsection's main result asserts that component solutions to \eqref{regularsystem} are superharmonic is useful to provide essential properties required to start the moving spheres method. 
More precisely, we obtain a lower bound estimate for any component solution. 
The idea is to use Proposition~\ref{superharmonicity} and the three spheres theorem for the bi-Laplacian.
Now we prove the Three-spheres type result.
To fix some notation, let us define $m(r):=\min_{\partial B_r}u$. 

\begin{lemma}\label{3S}
	Let $\Omega\subset\mathbb{R}^{n}$ be a region containing two concentric spheres of radii $r_1$ and $r_2$ and the region between them, and $u:\Omega\to\mathbb{R}$ be a superharmonic smooth function in $\Omega$. Assume that $u$ is radially symmetric and monotonically nonincreasing. Then, for every $r>0$ such that $0<r_1<r<r_2$, it follows 
	\begin{equation*}
		m(r)\geqslant\frac{m(r_1) \left(r_{2}^{4-n}-r^{4-n}\right)+m(r_2) \left(r^{4-n}-r_{1}^{4-n}\right)}{r_{2}^{4-n}-r_{1}^{4-n}}.
	\end{equation*}
	Moreover, equality occurs if, and only if, one has $u(|x|)=A+B|x|^{4-n}$, for some $A,B\in\mathbb{R}$.
\end{lemma}

\begin{proof}
	Furthermore, suppose that for some $A,B\in \mathbb{R}$, we have that $\varrho(r)=A+Br^{4-n}$. 
	Now choosing $A,B\in\mathbb{R}$ satisfying $\varrho(r_1)=m(r_1)$ and $\varrho(r_2)=m(r_2)$, we find
	\begin{equation*}
		\varrho(r)=\frac{m(r_1)\left(r_{2}^{4-n}-r^{4-n}\right)+m(r_2)\left(r^{4-n}-r_{1}^{4-n}\right)}{r_{2}^{4-n}-r_{1}^{4-n}}.
	\end{equation*}
	Defining $\psi(x)=u(x)-\varrho(|x|)$, we have that $\Delta u\leqslant0$. Moreover, since $u$ is radially symmetric and monotonically nonincreasing, we get $m(r_2)-m(r_1)<0$, which implies $-\Delta\varrho\leqslant0$. 
	Hence, it holds $\Delta \psi\leqslant0$  in $B_1\setminus B_2$ and $\psi\geqslant0$ on $\partial(B_1\setminus B_2)$, and the strong minimum principle yields $\psi\geqslant0$, equivalently $u(x)\geqslant\varrho(|x|)$ for all $r_1\leqslant |x|\leqslant r_2$. Therefore, $m(r)\geqslant\varrho(r)$, which proves the inequality. 
	Finally, the classification for the case of equality is straightforward to verify.
\end{proof}

\begin{remark}
	It can be trivially seen by the classification in Theorem~\ref{theorem1} that components are radially symmetric and monotonically nonincreasing, there is a way of proving this property without appealing to this full classification, which can be done by mimicking the moving planes technique in \cite[Theorem~1.3]{MR1611691} for the vectorial setting.
\end{remark}

\begin{corollary}\label{lowerestimate}
	Let $\mathcal{U}$ be a nonnegative non-singular solution to \eqref{oursystem}. Then, given $0<r_0<r$, it follows
	\begin{equation*}
		u_i(x)\geqslant\left(\frac{r_0}{|x|}\right)^{n-4}\min_{\partial B_{r_0}}u_i \quad {\rm for \ any} \quad x\in B_r\setminus B_{r_0}.
	\end{equation*}
\end{corollary}

\begin{proof}
	Fix $0<r_0<r$, by applying Lemma~\ref{3S}, we get
	\begin{align*}
		\left(r_{0}^{n-4}-r^{n-4}\right)u_i(x)\geqslant\left(|x|^{n-4}-r^{n-4}\right)\min_{\partial B_{r_0}}u_i,
	\end{align*}
	which, by letting $r\rightarrow\infty$, gives us the desired conclusion.
\end{proof}

\subsection{Moving spheres method}
We apply the moving sphere method to show that nonnegative solutions $\mathcal{U}$ to \eqref{regularsystem} are radially symmetric, that is, $u_i$ is radially symmetric for all $i\in I$.
Furthermore, we provide the classification for each $u_i$, and, in particular, for the norm $|\mathcal{U}|$.  
The moving spheres method is an alternative variant of the moving planes method, which can also be used to obtain radial symmetry or more robust Liouville-type results for solutions to more general PDEs \cite{MR1338474,MR1223899,MR2001065,MR1369398,MR1460076}. 
Our inspiration is a moving spheres argument due to O. Druet, E. Hebey and J. V\'etois \cite[Proposition~3.1]{MR2558186}, which is based on the work of Y. Li and L. Zhang \cite{MR2001065}.

Initially, we need to state two classification results that will be used later and whose proofs can be found in \cite[Lemma~11.1 and Lemma~11.2]{MR2001065}. 
We recall the notation $K_{z,\mu(z)}$ and $\mathcal{I}_{z,\mu(z)}$ from Subsection~\ref{subsec:kelvintransfform}.

\begin{propositionletter}[Weak Liouville-type result]\label{li-zhang1}
	Let $u\in C^{1}(\mathbb{R}^n)$ and $\nu>0$. Suppose that for all $z\in\mathbb{R}^n$, there exists $\mu(z)>0$ such that
	\begin{equation}\label{hyp1}
		K_{z,\mu(z)}(x)^{\nu}\left(u\circ\mathcal{I}_{z,\mu(z)}\right)(x)=u(x) \quad {\rm for \ all} \quad x\in\mathbb{R}^{n}\setminus\{z\}.
	\end{equation}
	Then, for some $\mu\geqslant0$, $\mu'>0$ and $x_0\in\mathbb{R}^{n}$, it follows that $u(x)=\pm\left(\frac{\mu'}{\mu+|x-x_0|^2}\right)^{{\nu}/{2}}$.
\end{propositionletter}

\begin{propositionletter}[Strong Liouville-type result]\label{li-zhang2}
	Let $u\in C^{1}(\mathbb{R}^n)$ and $\nu>0$. 
	Suppose that for all $z\in\mathbb{R}^n$, there exists $\mu(z)>0$ such that
	\begin{equation}\label{hyp2}
		K_{z,\mu(z)}(x)^{\nu}\left(u\circ\mathcal{I}_{z,\mu(z)}\right)(x)\leqslant u(x) \quad {\rm for \ all} \quad x\in\mathbb{R}^{n}\setminus \bar{B}_{{\mu}(z)}.
	\end{equation}
	Then, $u$ is constant. 
\end{propositionletter}

\begin{remark}
	In terms of the Kelvin transform, that is, $\nu=n-4$, notice that conditions \eqref{hyp1} and \eqref{hyp2} can be rewritten respectively as $u_{z,\mu(z)}=u$ in $\mathbb{R}^n\setminus\{z\}$ and $u_{z,\mu(z)}\leqslant u$ in $\mathbb{R}^n\setminus B_{\mu(z)}(z)$.
\end{remark}

In what follows, let us divide the moving spheres process into three parts, namely, Lemmas \ref{lemmaA}, \ref{lemmaB}, and \ref{lemmaC}.
First, we show that it is possible to start the process of moving spheres. For this, it will be crucial to use Corollaries~\ref{positivity} and \ref{lowerestimate}.

\begin{lemma}\label{lemmaA}
	Let $\mathcal{U}$ be a nonnegative solution to \eqref{regularsystem}. Then, for any $x_0\in\mathbb{R}^{n}$, there exists $\mu_{0}(x_0)>0$ satisfying that for $\mu\in(0,\mu_{0}(x_0))$,
	$(u_i)_{x_0,\mu}\leqslant u_i$ in $\mathbb{R}^{n}\setminus B_{\mu}(x_0) $ for all $ i\in I.$
\end{lemma}

\begin{proof}
	By translation invariance, we may take $x_0=0$. Let us denote $(u_i)_{0,\mu}=(u_i)_{\mu}$ for $i\in I$.
	
	\noindent{\bf Claim 1:} For any $i\in I_{+}$, there exists $r_0>0$ such that for $r\in (0,r_0]$ and $\theta\in\mathbb{S}^{n-1}$, we have 
	\begin{equation*}
		{\partial_r}\left(r^{\frac{n-4}{2}}u_i(r\theta)\right)>0.
	\end{equation*}
	
	\noindent In fact, since $u_i$ is a continuously differentiable function for each $i\in I_{+}$, there exists $\widetilde{r}_i>0$ satisfying $\inf_{0<y\leqslant\widetilde{r}_i}u_i>0$ and $\sup_{0<y\leqslant\widetilde{r}_i}|\nabla u_i|<\infty$. Then, choosing 
	\begin{equation*}
		r_i=\min\left\{\widetilde{r}_i,\frac{(n-4)\displaystyle\inf_{0<y\leqslant\widetilde{r}_i}u_i}{2\displaystyle\sup_{0<y\leqslant\widetilde{r}_i}|\nabla u_i|}\right\},
	\end{equation*}
	for $0<r<r_i$, we have
	\begin{align*}
		{\partial_r}\left(r^{\frac{n-4}{2}}u_i(r\theta)\right)\geqslant r^{\frac{n-6}{2}}\left(\frac{n-4}{2}u_i(r\theta)-r\left|\partial_r(r\theta)\right|\right).
	\end{align*}
	By our choice of $r_i>0$, we obtain that ${\partial_r}\left(r^{\frac{n-4}{2}}u_i(r\theta)\right)\geqslant0$, which, by choosing $r_0=\min_{i\in I_+}r_i$, concludes the proof of the claim.
	
	\noindent {\bf Claim 2:} For $\mu\in(0,r_0]$ and $x\in \bar{B}_{r_0}\setminus B_{\mu}$, it follows that $(u_i)_{\mu}\leqslant u_i$ in $B_{r_0}\setminus B_{\mu}$. 
	
	\noindent Indeed, using Claim 1, we observe that $\rho(r)=r^{(n-4)/2}u_i(r\theta)$ is radially increasing in $(0,r_0]$ for any $\theta\in\mathbb{S}^{n-1}$. Hence, taking $r=1$ and $r'=(\mu/|x|)^{2}$, we have $\rho(r')\leqslant\rho(1)$, which completes the proof.
	
	By Claim 2 and Proposition~\ref{superharmonicity}, the hypothesis in Corollary~\ref{lowerestimate} are satisfied. Consequently, for any $r>r_0$ and $i\in I$, we find
	\begin{equation*}
		u_i(x)\geqslant\left(\frac{r_0}{|x|}\right)^{n-4}\min_{\partial B_{r_0}}u_i \quad {\rm in} \quad  B_r\setminus B_{r_0}.
	\end{equation*}
	Setting $\mu_0=r_0\min_{i\in I_{+}}\left(\frac{\min_{\partial{B_{r_0}}}u_i}{\max_{\bar{B}_{r_0}}u_i}\right)^{{4-n}}$, we find
	\begin{equation*}
		(u_i)_{\mu}(x)\leqslant\left(\frac{\mu_0}{|x|}\right)^{n-4}\max_{\bar{B}_{r_0}}u_i\leqslant\left(\frac{r_0}{|x|}\right)^{n-4}\min_{\partial B_{r_0}}u_i,
	\end{equation*}
	for any $\mu\in(0,\mu_0)$, $x\in\mathbb{R}^n\setminus B_{r_0}$ and $i\in I$, which combined with Claim 1 completes the proof.
\end{proof}

After this lemma, let us introduce a well-defined quantity, namely the supremum for which a $p$-map and its Kelvin transform have the same Euclidean norm.

\begin{definition}
	Let $\mathcal{U}$ be a nonnegative solution to \eqref{regularsystem}. For any $x_0\in\mathbb{R}^n$ and $i\in I$, let us define 
	\begin{equation*}
		\mu_i^*(x_0)=\sup\{\mu>0 : (u_i)_{x_0,\mu}\leqslant u_i \quad {\rm in} \quad \mathbb{R}^{n}\setminus B_{\mu}(x_0) \}.
	\end{equation*}
	and
	\begin{equation}\label{coincidence}
		\mu^*(x_0)=\max_{i\in I}\mu_i^*(x_0).
	\end{equation}
\end{definition}

The second lemma states that if \eqref{coincidence} is finite, the moving spheres process must stop, and the euclidean norm of solution to \eqref{regularsystem} are invariant under Kelvin transform.

\begin{lemma}\label{lemmaB}
	Let $\mathcal{U}$ be a nonnegative solution to \eqref{regularsystem}. If $\mu_i^*(x_0)<\infty$, then $(u_i)_{x_0,\mu(x_0)}\equiv u_i$ in $\mathbb{R}^{n}\setminus\{x_0\}$ for all $i\in I$.
	Moreover, $|\mathcal{U}_{x_0,\mu^*(x_0)}|\equiv|\mathcal{U}|$ in $\mathbb{R}^{n}\setminus\{x_0\}$.
\end{lemma}

\begin{proof}
	Without loss of generality, we may take $x_0=0$. We denote $\mu^*(0)=\mu^*$. By the definition of $\mu^*$, when $\mu^*<\infty$, we get that for any $\mu\in(0,\mu^*]$ and $i\in I$, it holds
	\begin{equation}\label{eua}
		(u_i)_{\mu}\leqslant u_i \quad {\rm in} \quad \mathbb{R}^{n}\setminus B_{\mu}(0).
	\end{equation}
	By contradiction suppose that there exist $i_0\in I$ and $(\mu_{k})_{k\in\mathbb{N}}$ in $(\mu^*,\infty)$ satisfying $\mu_{k}\rightarrow\mu^*$ and such that \eqref{eua}  does not hold for $i=i_0$ and $\mu=\mu_{k}$. For $\mu>0$, let us define $\omega_{\mu}=(u_{i_0})-(u_{i_0})_{\mu}$. 
	
	\noindent {\bf Claim 1:} $\omega_{\mu^*}$ is superharmonic. 
	
	\noindent Indeed, as a combination of \eqref{regularsystem} and Lemma~\ref{confinva}, we obtain
	\begin{equation*}
		\begin{cases}
			\Delta^{2}\omega_{\mu^*}(x)=c_{\mu^*}(x)\omega_{\mu^*}& {\rm in} \quad \mathbb{R}^n\setminus B_{\mu^*}(0)\\
			\Delta \omega_{\mu^*}(x)=\omega_{\mu^*}(x)=0& {\rm on} \quad \partial B_{\mu^*}(0),
		\end{cases}
	\end{equation*}
	where
	\begin{equation*}
		c_{\mu^*}=\frac{c(n)|{\mathcal{U}}|^{2^{**}-2}u_{i_0}-c(n)|\mathcal{U}_{\mu^*}|^{2^{**}-2}({u}_{i_0})_{\mu^*}}{{u}_{i_0}-({u}_{i_0})_{\mu^*}}>0 \quad {\rm in} \quad \mathbb{R}^n\setminus B_{\mu^*}(0).
	\end{equation*}
	Therefore, by Claim 1 we can use the strong minimum principle in \cite[Theorem~3]{MR312040} to conclude
	\begin{equation*}
		\min_{\mathbb{R}^{n}\setminus B_{\mu^*}(0)}\omega_{\mu^*}=\min_{\partial B_{\mu}(0)}\omega_{\mu^*}.
	\end{equation*}
	
	\noindent {\bf Claim 2:} $\omega_{\mu^*}\equiv0$.
	
	\noindent Supposing that $\omega_{\mu^*}$ is not equivalently zero in $\mathbb{R}^{n}\setminus B_{\mu}(0)$, by Hopf Lemma \cite[Lemma~3.4]{MR1814364}, we have that ${\partial_\nu}\omega_{\mu^*}>0$ in $\partial B_{\mu^*}(0)$.
	Moreover, by the continuity of $\nabla u_{i_0}$, one can find $r_0>\mu^*$ such that for any $\mu^*\in[\mu,r_0)$, we get
	\begin{equation}\label{spo}
		\omega_{\mu^*}>0 \quad {\rm in} \quad \bar{B}_{r_0}(0)\setminus B_{\mu}(0).
	\end{equation}
	Again, applying Proposition~\ref{3S}, we obtain  
	\begin{equation*}
		\omega_{\mu^*}\geqslant\left(\frac{r_0}{|x|}\right)^{n-4}\omega_{\mu^*}.
	\end{equation*}
	On the other hand, by the uniform continuity of the $u_{i_0}$ on $B_{r_0}(0)$, there exists $\varepsilon>0$ such that for any $\mu\in[\mu^*,\mu^*+\varepsilon)$ and $x\in\mathbb{R}^{n}\setminus B_{r_0}(0)$, it follows
	\begin{equation}\label{san}
		|\omega_{\mu^*}(x)-\omega_{\mu}(x)|=|(u_i)_{\mu}(x)-(u_{i_0})_{\mu^*}(x)|\leqslant\frac{1}{2}\left(\frac{r_0}{|x|}\right)^{n-4}\min_{\partial B_{r_0}(0)}\omega_{\mu^*}.
	\end{equation}
	Therefore, a combination of \eqref{spo} and \eqref{san} yields $\omega_{\mu^*}\geqslant0$ in $\mathbb{R}^{n}\setminus B_{\mu}(0)$ for any $\mu\in[\mu^*,\mu^*+\varepsilon)$. This is a contradiction with the definition of $\mu^*$, thus $\omega_{\mu^*}\equiv0$ in $\mathbb{R}^{n}\setminus B_{\mu}(0)$. 
	
	Moreover, let us define
	\begin{equation*}
		\omega_{\mu}(x)=-\left(\frac{\mu^*}{|x|}\right)^{n-4}\omega_{\mu^*}\left(\left(\frac{\mu^*}{|x|}\right)^2x\right).
	\end{equation*}
	Hence, it follows that $\omega_{\mu^*}\equiv0$ in $\mathbb{R}^{n}\setminus\{0\}$. Since $u_{i_0}$ cannot be identically zero without contradicting the definition of $\mu^*$, by Proposition~\ref{positivity} $u_{i_0}$ is nowhere vanishing. Consequently, we obtain that $|\mathcal{U}_{\mu^*}|\equiv|\mathcal{U}|$ in $\mathbb{R}^{n}\setminus\{0\}$.
\end{proof}

In the last lemma, we show that the moving spheres process only stops if $\mathcal{U}$ is the trivial solution.

\begin{lemma}\label{lemmaC}
	Let $\mathcal{U}$ be a nonnegative solution to \eqref{regularsystem}. If $\mu^*(x_0)<\infty$ for some $x_0\in\mathbb{R}^n$, then $\mathcal{U}\equiv0$.
\end{lemma}

\begin{proof}
	By definition of $\mu^*(x_0)$, if $\mu^*(x_0)=\infty$, we get that for any $\mu>0$ and $i\in I$, $(u_i)_{x_0,\mu}\leqslant u_i$ in $\mathbb{R}^n\setminus B_{\mu}(x_0)$. Moreover, assuming that $x_0=0$, by \eqref{eua}, we have 
	\begin{equation*}
		\mu^{n-4}\leqslant \liminf_{|x|\rightarrow\infty}|x|^{n-4}u_i(x),
	\end{equation*}
	which by passing to limit as $\mu\rightarrow\infty$ provides that for $i\in I$, either $u_i(0)=0$ or $|x|^{n-4}u_i(x)\rightarrow0$ as $|x|\rightarrow\infty$. Using that $u_i(0)=0$ for all $i\in I$, by Propositions~\ref{superharmonicity} and \ref{positivity}, we conclude that $u_i\equiv0$. Therefore, we may assume $|x|^{n-4}u_i(x)\rightarrow\infty$ as $|x|\rightarrow\infty$ for all $i\in I_{+}$.
	
	\noindent{\bf Claim 1:} $\mu^*(z)=\infty$ for all $z\in\mathbb{R}^n$.
	
	\noindent Indeed, when $\mu^*(z)<\infty$ for some $y\in\mathbb{R}^n$, using Lemma~\ref{lemmaB}, we obtain
	\begin{equation*}
		|x|^{n-4}|\mathcal{U}(x)|=|x|^{n-4}|\mathcal{U}_{z,\mu^*(z)}(x)|\rightarrow\mu^*(z)^{n-4}|\mathcal{U}(z)| \quad {\rm as} \quad |x|\rightarrow\infty,
	\end{equation*}
	which is a contradiction.
	
	Combining Claim 1 and \cite[Lemma~11.1]{MR2001065}, we have that $\mathcal{U}$ is constant. Since $\mathcal{U}$ satisfies \eqref{regularsystem}, it follows that $\mathcal{U}\equiv0$.
\end{proof}

\subsection{Proof of Theorem~\ref{theorem1}}\label{proofoftheorem1}
Now using Lemmas~\ref{lemmaB} and Proposition~\ref{li-zhang1}, we have enough conditions to give a straightforward classification for $p$-solutions to \eqref{oursystem} in the blow-up limit critical case.  

\begin{proof}
	By Lemma \ref{lemmaB}, we can apply Proposition~\ref{li-zhang1} to find $\mu^{\prime}_i>0$ and $\mu^{\prime\prime}_i\geqslant0$ satisfying 
	\begin{equation*}
		u_i(x)=\left(\frac{\mu^{\prime}_i}{\mu^{\prime\prime}_i+|x-x_0|^2}\right)^{\frac{n-4}{2}} \quad {\rm for \ all} \quad x \in \mathbb{R}^n.
	\end{equation*}
	Then, we can obtain constants $\Lambda_{ij}>0$ depending on $\mu^{\prime}_i$, $\mu^{\prime\prime}_i$, $\mu^{\prime}_j$, and $\mu^{\prime\prime}_j$ such that $u_i\equiv \Lambda_{ij}u_j$ for any $i,j\in I_+$. 
	In particular, for all $i\in I_{+}$, we have the proportionality $u_i=\Lambda_iu_1$ where $\Lambda_i=\Lambda_{1i}$, which provides $\Delta^2u_1=c(n)|\Lambda^{\prime}|^{2^{**}-2}u^{2^{**}-1}_1$  in $ \mathbb{R}^n$, where $\Lambda^{\prime}=(\Lambda_i)_{i\in I_{+}}$. By Theorem~\ref{theoremA}, for some $x_0\in\mathbb{R}^n$ and $\mu>0$, we have that $u_1$ has the following {\it ansatz}
	\begin{equation*}
		u_1(x)=|\Lambda'|^{-1}\left(\frac{2\mu}{1+\mu^{2}|x-x_0|^{2}}\right)^{\frac{n-4}{2}},
	\end{equation*}
	which implies that our classification holds for $\Lambda=(\Lambda_1|\Lambda'|^{-1},\dots,\Lambda_p|\Lambda'|^{-1})$; thus the proof of Theorem~\ref{theorem1} is completed.
\end{proof}

\subsection{Classification for weak solutions}
As an aside, we use the weak formulation of solution to \eqref{oursystem} to prove a version of Theorem~\ref{theorem1}, which is based on the analysis of quotient functions $q_{ij}=u_i/u_j$ \cite{MR2603801}.
The main idea is use the integral representation for solutions to \eqref{oursystem}, which yields a more quantitative estimate for the constants $\Lambda_{ij}=q_{ij}$ (see \eqref{jurubeba}).
Before starting our method, we must be cautious that the quotient is well-defined since we may have solutions having zeros in the domain or even being identically null. 
By Proposition~\ref{positivity}, we know that the latter situation does not occur. Moreover, we can avoid the former situation by assuming that component solutions $u_i$ are strictly positive, that is, $i\in I_+$. 
Notice that Theorem~\ref{theorem1} is now equivalent to proving that all quotient functions are identically constant, {\it i.e.}, component solutions are proportional to each other $u_j=\Lambda_{ij}u_j$ \cite{MR3198645}.

Before, we need an auxiliary result, which is a variant of the classical strong Liouville-type result for biharmonic  functions in \cite[Theorem~1]{MR274792}.

\begin{lemma}\label{lm:classicallinearliouville}
	Let $q\in C^{4}(\mathbb{R}^n)$ be a nonnegative solution to the fourth linear problem 
	\begin{equation*}
		\Delta^2q+c_3(x)\nabla\Delta q+c_2(x)\Delta q+c_1(x)\nabla q=0 \quad {\rm in} \quad \mathbb{R}^n,
	\end{equation*}
	where $c_2\in C^{\infty}(\mathbb{R}^n)$ is a nonpositive smooth function, and $c_1,c_3\in C^{\infty}(\mathbb{R}^n,\mathbb{R}^n)$ are smooth matrices.
	If $q$ is bounded above and below, then $q$ is constant. 
\end{lemma}

\begin{proof}
	Since at critical point $x_0\in\mathbb{R}^n$ of $q$, one has that $\nabla q(x_0)=0$. The proof follows by noticing that $\nabla q$ is a solution a second order uniformly elliptic operator 
	a uniqueness result based on the weak maximum principle from \cite[Theorem~1]{MR312040} and the Harnack inequality from \cite[Theorem~3.6]{MR2240050}.
\end{proof}

We state the main result of this part, which is a fourth order version of \cite[Proposition~3.1]{MR2603801}

\begin{theoremtio}\label{hebey-druet}
	{\it 
		Let $\mathcal{U}$ be a weak nonnegative non-singular solution to \eqref{oursystem}. 
		Then, there exists $\Lambda\in\mathbb{S}^{p-1}_{+}$ such that $\mathcal{U}=\Lambda u_{x_0,\mu}$, where $u_{x_0,\mu}$ is a fourth order spherical solution given by \eqref{sphericalfunctions}.
	}
\end{theoremtio}

\begin{proof}
	For a weak nonnegative non-singular solution $\mathcal{U}$ to \eqref{oursystem} and $i,j\in I_{+}$, let us consider the quotient function $q_{ij}:\mathbb{R}^n\rightarrow(0,\infty)$ given by $q_{ij}:=u_i/u_j$. 
	Besides, by the smoothness result in Corollary~\ref{regregularity}, we get $q_{ij}\in C^{\infty}(\mathbb{R}^n)$ for all $i,j\in I_{+}$.
	Moreover, by \eqref{shape}, there exists $C>0$ satisfying $0\leqslant q_{ij}(x)\leqslant C$ for all $x\in\mathbb{R}^n$.
	
	In what follows, we divide the argument into two claims.
	The first one provides a strong classification for any quotient function.
	
	\noindent{\bf Claim 1:} For all $i,j\in I_+$, there exists a constant $\Lambda_{ij}>0$ such that $q_{ij}\equiv \Lambda_{ij}$.
	
	\noindent As a matter of fact, a straightforward computation yields
	\begin{equation*}
		\Delta^2 q_{ij}=\frac{u_j\Delta^2 u_i-u_i\Delta^2 u_j}{u_j^2}-\frac{4}{u_j}\nabla\Delta q_{ij}\nabla u_j-\frac{6}{u_j}\Delta q_{ij}\Delta u_j-\frac{4}{u_j}\nabla q_{ij}\nabla\Delta u_j.
	\end{equation*}
	Notice that since $\mathcal{U}$ solves \eqref{oursystem}, the first term on the right-hand side of the last equation is zero. 
	Thus, we are left with
	\begin{equation}\label{celular}
		\Delta^{2}q_{ij}+\frac{4}{u_j}\nabla u_j\nabla\Delta q_{ij}+\frac{6}{u_j}\Delta u_j\Delta q_{ij}+\frac{4}{u_j}\nabla\Delta u_j\nabla q_{ij}=0.
	\end{equation}
	The conclusion is a consequence of Lemma~\ref{lm:classicallinearliouville}.
	
	\noindent{\bf Claim 2:} For all $i,j\in I_+$, it follows
	\begin{equation}\label{jurubeba}
		\Lambda_{ij}=\frac{\displaystyle\int_{\mathbb{R}^n}|\mathcal{U}|^{2^{**}-2}u_i\ud x}{\displaystyle\int_{\mathbb{R}^n}|\mathcal{U}|^{2^{**}-2}u_j\ud x}.
	\end{equation}
	
	\noindent In fact, notice this equivalent to prove that 
	$\min_{\partial B_R(0)}q_{ij}\rightarrow\Lambda_{ij}$ and $\max_{\partial B_R(0)}q_{ij}\rightarrow\Lambda_{ij}$ as $R\rightarrow\infty$.
	To this end, we divide the proof into three steps.
	The first one concerns the behavior at infinity of component solutions to \eqref{oursystem}.
	
	\noindent{\bf Step 1:} $|x|^{{(n-4)}/{2}}u_i(x)=o_{R}(1)$ as $R\rightarrow\infty$.
	
	\noindent For $R>0$, let us consider the rescaling of $\mathcal{U}$ given by $\mathcal{W}_R=R^{{(n-4)}/{2}}\mathcal{U}(Rx)$, which in terms of component takes the form $(w_R)_i=R^{(n-4)/2}u_i(Rx)$. 
	Since $u_i\in L^{2^{**}}(\mathbb{R}^n)$, we get
	\begin{equation*}
		\Delta^2 (w_R)_i=c(n)|\mathcal{W}|^{2^{**}-2}(w_R)_i \quad {\rm and} \quad \displaystyle\int_{B_{2}(0)\setminus B_{1/2}(0)}|\mathcal{W}_R|^{2^{**}}\ud x=o_{R}(1) \quad {\rm as} \quad R\rightarrow\infty.
	\end{equation*}
	Thus, $(w_R)_i\rightarrow0$ in $C^{\infty}_{\loc}(B_{3/2}(0)\setminus B_{3/4}(0))$ as $R\rightarrow\infty$.
	
	In the next step, we obtain an upper bound for component solutions to \eqref{oursystem}, which provides an interpolation estimate, showing that $u_i\in L^{p}(\mathbb{R}^n)$ for $2<p<2^{**}$.
	
	\noindent{\bf Step 2:} For any $0<\varepsilon<1/2$, there exists $C_{\varepsilon}>0$ such that $u_i(x)\leqslant C_{\varepsilon}|x|^{(4-n)(1-\varepsilon)}$ for all $x\in\mathbb{R}^n$.
	
	\noindent First, by {Step 1} for a given $0<\varepsilon<1/2$, there exists $R_{\varepsilon}\gg1$ sufficiently large satisfying 
	\begin{equation}\label{luxemborg}
		\displaystyle\sup_{\mathbb{R}^n\setminus B_{R_{\varepsilon}}(0)}|x|^{2}|\mathcal{U}(x)|^{2^{**}-2}<\frac{(n-4)^2}{2}\varepsilon(1-\varepsilon).
	\end{equation}
	For $R\geqslant R_{\varepsilon}$, let us consider $\sigma(R)=\displaystyle\max_{i\in I_{+}}\max_{\partial{B_{R}(x_0)}}u_i$ and the auxiliary function
	\begin{equation*}
		G_{\varepsilon}(x)=\sigma(R_{\varepsilon})\left(\frac{|x|}{R_{\varepsilon}}\right)^{(4-n)(1-\varepsilon)}+\sigma(R)\left(\frac{|x|}{R}\right)^{(4-n)\varepsilon}.
	\end{equation*}
	Notice that, by construction, we clearly have that $u_i\leqslant G_{\varepsilon}$ on $\partial B_{R}(0)\cup\partial B_{R_{\varepsilon}}(0)$. 
	Let us suppose that there exists $x_0\in B_{R}(0)\setminus \bar{B}_{R_{\varepsilon}}(0)$, a maximum point of $u_i/G_{\varepsilon}$, which would imply that $\Delta(u_iG^{-1}_{\varepsilon}(x_0))\leqslant0$, and then 
	\begin{equation}\label{eslovenia}
		\frac{\Delta u_i(x)}{u_i(x)}\geqslant\frac{\Delta G^{-1}_{\varepsilon}(x)}{G^{-1}_{\varepsilon}(x)}.
	\end{equation}
	Furthermore, a direct computation implies
	\begin{equation}\label{crocia}
		\Delta G^{-1}_\varepsilon(x)=G^{-1}_{\varepsilon}(x)\frac{(n-4)^2}{2}\varepsilon(1-\varepsilon)|x|^{-2}.
	\end{equation}
	Therefore, by Proposition~\ref{superharmonicity} we obtain that $\Delta^{2}u_i(x)-\Delta u_i(x)\geqslant0$, which combined with \eqref{eslovenia}-\eqref{crocia} yields 
	\begin{equation*}
		|x|^2|\mathcal{U}(x)|^{2^{**}}=\frac{\Delta^2 u_i(x)}{u_i(x)}\geqslant\frac{\Delta u_i(x)}{u_i(x)}\geqslant\frac{\Delta G^{-1}_{\varepsilon}(x)}{G^{-1}_{\varepsilon}(x)}=\frac{(n-4)^2}{2}\varepsilon(1-\varepsilon).
	\end{equation*}
	This is a contradiction with \eqref{luxemborg} since our choice of $R_{\varepsilon}>0$. Then, applying the strong maximum principle, we have
	\begin{equation}\label{nicaragua}
		u_i(x)\leqslant\sigma(R_{\varepsilon})\left(\frac{|x|}{R_{\varepsilon}}\right)^{(4-n)(1-\varepsilon)}+\sigma(R)\left(\frac{|x|}{R}\right)^{(4-n)\varepsilon} \quad {\rm in} \quad B_{R}(0)\setminus\bar{B}_{R_{\varepsilon}}(0),
	\end{equation}
	for all $R>R_{\varepsilon}$.
	Thus, using \eqref{nicaragua} combined with Step 1, and taking the limit as $R\rightarrow\infty$, we get
	\begin{equation*}
		u_i(x)\leqslant\sigma(R_{\varepsilon})\left(\frac{|x|}{R_{\varepsilon}}\right)^{(4-n)(1-\varepsilon)} \quad {\rm in} \quad \mathbb{R}^n.
	\end{equation*}
	\noindent{\bf Step 3:} $|x|^{4-n}u_i(x)=\int_{\mathbb{R}^n}c(n)|\mathcal{U}(x)|^{2^{**}-2}u_i(x)\ud x+o_{R}(1)$ as $R\rightarrow\infty$.\\
	\noindent First, since $u_i\in L^{2^{**}}(\mathbb{R}^n)$, we have $|\mathcal{U}|^{2^{**}-2}u_i\in L^{2n/(n+4)}(\mathbb{R}^n)$ for all $i\in I$, which implies $|\mathcal{U}|^{2^{**}-2}u_i\in W^{-2,2}(\mathbb{R}^n)$. Hence, we get that \eqref{oursystem} can be reduced to the following integral system,
	\begin{equation*}
		u_i(x)=C_n\displaystyle\int_{\mathbb{R}^n}|x-y|^{4-n}|\mathcal{U}(y)|^{2^{**}-2}u_i(y)\ud y,
	\end{equation*}
	from which follows
	\begin{align*}
		|x|^{n-4}u_i(x)
		=C_n\displaystyle\int_{\mathbb{R}^n}\left(\frac{|x|}{|x-y|}\right)^{n-4}|\mathcal{U}(y)|^{2^{**}-2}u_i(y)\ud y
		=C_n(I_1+I_2),
	\end{align*}
	where 
	\begin{equation*}
		I_1=\int_{B_{R}(0)}\left(\frac{|x|}{|x-y|}\right)^{n-4}|\mathcal{U}(y)|^{2^{**}-2}u_i(y)\ud y, \; 
		I_2= \displaystyle\int_{\mathbb{R}^n\setminus B_{R}(0)}\left(\frac{|x|}{|x-y|}\right)^{n-4}|\mathcal{U}(y)|^{2^{**}-2}u_i(y)\ud y.
	\end{equation*}
	To control $I_1$, we observe that since 
	\begin{equation}\label{mari}
		\displaystyle\int_{B_R(0)}\left[\left(\frac{|x|}{|x-y|}\right)^{n-4}-1\right]|\mathcal{U}(y)|^{2^{**}-2}u_i(y)\ud y=o_{R}(1),
	\end{equation}
	the following asymptotic identity holds
	\begin{align}\label{finland}
		I_1=\displaystyle\int_{B_R(0)}|\mathcal{U}(y)|^{2^{**}-2}u_i(y)\ud y+o_{R}(1) \quad {\rm as} \quad R\rightarrow\infty,
	\end{align}
	where the identity \eqref{mari} holds because the integrand is bounded.
	
	Now it remains to estimate $I_2$. Accordingly, using {Step 2}, we can write
	\begin{align}\label{iceland}
		I_2&=\displaystyle\int_{\mathbb{R}^n\setminus B_{R}(0)}\left(\frac{|x|}{|x-y|}\right)^{n-4}|\mathcal{U}(y)|^{2^{**}-2}u_i(y)\ud y&\\\nonumber
		&\leqslant\displaystyle\int_{B_{|x|/2}(x)}\left(\frac{|x|}{|x-y|}\right)^{n-4}|\mathcal{U}(y)|^{2^{**}-2}u_i(y)\ud y+\displaystyle\int_{\mathbb{R}^n\setminus B_{|x|/2}(x)}\left(\frac{|x|}{|x-y|}\right)^{n-4}|\mathcal{U}(y)|^{2^{**}-2}u_i(y)\ud y&\\\nonumber
		&\leqslant C_{\varepsilon}^{2^{**}-1}\displaystyle\int_{B_{|x|/2}(x)}\left(\frac{|x|}{|x-y|}\right)^{n-4}\left(\frac{|x|}{2}\right)^{-(n+4)(1-\varepsilon)}\ud y+2^{n-4}\displaystyle\int_{\mathbb{R}^n\setminus B_R(0)}|\mathcal{U}(y)|^{2^{**}-2}u_i(y)\ud y&\\\nonumber
		&\leqslant C_{\varepsilon}^{2^{**}-1}2^{(n+4)(1-\varepsilon)-2}\omega_{n-1}|x|^{n-(n+4)(1-\varepsilon)}+2^{n-4}\displaystyle\int_{\mathbb{R}^n\setminus B_R(x)}|\mathcal{U}(y)|^{2^{**}-2}u_i(y)\ud y,&
	\end{align}
	Choosing $\varepsilon=4/(n+4)$ in \eqref{iceland}, we obtain that $n-(n+4)(1-\varepsilon)\leqslant0$, and so
	\begin{equation*}
		I_2=o_{R}(1) \quad {\rm as} \quad R\rightarrow\infty,
	\end{equation*}
	which combined with \eqref{iceland} and \eqref{finland}, concludes the proof of Step 3.
	
	Now using Step 3, we obtain that for all $i,j\in I_{+}$, it holds
	\begin{equation*}
		q_{ij}(x)=\frac{u_i(x)}{u_j(x)}=\frac{|x|^{n-4}u_i(x)}{|x|^{n-4}u_j(x)}=\frac{\int_{\mathbb{R}^n}|\mathcal{U}(x)|^{2^{**}-2}u_i(x)\ud x+o_{R}(1)}{\int_{\mathbb{R}^n}|\mathcal{U}(x)|^{2^{**}-2}u_j(x)\ud x+o_{R}(1)},
	\end{equation*}
	which by taking the limit as $R\rightarrow\infty$ yields \eqref{jurubeba}. 
	
	Finally, combining Claims 1 and 2, we find that $u_i=\Lambda_{ij}u_j$, which concludes the proof using the same argument as in the other proof in the last section.
\end{proof}

Using Lemma~\ref{lemmaB} and the classification in Proposition~\ref{li-zhang1}, we can compute the $\mathcal{D}^{2,2}(\mathbb{R}^n,\mathbb{R}^p)$-norm of any classical solution to \eqref{regularsystem}, which enables us to conclude that classical solutions are weak solutions, then Theorem~\ref{hebey-druet}\textcolor{blue}{'} can be applied to give an alternative proof for Theorem~\ref{theorem1} (See Remark~\ref{classical/weak}). 

\begin{proof}[Alternative proof of Theorem~\ref{theorem1}]
	By Lemma~\ref{lemmaC}, we may assume $\mu^*(y)<\infty$ for any $y\in\mathbb{R}^n$. Moreover, using Proposition~\ref{li-zhang1}, there exist $x_0\in\mathbb{R}^n$ and $\mu^{\prime}_i>0$ and $\mu^{\prime\prime}_i\geqslant0$ such that
	\begin{equation}\label{shape}
		|\mathcal{U}(x)|=\left(\frac{\mu^{\prime}_i}{\mu^{\prime\prime}_i+|x-x_0|^2}\right)^{\frac{n-4}{2}} \quad {\rm for \ all} \quad x \in \mathbb{R}^n.
	\end{equation}
	
	Let us consider a smooth cut-off function satisfying $\eta\equiv1$ in $[0,1]$, $0\leqslant\eta\leqslant1$ in $[1,2)$ and $\eta\equiv0$ in $[2,\infty)$. For $R>0$, setting $\eta_R(x)=\eta(R^{-1}x)$, and multiplying the equation \eqref{oursystem} by $\eta_{R}u_i$, we obtain $\Delta^{2}u_i\eta_{R}u_i=|\mathcal{U}|^{2^{**}-2}\eta_{R}u^2_i$,
	which gives us
	\begin{equation*}
		\displaystyle\sum_{i=1}^{p}\Delta^{2}u_i\eta_{R}u_i=c(n)\displaystyle\sum_{i=1}^{p}|\mathcal{U}|^{2^{**}-2}\eta_{R}u^2_i=c(n)|\mathcal{U}|^{2^{**}}\eta_R.
	\end{equation*}
	Thus, 
	\begin{equation}\label{cansado}
		\displaystyle\int_{\mathbb{R}^n}\displaystyle\sum_{i=1}^{p}\Delta^{2}u_i\eta_{R}u_i \ \ud x=\displaystyle c(n)\int_{\mathbb{R}^n}|\mathcal{U}|^{2^{**}}\eta_R \ \ud x.
	\end{equation}
	Using integration by parts on the left-hand side,
	\begin{equation}\label{chato}
		\displaystyle\int_{\mathbb{R}^n}\displaystyle\sum_{i=1}^{p}\Delta^{2}u_i\eta_{R}u_i\ud x=\displaystyle\sum_{i=1}^{p}\displaystyle\int_{\mathbb{R}^n}u_i\Delta^2(\eta_{R}u_i)\ud x.
	\end{equation}
	Applying the formula for the bi-Laplacian of the product on the right-hand side of \eqref{chato},
	\begin{align*}
		\displaystyle\sum_{i=1}^{p}\displaystyle\int_{\mathbb{R}^n}u_i\Delta^2(\eta_{R}u_i)\ud x = & \displaystyle\sum_{i=1}^{p}\displaystyle\int_{\mathbb{R}^n}\left[u_i\Delta^2(\eta_R)u_i+4u_i\nabla\Delta\eta_{R}\nabla u_i \right]\ud x&\\
		+ &\displaystyle\sum_{i=1}^{p}\displaystyle\int_{\mathbb{R}^n}\left[
		6u_i\Delta\eta_R\Delta u_i+4u_i\nabla\eta_{R}\nabla\Delta u_i+u_i\eta_R\Delta^2 u_i\right]\ud x,&
	\end{align*}
	which combined with \eqref{chato} provides
	\begin{align}\label{demais}
		&\displaystyle\sum_{i=1}^{p}\displaystyle\int_{\mathbb{R}^n}\left[u_i\Delta^2(\eta_R)u_i+4u_i\nabla\Delta\eta_{R}\nabla u_i+6u_i\Delta\eta_R\Delta u_i+4u_i\eta_{R}\nabla\Delta u_i\right]\ud x=0.&
	\end{align}
	Again, we use integration by parts in \eqref{demais} to find
	\begin{align*}
		&\displaystyle\sum_{i=1}^{p}\left[\displaystyle\int_{\mathbb{R}^n}u_i^2\Delta\eta_{R}\ud x-4\left(\displaystyle\int_{\mathbb{R}^n}\Delta\eta_R|\nabla u_i|^2\ud x+\displaystyle\int_{\mathbb{R}^n}u_i\Delta\eta_R\Delta u_i\ud x\right)\right.&\\
		&\left.+6\displaystyle\int_{\mathbb{R}^n}u_i\Delta\eta_R\Delta u_i\ud x-4\left(\displaystyle\int_{\mathbb{R}^n}u_i\eta_R\Delta^2 u_i\ud x+\displaystyle\int_{\mathbb{R}^n}\eta_R\nabla u_i\nabla\Delta u_i\ud x\right)\right]=0,&
	\end{align*}
	which yields
	\begin{align}\label{austria}
		4\displaystyle\sum_{i=1}^{p}\displaystyle\int_{\mathbb{R}^n}\Delta^{2}u_i\eta_Ru_i\ud x
		&=\displaystyle\int_{\mathbb{R}^n}(u_i)^2\Delta^2\eta_R\ud x-4\displaystyle\int_{\mathbb{R}^n}\Delta\eta_R|\nabla u_i|^2\ud x\\\nonumber
		&+2\displaystyle\int_{\mathbb{R}^n}u_i\Delta\eta\Delta u_i\ud x+4\displaystyle\int_{\mathbb{R}^n}\Delta u_i\nabla u_i\nabla\eta_{R}\ud x+4\displaystyle\int_{\mathbb{R}^n}\eta_R|\Delta u_i|^2\ud x.
	\end{align}
	As a result of \eqref{cansado} and \eqref{austria}, we obtain
	\begin{align}\label{rosario}
		\displaystyle\int_{\mathbb{R}^{n}}|\mathcal{U}|^{2^{**}}\eta_{R}\ud x
		&=\frac{1}{4}\displaystyle\int_{\mathbb{R}^{n}}|\mathcal{U}|^2\Delta^2\eta_{R}\ud x-\displaystyle\int_{\mathbb{R}^{n}}|\nabla\mathcal{U}|^2\Delta\eta_{R}\ud x\\\nonumber
		&+\frac{1}{2}\displaystyle\int_{\mathbb{R}^{n}}\langle\mathcal{U},\Delta\mathcal{U}\rangle\Delta\eta_{R}\ud x+\displaystyle\int_{\mathbb{R}^{n}}\langle\Delta\mathcal{U},\nabla \mathcal{U}\rangle\nabla\eta_R\ud x+\displaystyle\int_{\mathbb{R}^{n}}|\Delta\mathcal{U}|^2\eta_{R}\ud x.
	\end{align}
	Moreover, we have
	\begin{equation*}
		\displaystyle\int_{\mathbb{R}^{n}}|\mathcal{U}|^2\Delta^2\eta_{R}\ud x=\mathcal{O}(R^{4-n}) \quad {\rm as} \quad R\rightarrow\infty.
	\end{equation*}
	Indeed, we observe 
	\begin{align*}
		\displaystyle\left|\int_{\mathbb{R}^{n}}|\mathcal{U}|^2\Delta^2\eta_{R}\ud x\right|&\leqslant\displaystyle\int_{\mathbb{R}^{n}}|\mathcal{U}|^2|\Delta^2\eta_{R}|\ud x\\
		&\leqslant\|\Delta^2\eta_R\|_{C^{0}(\mathbb{R}^n)}\displaystyle\int_{B_{2R}(0)\setminus B_{R}(0)}|\mathcal{U}|^2\ud x\\
		&\leqslant\frac{\|\Delta^2\eta\|_{C^{0}(\mathbb{R}^n)}}{R^4}\displaystyle\int_{R}^{2R}|\mathcal{U}(r)|^2r^{n-1}\ud r\\
		&\leqslant\frac{\|\Delta^2 \eta\|_{C^{0}(\mathbb{R}^n)}\|\mathcal{U}\|^2_{L^{\infty}(\mathbb{R}^n)}}{R^4}\displaystyle\int_{R}^{2R}r^{n-1}\ud r\\
		&=C(n)R^{n-4}.
	\end{align*}
	Analogously to the others terms, we get the following estimates
	\begin{gather*}
		\displaystyle\int_{\mathbb{R}^{n}}|\nabla\mathcal{U}|^2\Delta\eta_{R}\ud x=\mathcal{O}(R^{2-n}) \ \mbox{and} \
		\displaystyle\int_{\mathbb{R}^{n}}\langle\mathcal{U},\Delta\mathcal{U}\rangle\Delta\eta_{R}\ud x=\displaystyle\int_{\mathbb{R}^{n}}\langle\Delta\mathcal{U},\nabla\mathcal{U}\rangle\nabla\eta_R\ud x=\mathcal{O}(R^{1-n})
	\end{gather*} 
	as $R\rightarrow\infty$,
	which, by taking $R\rightarrow\infty$ in \eqref{rosario}, we find that $\eta_{R}\rightarrow1$ in the $C^{0}(\mathbb{R}^n)$-topology, and 
	\begin{equation*}
		\displaystyle\int_{\mathbb{R}^{n}}|\Delta\mathcal{U}|^2\ud x=c(n)\displaystyle\int_{\mathbb{R}^{n}}|\mathcal{U}|^{2^{**}}\ud x<\infty.
	\end{equation*}
	Since $|\mathcal{U}|$ has the classification \eqref{shape}, a direct computation yields
	\begin{equation*}
		\int_{\mathbb{R}^{n}}|\mathcal{U}|^{2^{**}}\ud x=S(2,2,n)^{-n},
	\end{equation*}
	where $S(2,2,n)$ is the best constant of Sobolev defined in \eqref{fourthtalaub}. Hence, $\mathcal{U}\in \mathcal{D}^{2,2}(\mathbb{R}^n,\mathbb{R}^p)$ is a weak solution to \eqref{oursystem}, and the proof follows as a direct application of Theorem~\ref{hebey-druet}.
\end{proof}

\begin{remark}
	System \eqref{regularsystem} is equivalent to the following integral system
	\begin{equation}\label{globalintegralsystem}
		u_i(x)=C_n\int_{\mathbb{R}^n}|x-y|^{4-n}f_i(\mathcal{U}(y))\ud y \quad {\rm in} \quad \mathbb{R}^n.
	\end{equation}
	In the sense that every solution to \eqref{regularsystem} is a solution \eqref{globalintegralsystem} plus a constant, and the reciprocal also holds. W. Chen and C. Li \cite[Theorem~3]{MR2510000} used the moving planes method in their integral form to classify solutions to a class of systems like \eqref{globalintegralsystem} involving more general nonlinearities. Let us mention that this approach can also be extended to study higher order systems.
\end{remark}

\subsection{Maximizers for a vectorial Sobolev inequality} 
We show that solutions obtained in Theorem \ref{theorem1} are the extremal $p$-maps for a type of vectorial higher order Sobolev embedding  \cite{MR2852264,MR1992377,MR1911916,MR2221095,MR1967041}. 
As usual, let us denote by  $\mathcal{D}^{k,q}(\mathbb{R}^n,\mathbb{R}^p)$ the Beppo--Levi space defined as the completion of  $C_c^{\infty}(\mathbb{R}^n,\mathbb{R}^p)$ with respect to the norm provided by the highest derivative term. 
Notice that if $q=2$, then $\mathcal{D}^{k,2}(\mathbb{R}^n,\mathbb{R}^p)$ is a Hilbert space furnished with the scalar product given by $\langle\mathcal{U},\mathcal{V}\rangle=\sum_{i=1}^{p}\langle u_i,v_i\rangle_{\mathcal{D}^{k,2}(\mathbb{R}^n)}$. Moreover, for the {\it higher order critical Sobolev exponent} $q_k^{*}=nq/(n-kq)$, we have the continuous embedding, $\mathcal{D}^{k,q}(\mathbb{R}^n,\mathbb{R}^p)\hookrightarrow L^{q_k^{*}}(\mathbb{R}^n,\mathbb{R}^p)$ with
\begin{equation*}
\|\mathcal{U}\|_{L^{q_k^{*}}(\mathbb{R}^n,\mathbb{R}^p)}\leqslant S(k,q,n,p)\|\mathcal{U}\|_{\mathcal{D}^{k,q}(\mathbb{R}^n,\mathbb{R}^p)}.
\end{equation*}
In this fashion, a natural problem to obtain and classify extremal functions and best constants for the inequality above.

For the scalar case, the celebrated papers \cite{MR0448404,MR0463908} contain the sharp Sobolev constant for $k=1$ as follows
\begin{equation*}
	S(1,q,n)=
	\begin{cases}
		\pi^{-\frac{1}{2}}n^{-\frac{1}{q}}\left(\frac{q-1}{n-q}\right)^{1-\frac{1}{q}}\left[\frac{\Gamma\left(1+\frac{n}{2}\right)\Gamma_2(n)}{\Gamma\left(\frac{n}{q}\right)\Gamma\left(n+1-\frac{n}{q}\right)}\right]^{-\frac{1}{n}},& {\rm if} \ 1<q<n\\
		\pi^{-\frac{1}{2}}{n^{-1}}\left[\Gamma\left(1+\frac{n}{2}\right)\right]^{-\frac{1}{n}},& {\rm if} \ q=1,
	\end{cases}
\end{equation*}
with extremals given by the spherical functions, {\it i.e.}, for some $\mu>0$ and $x_0\in\mathbb{R}^n$,
\begin{equation*}
	u(x)=\left(\frac{2\mu}{1+\mu^2|x-x_0|^{q/(q-1)}}\right)^{\frac{n-q}{q}}.
\end{equation*}
In particular, when $q=2$, we get 
\begin{equation*}
	S(1,2,n)=\left(\frac{4}{n(n-2)\omega_n^{2/n}}\right)^{1/2} \quad {\rm and} \quad u_{x_0,\mu}(x)=\left(\frac{2\mu}{1+\mu^2|x-x_0|^2}\right)^{\frac{n-2}{2}}.
\end{equation*}
On the fourth order case, $k=2$ and $q=2$, C. S. Lin \cite{MR1611691} found the best constant and characterized the set of maximizers,
\begin{equation}\label{fourthtalaub}
	S(2,2,n)=\left(\frac{16}{n(n-4)(n^2-4)\omega_n^{{4}/{n}}}\right)^{1/2} \quad {\rm and} \quad u_{x_0,\mu}(x)=\left(\frac{2\mu}{1+\mu^2|x-x_0|^2}\right)^{\frac{n-4}{2}}.
\end{equation}

In the vectorial case, we quote the second order Sobolev inequality
\begin{equation}\label{mimosa}
	\|\mathcal{U}\|_{L^{2^{*}}(\mathbb{R}^n,\mathbb{R}^p)}\leqslant S(1,2,n,p)\|\mathcal{U}\|_{\mathcal{D}^{1,2}(\mathbb{R}^n,\mathbb{R}^p)},
\end{equation}
where the extremal maps are the multiples of the second order spherical functions and $S(1,2,n,p)=S(1,2,n)$ for all $p>1$. Let us also consider the fourth order case of \eqref{mimosa} as
\begin{equation}\label{fourthvectsob}
	\mathcal{D}^{2,2}(\mathbb{R}^n,\mathbb{R}^p)\hookrightarrow L^{2^{**}}(\mathbb{R}^n,\mathbb{R}^p). 
\end{equation}
Ou main result here states that the solutions to \eqref{regularsystem} are the extremal functions for
\begin{equation}\label{fourthsobine}
	\|\mathcal{U}\|_{\mathcal{D}^{2,2}(\mathbb{R}^n,\mathbb{R}^p)}\leqslant S(2,2,n,p)\|\mathcal{U}\|_{L^{2^{**}}(\mathbb{R}^n,\mathbb{R}^p)}.
\end{equation}
Remarkably, the best constant in \eqref{fourthsobine} coincides with the one when $p=1$, that is, it follows that $S(2,2,n,1)=S(2,2,n,p)$ for all $p>1$. 
In other terms, the number of equations of the system has no effects in the best Sobolev constant for product spaces. In what follows, let us fix the notation $S(2,2,n,p)=S(n,p)$. 

\begin{proposition}
	Let ${\mathcal{U}}_{x_0,\mu}$ be a spherical solution to \eqref{regularsystem}. Then, up to constant, ${\mathcal{U}}_{x_0,\mu}$ is the unique extremal family of extremal $p$-maps for the Sobolev inequality \eqref{fourthsobine}, that is, 
	\begin{equation}\label{equalitysobolev}
		\|{\mathcal{U}}_{x_0,\mu}\|_{\mathcal{D}^{2,2}(\mathbb{R}^n,\mathbb{R}^p)}= S(n,p)\|{\mathcal{U}}_{x_0,\mu}\|_{L^{2^{**}}(\mathbb{R}^n,\mathbb{R}^p)}.
	\end{equation}
	Moreover, $S(n,p)=S(n)$ for all $p>1$.
\end{proposition}

\begin{proof}
	Initially, we observe
	\begin{equation}\label{varsobconst}
		S(n,p)^{-2}=\inf_{\mathcal{H}^p(\mathbb{R}^n)}\sum_{i=1}^{p}\int_{\mathbb{R}^n}|\Delta u_i|^{2}\ud x, \; {\rm where} \; \mathcal{H}^p(\mathbb{R}^n)=\left\{\mathcal{U}\in\mathcal{D}^{2,2}(\mathbb{R}^n,\mathbb{R}^p) : \|\mathcal{U}\|_{L^{2^{**}}(\mathbb{R}^n,\mathbb{R}^p)}=1\right\}.
	\end{equation}
	When $p=1$ our result is a consequence of Theorem~\ref{theoremA} with best constant $S(n)$ given by \eqref{fourthtalaub}.
	
	\noindent{\bf Claim 1:} $S(n,p)=S(n)$ for all $p>1$.
	
	\noindent In fact, by taking $u\in \mathcal{D}^{2,2}(\mathbb{R}^n)$ satisfying $\|u\|_{L^{2^{**}}(\mathbb{R}^n)}=1$, we have that $\mathcal{U}=u{\bf e_1}$ belongs to $\mathcal{H}^p(\mathbb{R}^n)$, where ${\bf e_1}=(1,0,\dots,0)$. Substituting $\mathcal{U}$ in \eqref{varsobconst}, we get that $S(n,p)\leqslant S(n)$. Conversely, we have
	\begin{align}\label{hebeypot}
		\left(\sum_{i=1}^{p}\int_{\mathbb{R}^n}|u_i|^{2^{**}}\ud x\right)^{{2}/{2^{**}}}&\leqslant\left(\sum_{i=1}^{p}\left(S(n,p)^{-1}\int_{\mathbb{R}^n}|\Delta^2 u_i|\ud x\right)^{2^{**}/2}\right)^{{2}/{2^{**}}}&\\\nonumber
		&\leqslant S(n)^{-1}\sum_{i=1}^{p}\int_{\mathbb{R}^n}|\Delta u_i|^{2}\ud x.&
	\end{align}
	Therefore, by \eqref{hebeypot} we find that $S(n,p)^{-1}\leqslant S(n)^{-1}$, which gives us the proof of the claim.	
	Also, using the following computation
	\begin{equation*}
		\frac{\|{\mathcal{U}}_{x_0,\mu}\|_{\mathcal{D}^{2,2}(\mathbb{R}^n,\mathbb{R}^p)}}{\|{\mathcal{U}}_{x_0,\mu}\|_{L^{2^{**}}(\mathbb{R}^n,\mathbb{R}^p)}}=\frac{\|u_{x_0,\mu}\|_{\mathcal{D}^{2,2}(\mathbb{R}^n)}}{{\|u_{x_0,\mu}}\|_{L^{2^{**}}(\mathbb{R}^n)}}=S(n).
	\end{equation*}
	To prove the uniqueness, observe that if $\mathcal{U}$ is such that \eqref{equalitysobolev}, then, up to constant, $\mathcal{U}$ satisfy $\Delta^2\mathcal{U}=|\mathcal{U}|^{2^{**}-2}\mathcal{U}$ in $\mathbb{R}^n$, which, by Theorem~\ref{theorem1} concludes the proof of the proposition. 
\end{proof}

\section{Classification result for singular solutions}\label{section4}
The objective of this section is to present the proof of Theorem~\ref{theorem2}. 
We show that singular solutions to \eqref{oursystem} are radially symmetry about the origin.
Then, we obtain radial symmetry via an asymptotic moving planes technique due to \cite{MR982351} (see also \cite{MR1611691,MR3394387,MR4002167}); this property turns \eqref{oursystem} into a fourth order ODE system. 
Eventually, we define a Pohozaev-type invariant by integrating the Hamiltonian energy of the associated Emden-Fowler system \cite{MR1666838,MR2393072,MR4085120,MR3869387,MR4094467,MR4123335}. 
Moreover, we prove that the Pohozaev invariant sign provides a removable-singularity classification for nonnegative solutions to \eqref{oursystem}, which combined with a delicate ODE analysis as in \cite{MR3869387} completes our argument. 

\subsection{Regularity}
As in Proposition~\ref{regregularity}, an important question is whether weak singular solutions to \eqref{oursystem} are classical solutions as well. This the content of the next result.

\begin{proposition}\label{singregularity}
	Let $\mathcal{U}\in\mathcal{D}^{2,2}(\mathbb{R}^n\setminus\{0\},\mathbb{R}^p)$ be a nonnegative weak singular solution to \eqref{oursystem}. Then, $\mathcal{U}\in C^{4,\zeta}(\mathbb{R}^n\setminus\{0\},\mathbb{R}^p)$ is a  classical singular solution to \eqref{oursystem}.
\end{proposition}

	\begin{proof}
		Since the right hand-side of \eqref{oursystem} belong to same Lebesgue class as in \eqref{scalarsystem}, that is, $f_i(\mathcal{U})=c(n)|\mathcal{U}|^{2^{**}-2}u_i\in L^{\frac{2n}{n+4}}(\mathbb{R}^n\setminus\{0\})$ for all $i\in I$, we can directly apply the regularity bootstrap method to each component $u_i$ as in the proof of \cite[theorem~1.1]{MR1809291} (see also Remark~\ref{classical/weak}). 
	\end{proof}

\subsection{Superharmonicity} 
Now, we use the radial symmetry to prove that any component of a nonnegative singular solution to \eqref{oursystem} is superharmonic.
The next result is a version of Proposition~\ref{superharmonicity} for singular solution to \eqref{oursystem}, which can be found in \cite[Lemma~2.3]{MR4094467} for the scalar case. 
We also remark that \cite[Theorem~1.4]{MR4094467} contains a improved version of this result called a Modica estimate.

\begin{proposition}\label{supersingular} 
	Let $\mathcal{U}$ be a nonnegative singular solution to \eqref{oursystem}. Then, $-\Delta u_i\geqslant0$ in $\mathbb{R}^n\setminus\{0\}$ for all $i\in I$.
\end{proposition}

\begin{proof} 
	Let us recall that $u_i(r)=r^{-\gamma}v_i(-\ln r)$, thus $u_i(r)\geqslant C_1r^{-\gamma}$, which together with \eqref{oursystem} implies 
	\begin{equation*}
		0<\omega_{n-1}r^{n-1}{\partial_r}\Delta u_i(r)=c(n)\int_{B_r}|\mathcal{U}|^{2^{**}-1}u_i\ud x,
	\end{equation*}
	for $0<r\ll1$ sufficiently small. Then, we get
	\begin{equation}\label{supsing2}
		\lim_{r\rightarrow0^{+}}r^{n-1}{\partial_r}\Delta u_i(r)=0.
	\end{equation}
	Moreover, $u_i$ satisfies
	\begin{equation*}
		{\partial_r}\left[r^{n-1}{\partial_r}\Delta u_i(r)\right]=r^{n-1}c(n)|\mathcal{U}|^{2^{**}-2}u_i,
	\end{equation*}
	which combined with \eqref{supsing2} gives us that ${\partial_r}\Delta u_i(r)>0$. Therefore, $\Delta u_i(r)$ is strictly increasing, and by the relation between $u_i$ and $v_i$, we find that $\lim_{r\rightarrow\infty}\Delta u_i(r)=0$, which completes the proof.
\end{proof}

\begin{remark}
	Another way of proving this superharmonicity property for singular solutions in the punctured space is to use the arguments in \cite[Theorem~3.7]{MR1769247} (or \cite[Lemma~3.2]{MR4123335}) for each component $u_i$. 
	In this manner, we would avoid using the rotational invariance of vectorial solutions to \eqref{oursystem}. 
\end{remark}
	
\subsection{Moving planes technique}
In this subsection, using a variant of the {\it moving planes technique}, we prove that singular solutions to \eqref{oursystem} are radially symmetric about the origin. 
The first work proving radial symmetry for solutions to PDEs via this method is due to J. Serrin \cite{MR0333220} (see also \cite{MR634248,MR544879}).
His approach was based on the reflection method developed earlier by A. D. Aleksandrov to study embedded surfaces of constant mean curvature. 
In our case, solutions are singular at the origin, thus, to show that they are rotationally invariant, we need to perform an adaptation of Aleksandrov's method \cite{MR982351}. Furthermore, this tool can be extended to fourth order problems as in \cite{MR3394387} and for second order systems \cite{MR1611691,MR4002167} with isolated singularities. 
To the best of our knowledge, our result are the first to use this method in the context of strongly coupled fourth order systems.

To prove our main result, we require three background lemmas from \cite[Section~2]{MR982351}

\begin{lemmaletter}\label{harmasymp}
	Let $\vartheta$ be a harmonic function and consider $({\vartheta})_{z,1}=|x|^{2-n}\vartheta\left(z+{x}{|x|^{-2}}\right)$ the second order Kelvin transform of $\vartheta$, which for simplicity it is denoted by $({\vartheta})_{z,1}=\widetilde{\vartheta}$. Then, $\widetilde{\vartheta}$ is harmonic in a neighborhood at infinity, and it satisfies the asymptotic expansion
	\begin{equation}\label{harmonicexpansion}
	\begin{cases}
	\widetilde{\vartheta}(x)=\mathfrak{a}_{0}|x|^{2-n}+\mathfrak{a}_{j}x_j|x|^{-n}+\mathcal{O}(|x|^{-n})\\
	\partial_{x_j}\widetilde{\vartheta}(x)=(2-n)\mathfrak{a}_{0}x_j|x|^{-n}+\mathcal{O}(|x|^{-n})\\
	\partial_{x_{k}x_{j}}\widetilde{\vartheta}(x)=\mathcal{O}(|x|^{-n}).
	\end{cases}
	\end{equation}
	where $\{\mathfrak{a}_j\}_{j\in\mathbb{N}}\in\mathbb{R}$ are the coefficients of the Taylor expansion.
\end{lemmaletter}

\begin{lemmaletter}\label{starting}
	Let $\vartheta$ be a positive function defined in a neighborhood at infinity satisfying the asymptotic expansion \eqref{harmonicexpansion}. 
	Then, there exist $\bar{\lambda}<0$ and $R>0$ such that $\vartheta(x)>\vartheta(x_{\lambda})$ for $\lambda\leqslant\bar{\lambda}$, $|x|\geqslant R$ and $x\in\Sigma_{\lambda}$.  
\end{lemmaletter}

\begin{lemmaletter}\label{boundedness}
	Let $\vartheta$ satisfy the assumptions of Lemma~\ref{starting} with $\vartheta(x)=\vartheta(x_{\lambda})$ for some $x\in\Sigma_{\bar{\lambda}}$. Then, there exist $\varepsilon>0$ and $R>0$ satisfying\\
	\noindent{\rm (i)} $\vartheta_{x_n}(x)>0$ in $|x_n-\lambda_{0}|<\varepsilon$ and $|x|>R$;\\
	\noindent{\rm (ii)} $\vartheta(x)>\vartheta(x_{\lambda})$ in $x_n\geqslant\lambda_{0}+{\varepsilon}/{2}>\lambda$ and $|x|>R$ for all $x\in\Sigma_{\lambda}$, $\lambda\leqslant\lambda_{0}$ with $|\lambda_{0}-\bar{\lambda}|<c_{0}\varepsilon$, where $C_0>0$ is small and depends on $\bar{\lambda}$ and $v$.
\end{lemmaletter}

We also require a maximum principle for singular domains, which is the content of \cite[Lemma~2.1]{MR1333503}

\begin{propositionletter}\label{mp}
	Let $\Omega$ be a smooth bounded domain in $\mathbb{R}^n$ and $\mathcal{Z}$ be a compact set in $\mathbb{R}^n$ with $\capt(\mathcal{Z})=0$. Assume that $\vartheta(x),h(x)$ are nonnegative continuous functions in ${\Omega}\setminus \mathcal{Z}$ satisfying
	\begin{equation*}
	-\Delta \vartheta(x)+h(x)\leqslant0 \quad {\rm in} \quad \Omega\setminus \mathcal{Z}
	\end{equation*}
	in the distributional sense. Then,
	\begin{equation*}
	\vartheta(x)\geqslant\int_{E}G(x,y)h(y)\ud y+\int_{\partial E}{\partial_\nu}G(x,y)\vartheta(y)ds_y \quad {\rm in} \quad \Omega\setminus \mathcal{Z},
	\end{equation*}
	where $G(x,y)$ is the Green function of $-\Delta$ in $\Omega$ with Dirichlet boundary condition. In particular,
	\begin{equation*}
	\vartheta(x)\geqslant\inf_{\partial(\Omega\setminus\mathcal{Z})}\vartheta.
	\end{equation*}
\end{propositionletter}

\begin{proposition}\label{symmetry}
	Let $\mathcal{U}$ be a nonnegative singular solution to equation \eqref{oursystem}. Then, $|\mathcal{U}|$ is radially symmetric about the origin and monotonically decreasing.
\end{proposition}

\begin{proof}
	Since $\mathcal{U}$ is a singular solution, we may suppose without loss of generality that the origin is a non-removable singularity of $u_1$. Fixing $z\neq0$ a non-singular point of $\mathcal{U}$, that is, $\lim_{|x|\rightarrow z}|\mathcal{U}(x)|<\infty$, we perform the fourth order Kelvin transform with center at the $z$ and unitary radius,
	\begin{equation*}
	(u_i)_{z,1}(x)=|x|^{4-n}u_i\left(z+\frac{x}{|x|^2}\right) \quad \mbox{for} \quad i\in I.
	\end{equation*}
	Denoting $\widetilde{u}_{i}=(u_i)_{z,1}$, we observe that $\widetilde{u}_{1}$ is singular at zero and $z_{0}=-z/|z|^2$, whereas the others components are singular only at zero. Furthermore, using the conformal invariance of \eqref{oursystem}, we get
	\begin{equation*}
	\Delta^{2}\widetilde{u}_{i}=c(n)|\widetilde{\mathcal{U}}|^{2^{**}-2}\widetilde{u}_i \quad {\rm in} \quad \mathbb{R}^{n}\setminus\{0,z_{0}\}.
	\end{equation*}
	Let us set $\vartheta_i(x)=-\Delta \widetilde{u}_i(x)$, thus $\vartheta_i(x)=\mathcal{O}(|x|^{2-n})$ as $|x|\rightarrow\infty$. Using Lemma~\ref{harmasymp}, we have that $\vartheta_i$ has the following harmonic asymptotic expansion at infinity,
	\begin{equation*}
	\begin{cases}
	\vartheta_i(x)=a_{i0}|x|^{2-n}+a_{ij}x_j|x|^{-n}+\mathcal{O}(|x|^{-n})\\
	\partial_{x_j}\vartheta_i(x)=(2-n)a_{i0}x_j|x|^{-n}+\mathcal{O}(|x|^{-n})\\
	\partial_{x_{k}x_{j}}\vartheta_i(x)=\mathcal{O}(|x|^{-n}),
	\end{cases}
	\end{equation*}
	where $a_{i0}=\vartheta_{i}(z)$ and $a_{ij}=\partial_{y_j}\vartheta_{i}(z)$.
	
	Considering the axis defined by $0$ and $z$ as the reflection direction, we can suppose that this axis is orthogonal to the positive $x_n$ direction, that is, given the unit vector ${\bf e}_n=(0,0,\dots,1)$. 
	For $\lambda>0$, we consider the sets
	\begin{equation*}
	\Sigma_{\lambda}:=\{x\in\mathbb{R}^n : x_n>\lambda\} \quad \mbox{and} \quad T_{\lambda}:=\partial \Sigma_{\lambda},
	\end{equation*}
	and we define the reflection about the plane $T_{\lambda}$ by
	\begin{equation*}
	x=(x_1,\dots,x_{n-1},x_{n}) \mapsto x_{\lambda}=(x_1,\dots,x_{n-1},2\lambda-x_{n}).
	\end{equation*}
	Let us also introduce the notation $(w_{i})_{\lambda}(x)=\widetilde{u}_{i}(x)-(\widetilde{u}_{i})_{\lambda}(x)$, where $(\widetilde{u}_{i})_{\lambda}(x)=\widetilde{u}_{i}(x_{\lambda})$. Then, showing radial symmetry about the origin for singular solutions to \eqref{oursystem} is equivalent to prove the following
	\begin{equation}\label{cg}
	(w_{i})_{\lambda}\equiv0 \quad {\rm for} \quad \lambda=0.
	\end{equation}
	Subsequently, we divide the proof of \eqref{cg} into three claims.
	
	\noindent{\bf Claim 1:} There exists $\bar{\lambda}_0<0$ such that $(w_{i})_{\lambda}>0$ in $\Sigma_{\lambda}$ for all $\lambda<\bar{\lambda}_0$ and $i\in I$.
	
	\noindent In fact, notice that $(w_{i})_{\lambda}$ satisfies the following Navier problem
	\begin{equation}\label{df}
	\begin{cases}
	\Delta^{2}(w_{i})_{\lambda}=(b_{i})_{\lambda}(w_{i})_{\lambda}& \quad {\rm in} \quad \Sigma_{\lambda}\\
	\Delta (w_{i})_{\lambda}=(w_{i})_{\lambda}=0& \quad {\rm on} \quad T_{\lambda},
	\end{cases}
	\end{equation}
	where 
	\begin{equation*}
	(b_{i})_{\lambda}=\frac{c(n)|\widetilde{\mathcal{U}}_{\lambda}|^{2^{**}-2}(\widetilde{u}_i)_\lambda-c(n)|{\widetilde{\mathcal{U}}}|^{2^{**}-2}\widetilde{u}_i}{\widetilde{u}_{i}-(\widetilde{u}_{i})_{\lambda}}>0 \quad {\rm in} \quad \bar{\Sigma}_{\lambda}.
	\end{equation*}
	Then, as a consequence of Lemma \ref{starting}, there exist $\bar{\lambda}<0$ and $R>|z_0|+10$ such that 
	\begin{equation}\label{pb}
	\Delta (w_{i})_{\lambda}(x)=(\vartheta_{i})_{\lambda}(x)-\vartheta_{i}(x)<0 \quad \hbox{for} \quad x\in\Sigma_{\lambda}, \quad\lambda\leqslant\bar{\lambda} \quad \hbox{and} \quad |x|>R.
	\end{equation}
	In addition, by Proposition~\ref{mp} we can find $C>0$ satisfying
	\begin{equation}\label{pe}
	\vartheta_i(x)\geqslant C \quad \hbox{for} \quad x\in \bar{B}_{R}\setminus\{0,z_0\}.
	\end{equation}
	Since $v_i\rightarrow 0$ as $|x|\rightarrow\infty$, combining \eqref{pb} and \eqref{pe}, there exists $\bar{\lambda}_{0}<\bar{\lambda}$ such that 
	\begin{equation}\label{rn}
	\Delta (w_{i})_{\lambda}(x)=(\vartheta_{i})_{\lambda}(x)-\vartheta_{i}(x)<0 \quad \hbox{for} \quad x\in\Sigma_{\lambda} \quad \mbox{and} \quad \lambda\leqslant\bar{\lambda}_0.
	\end{equation}
	Using that $\lim_{|x|\rightarrow\infty}(w_{i})_{\lambda}(x)=0$, we can apply the strong maximum principle to conclude that $(w_{i})_{\lambda}(x)>0$ for all $\lambda\leqslant\bar{\lambda}_0$ and $i\in I$, which implies the proof of the claim.
	
	Now thanks to Claim 1, we can define the {\it critical sliding parameter} given by
	\begin{equation*}
	\lambda^{*}=\sup\{\bar{\lambda}>0 : \eqref{rn} \ \hbox{holds for} \ \lambda\geqslant\bar{\lambda}\}.
	\end{equation*}
	
	\noindent{\bf Claim 2:} $(w_{i})_{\lambda^{*}}\equiv0$ for all $i\in I$.
	
	\noindent Fix $i\in I$ and suppose by contradiction that $(w_{i})_{\lambda^{*}}(x_0)\neq0$ for some $x_0\in\Sigma_{\lambda^{*}}$. 
	By continuity, we have that $\Delta (w_{i})_{\lambda^{*}}\leqslant0$ in $\Sigma_{\lambda^{*}}$. Since $\lim_{|x|\rightarrow\infty}(w_{i})_{\lambda}(x)=0$, a strong maximum principles yields that $(w_{i})_{\lambda^{*}}>0$ in $\Sigma_{\lambda^{*}}$.
	Also, by \eqref{oursystem}, we get $\Delta^2 (w_{i})_{\lambda^{*}}=|\widetilde{\mathcal{U}}|^{2^{**}-2}\widetilde{u}_{i}-|\mathcal{U}_\lambda|^{2^{**}-2}(\widetilde{u}_{i})_{\lambda}(x)>0$. Hence, $\Delta (w_{i})_{\lambda^{*}}$ is subharmonic. 
	By employing again the strong maximum principle, we obtain that $\Delta (w_{i})_{\lambda}<0$. 
	In addition, by the definition of $\lambda^{*}$, there exists a sequence $\{\lambda_k\}_{k\in\mathbb{N}}$ such that, $\lambda_k\nearrow\lambda^{*}$ and $ \sup_{\Sigma_{\lambda_k}}\Delta (w_{i})_{\lambda_k}(x)>0$. 
	Observing that $\lim_{|x|\rightarrow\infty}\Delta (w_{i})_{\lambda_k}(x)=0$, we can find $x_k\in\Sigma_{\lambda_k}$ satisfying
	\begin{equation}\label{peru}
	\Delta (w_{i})_{\lambda_k}(x_k)=\sup_{\Sigma_{\lambda_k}}\Delta (w_{i})_{\lambda_k}(x).
	\end{equation}
	By Lemma~\ref{boundedness}, we observe that $\{x_k\}_{k\in\mathbb{N}}$ is bounded. Thus, up to subsequence, we may assume that $x_k\rightarrow x_0$. If $x_0\in\Sigma_{\lambda^{*}}$, passing to the limit in \eqref{peru}, we obtain $\Delta (w_{i})_{\lambda^{*}}(x_0)=0$, which is a contradiction with $\Delta (w_{i})_{\lambda^{*}}(x_0)\leqslant0$. If $x_0\in T_{\lambda^{*}}$ we have that $\nabla(\Delta (w_{i})_{\lambda^{*}}(x_0))=0$. This contradicts the Hopf boundary Lemma, because $\Delta (w_{i})_{\lambda^{*}}$ is negative and subharmonic in $\Sigma_{\lambda^{*}}$.
	
	\noindent{\bf Claim 3:} $\lambda^{*}=0.$
	
	\noindent Let us assume that the claim is not valid, that is, $\lambda^{*}<0$. 
	Then, for $\lambda=\lambda^{*}$, it holds $\Delta w_{\lambda^{*}}(x)<0$. Since $\lim_{|x|\rightarrow z_0}u_1(x)=\infty$, we observe that $\widetilde{u}_1$ cannot be invariant under the reflection $x_{\lambda^{*}}$. 
	Thus, using a strong maximum principle for \eqref{df}, we conclude 
	\begin{equation}\label{rj}
	\widetilde{u}_i(x)<u_i(x_{\lambda}) \quad {\rm for} \quad x\in\Sigma_{\lambda^{*}} \quad {\rm and} \quad x_{\lambda^{*}}\notin\{0,z_0\}.
	\end{equation}
	Notice that as a consequence of $\lambda^{*}<0$, we have that $\{0,z_0\}\notin T_{\lambda^{*}}$. Whence, applying the Hopf boundary Lemma, we get
	\begin{equation}\label{rs}
	\partial_{x_k}(\widetilde{u}_i(x_{\lambda^{*}})-\widetilde{u}_i(x))=-2\partial_{x_k}\widetilde{u}_i(x)>0.
	\end{equation}
	Now choose $\{\lambda_k\}_{k\in\mathbb{N}}$ such that $\lambda_k\nearrow\lambda^{*}$ as $k\rightarrow\infty$ and $x_k\in\Sigma_{\lambda_k}$ such that $\widetilde{u}_{1}({x_k}_{\lambda_k})<\widetilde{u}_1(x_k)$. Then, by Lemma~\ref{starting}, we obtain that $\{x_k\}_{j\in\mathbb{N}}$ is bounded. Whence, $x_k\rightarrow\bar{x}\in\bar{\Sigma}_{\lambda^{*}}$ with $\widetilde{u}_{1}(\bar{x}_{\lambda^{*}})\leqslant\widetilde{u}_1(\bar{x})$. By \eqref{rj} we know that $\bar{x}\in\partial\Sigma_{\lambda^{*}}$ and then $\partial_{x_k}\widetilde{u}_1(\bar{x})\geqslant0$, a contradiction with \eqref{rs}, which proves \eqref{cg}.
\end{proof}

As a direct consequence of Proposition~\ref{supersingular}, we show that singular solutions to \eqref{oursystem} are weakly positive. Again, this property is fundamental to define the quotient function $q_{ij}=u_i/u_j$.

\begin{proposition}\label{singweakpos}
	Let $\mathcal{U}$ be a nonnegative singular solution to \eqref{oursystem}. Then, $\mathcal{U}$ is weakly positive.
\end{proposition}

\begin{proof}
	It follows directly by Proposition~\ref{supersingular} and the strong maximum principle.
\end{proof} 

Later, we will prove that singular solutions are more than weakly positive; indeed, they are strongly positive (see Corollary~\ref{stronglypositive}). In this case, either $I_0=\emptyset$ or $I_+=\emptyset$.

\subsection{Fourth order Emden--Fowler system}
Since we already know that solutions are rotationally invariant, the cylindrical transformation converts \eqref{oursystem} into a fourth order ODE system with constant coefficients. More specifically, using Proposition~\ref{symmetry}, we eliminate the angular components in expression \eqref{Pdespherical}, arriving at
\begin{align}\label{vectfowlersystem}
\begin{cases}
v_i^{(4)}-K_2v_i^{(2)}+K_0v_i=c(n)|\mathcal{V}|^{2^{**}-2}v_i \quad {\rm in} \quad \mathbb{R} \quad {\rm for} \quad i\in I,\\
v_i(0)=a_i, \quad v_i^{(1)}(0)=0, \quad v_i^{(2)}(0)=b_i, \quad v_i^{(3)}(0)=0.
\end{cases}
\end{align}
where $a_i,b_i\in\mathbb{R}$ for all $i\in I$.

\subsection{Pohozaev invariant}
The Pohozaev invariant is a homological constant related to the existence and classification of solutions to a large class of PDEs. Its first appearance dates back to the classical paper of S. Pohozaev \cite{MR0192184}. 
After that, N. Korevaar et al. \cite{MR1666838} used this tool together with rescaling analysis to prove removable-singularity theorems for solutions to the singular Yamabe equation for flat background metrics setting. 
See also the related works \cite{MR2393072,MR4002167,MR3394387,MR4085120}. 
Let us also emphasize that the existence of the Pohozaev-type invariant is closely related to conservation laws for the Hamiltonian energy of the ODE system \eqref{vectfowlersystem}. 
In our fourth order vectorial setting, let us define an energy which is conserved in time for all $p$-map solutions $\mathcal{V}$ to \eqref{vectfowlersystem} \cite{MR3869387,MR4094467,MR4123335}. 

\begin{definition}
	For any $\mathcal{V}$ nonnegative solution to \eqref{vectfowlersystem}, let us consider its {\it Hamiltonian Energy} given by
	\begin{equation}\label{vectenergy}
	\mathcal{H}(t,\mathcal{V})=-\langle \mathcal{V}^{(3)}(t),\mathcal{V}^{(1)}(t)\rangle+\frac{1}{2}|\mathcal{V}^{(2)}(t)|^{2}+\frac{K_2}{2}|\mathcal{V}^{(1)}(t)|^{2}-\frac{K_0}{2}|\mathcal{V}(t)|^2+\hat{c}(n)|\mathcal{V}(t)|^{2^{**}},
	\end{equation}
	or more explicitly in components,
	\begin{align*}
	\mathcal{H}(t,\mathcal{V})&=-\left(\sum_{i=1}^{p}v_{i}^{(3)}(t)v_{i}^{(1)}(t)\right)+\frac{1}{2}\left(\sum_{i=1}^{p}\left((v_{i}^{(2)}(t)\right)^2\right)+\frac{K_2}{2}\left(\sum_{i=1}^{p}\left(v_{i}^{(1)}(t)\right)^2\right)&\\\nonumber
	&\quad -\frac{K_0}{2}\left(\sum_{i=1}^{p}v_{i}(t)^2\right)^2+\hat{c}(n)\left(\sum_{i=1}^{p}v_{i}(t)^2\right)^{2^{**}/2},&
	\end{align*}
	where $\hat{c}(n)=(2^{**})^{-1}c(n)$
\end{definition}
Let us remark that this quantity satisfies
\begin{equation}\label{conservation}
{\partial_t}\mathcal{H}(t,\mathcal{V})=0.
\end{equation}
In other words, the Hamiltonian energy is invariant on the variable $t$. In addition, we can integrate \eqref{vectenergy} over $\mathbb{S}_t^{n-1}$ to define another conserved quantity.

\begin{definition} 
	For any $\mathcal{V}$ nonnegative solution to \eqref{vectfowlersystem}, let us define its {\it cylindrical Pohozaev functional} by 
	\begin{align*}
	\mathcal{P}_{\rm cyl}(t,\mathcal{V})&=\displaystyle\int_{\mathbb{S}_t^{n-1}}\mathcal{H}(t,\mathcal{V})\ud\theta.&
	\end{align*}
	Here $\mathbb{S}_t^{n-1}=\{t\}\times\mathbb{S}^{n-1}$ is the cylindrical ball with volume element given by $\ud\theta=e^{-2t}\ud\sigma$, where $\ud\sigma_r$ is the volume element of the euclidean ball of radius $r>0$. 
\end{definition}

By definition, $\mathcal{P}_{\rm cyl}$ also does not depend on $t\in\mathbb{R}$. 
Then, let us consider the {\it cylindrical Pohozaev invariant} $\mathcal{P}_{\rm cyl}(\mathcal{V}):=\mathcal{P}_{\rm cyl}(t,\mathcal{V})$. 
Thus, by applying the inverse of cylindrical transformation, we can recover the classical {\it spherical Pohozaev functional} defined by $\mathcal{P}_{\rm sph}(r,\mathcal{U}):=\mathcal{P}_{\rm cyl}\circ\mathfrak{F}^{-1}\left(\mathcal{V}\right)$. 

\begin{remark}
	We are not providing the formula explicitly for the spherical Pohozaev, because it is too lengthy and is not required in the rest of this manuscript. 
	The cylindrical Pohozaev-invariant is enough to perform our methods. Indeed, fixing $\mathcal{H}(t,\mathcal{V})\equiv H$ and $\mathcal{P}_{\rm sph}(\mathcal{U})=P$, we have that $\omega_{n-1}H=P$. 
	In other words, the Hamiltonian energy $H$ and spherical Pohozaev invariant $ P $ have the same sign.
	For an expression of the Pohozaev invariant in the spherical case, we refer the reader to \cite[Proposition~3.3]{arXiv:1503.06412}.
\end{remark}

\begin{remark}
	There exists a natural relation between the derivatives of $\mathcal{P}_{\rm sph}$ and $\mathcal{H}$ respectively,
	\begin{equation*}
	{\partial_r}\mathcal{P}_{\rm sph}(r,\mathcal{U})=r{\partial_t}\mathcal{H}(t,\mathcal{V}).
	\end{equation*}
	Thus, for any solution $\mathcal{U}$, the value $\mathcal{P}_{\rm sph}(r,\mathcal{U})$ is also radially invariant. 
\end{remark}

Now it is convenient to introduce an important ingredient of our next results. For more details, see \cite[Proposition~4.1]{MR1991145}.

\begin{definition}
	For any $\mathcal{U}$ nonnegative solution to \eqref{oursystem}, let us define its {\it spherical Pohozaev invariant} given by $\mathcal{P}_{\rm sph}(r,\mathcal{U}):=\mathcal{P}_{\rm sph}(\mathcal{U})$, which is defined by
	\begin{equation*}
		\mathcal{P}_{\rm sph}(\mathcal{U})=\int_{\partial\mathcal{B}_r}B\left(r,x,\mathcal{U},\nabla \mathcal{U},\nabla^{2}\mathcal{U},\nabla^{3} \mathcal{U}\right) \ud\sigma_r,
	\end{equation*}
	where the integrand is given in vectorial notation by
	\begin{align*}
		B\left(r,x,\mathcal{U},\nabla \mathcal{U},\nabla^{2} \mathcal{U}, \nabla^{3} \mathcal{U}\right)&=\frac{2-n}{2}\left\langle\Delta \mathcal{U},\partial_{\nu} \mathcal{U}\right\rangle-\frac{r}{2}|\Delta \mathcal{U}|^{2}+\frac{n-4}{2} \left\langle\mathcal{U},\partial_{\nu} \Delta\mathcal{U} \right\rangle+\langle x,\nabla \mathcal{U}\rangle\partial_{\nu} \Delta\mathcal{U}\\
		&-\Delta \mathcal{U} \sum_{j=1}^n x_{j}\partial_{\nu} \mathcal{U}_{j}.
	\end{align*}
\end{definition}

\begin{remark} For easy reference, let us summarize the following facts:\\
	\noindent{\rm (i)} There exists a type of equivalence between the cylindrical and spherical Pohozaev invariants, $\mathcal{P}_{\rm sph}(\mathcal{U})=\omega_{n-1}\mathcal{P}_{\rm cyl}(\mathcal{V})$, where $\omega_{n-1}$ is the Lebesgue measure of the unit sphere in $\mathbb{R}^{n-1}$.\\ 
	\noindent{\rm (ii)} The Pohozaev invariant of the vectorial solutions are equal to the Pohozaev invariant in the scalar case, which can be defined in a similar way using the Hamiltonian energy associated to \eqref{scalarsystem}. 
	More precisely, we define $\mathcal{P}_{\rm sph}(u)=\mathcal{P}_{\rm cyl}(r^{\gamma}u)$, where
	\begin{equation*}
	\mathcal{P}_{\rm cyl}(v)=\displaystyle\int_{\mathbb{S}^{n-1}_t}\left[-v^{(3)}v^{(1)}+\frac{1}{2}|v^{(2)}|^{2}+\frac{K_2}{2}|v^{(1)}|^{2}-\frac{K_0}{2}|v|^{2}+\hat{c}(n)|v|^{2^{**}}\right]\ud\theta.
	\end{equation*}
	Hence, if the non-singular solution is ${\mathcal{U}}_{x_0,\mu}=\Lambda u_{x_0,\mu}$ for some $\Lambda\in\mathbb{S}^{p-1}_+$ and $u_{x_0,\mu}$ a spherical solution from Theorem~\ref{theoremA}, we obtain that $\mathcal{P}_{\rm sph}({\mathcal{U}}_{x_0,\mu})=\mathcal{P}_{\rm sph}(u_{x_0,\mu})=0$. Analogously, if the singular solution has the form $\mathcal{U}_{a,T}=\Lambda u_{a,T}$ for some $\Lambda\in\mathbb{S}^{p-1}_{+,*}$ and $u_{a,T}$ a Emden--Fowler solution from Theorem~\ref{theoremB}, we get that $\mathcal{P}_{\rm sph}({\mathcal{U}}_{a,T})=\mathcal{P}_{\rm sph}(u_{a,T})<0$.
\end{remark}

\subsection{ODE system analysis}
In this subsection, we perform an asymptotic analysis program due to Z. Chen and C. S. Lin \cite[Section~3]{MR3869387}. 
This analysis is based on the Pohozaev invariant sign, which combined with some results from \cite{MR4094467,MR3394387} determines whether a solution to \eqref{oursystem} has a removable or a non-removable singularity at the origin. Before studying how this invariant classifies solutions to \eqref{oursystem}, we need to set some background results concerning the asymptotic behavior for solutions to \eqref{vectfowlersystem} and their derivatives.

\begin{definition}
	For any $\mathcal{V}$ nonnegative solution to \eqref{vectfowlersystem}, let us define its {\it asymptotic set} by
	\begin{equation*}
	\mathcal{A}(\mathcal{V}):=\displaystyle\bigcup_{i=1}^{p}\mathcal{A}(v_i)  \subset[0,\infty],
	\quad {\rm where} \quad \mathcal{A}(v_i):=\left\{l\in [0,\infty] : \lim_{t\rightarrow\pm\infty}v_i(t)=l\right\}.
	\end{equation*}
	In other words, $\mathcal{A}(\mathcal{V})$ is the set of all possible limits at infinity of the component solutions $v_i$.
\end{definition}

The first of our lemmas states that the asymptotic set of $\mathcal{V}$ is quite simple, in the sense that it does not depend on $i\in I$, and coincides with the one in the scalar case. 

\begin{lemma}\label{asymptotics}
	Let $\mathcal{V}$ be a nonnegative solution to \eqref{vectfowlersystem}. 
	Suppose that for all $i\in I$ there exists $l_i\in[0,\infty]$ such that $\lim_{t\rightarrow\pm\infty}v_i(t)=l_i$. 
	Thus, $l_i\in\{0,l^{*}\}$, where $l^{*}=p^{-1}{K_0}^{\frac{n-4}{8}}$; in other terms, $\mathcal{A}(\mathcal{V})=\{0,l^{*}\}$. 
	Moreover, if $\mathcal{P}_{\rm cyl}(\mathcal{V})\geqslant0$,  then $l^{*}=0$.
\end{lemma}

\begin{proof}
	Here it is only necessary to consider the case $t\rightarrow\infty$ since when $t\rightarrow-\infty$, taking $\tau=-t$, and observing that $\widetilde{\mathcal{V}}(\tau):=\mathcal{V}(t)$ also satisfies \eqref{vectfowlersystem}, the result follows equally.
	
	Suppose by contradiction that the lemma does not hold. 
	Thus, for some fixed $i\in I$, one of the following two possibilities shall happen: either the asymptotic limit of $v_{i}$ is a finite constant $l_{i}>0$, which does not belong to the asymptotic set $\mathcal{A}$, or the limit blows-up, that is, $l_i=+\infty$. Subsequently, we consider these two cases separately:
	
	\noindent{\bf Case 1:} $l_{i}\in [0,\infty)\setminus\{0,l^{*}\}$. 
	
	\noindent By assumption, we have 
	\begin{equation}\label{bound1}
	\displaystyle\lim_{t\rightarrow\infty}\left(c(n)|\mathcal{V}|^{\frac{8}{n-4}}v_{i}(t)-K_0v_{i}(t)\right)=\kappa, \quad \hbox{where} \quad \kappa:=c(n)pl_i^{\frac{n+4}{n-4}}-K_0l_i\neq0,
	\end{equation}
	which implies
	\begin{equation}\label{bound2}
	c(n)|\mathcal{V}|^{\frac{8}{n-4}}v_{i}(t)-K_0v_{i}(t)=v_i^{(4)}(t)-K_2v_i^{(2)}(t).
	\end{equation}
	A combination of \eqref{bound1} and \eqref{bound2} implies that for any $\varepsilon>0$ there exists $T_i\gg1$ sufficiently large satisfying
	\begin{equation}\label{bound3}
	\kappa-\varepsilon<v_i^{(4)}(t)-K_2v_i^{(2)}(t)<\kappa+\varepsilon.
	\end{equation}
	Now, integrating \eqref{bound3}, we obtain
	\begin{equation*}
	\int_{T_i}^{t}(\kappa-\varepsilon)\ud\tau<\int_{T_i}^{t}\left[v_i^{(4)}(\tau)-K_2v_i^{(2)}(\tau)\right]\ud\tau<\int_{T_i}^{t}(\kappa+\varepsilon)\ud\tau,
	\end{equation*}
	which provides
	\begin{equation}\label{bound5}
	(\kappa-\varepsilon)(t-T_i)+C_1(T_i)<v_i^{(3)}(t)-K_2v_i^{(1)}(t)<(\kappa+\varepsilon)(t-T_i)+C_1(T_i),
	\end{equation}
	where $C_1(T_i)>0$ is a constant. Defining $\delta:=\sup_{t\geqslant T_i}|v_i(t)-v_i(T_i)|<\infty$, we obtain
	\begin{equation*}
	\left|\int_{T_i}^{t}K_2v_i^{(1)}(\tau)\ud\tau\right|\leqslant|K_2|\delta.
	\end{equation*}
	Hence, integrating $\eqref{bound5}$ provides
	\begin{equation}\label{bound10}
	\frac{(\kappa-\varepsilon)}{2}(t-T_i)^2+L(t)<v_i^{(2)}(t)<\frac{(\kappa+\varepsilon)}{2}(t-T_i)^2+R(t),
	\end{equation}
	where $L(t),R(t)\in \mathcal{O}(t^2)$, namely
	\begin{equation*}
	L(t)=C_1(T_i)(T_i-t)-|K_2|\delta+C_2(T_i) \quad {\rm and} \quad R(t)=C_1(T_i)(T_i-t)+|K_2|\delta+C_2(T_i).
	\end{equation*}
	Then, repeating the same integration procedure in \eqref{bound10}, we find
	\begin{equation}\label{bound11}
	\frac{(\kappa-\varepsilon)}{24}(t-T_i)^4+\mathcal{O}(t^4)<v_i(t)<\frac{(\kappa+\varepsilon)}{2}(t-T_i)^4+\mathcal{O}(t^4) \quad \mbox{as} \quad t\rightarrow\infty.
	\end{equation}
	Therefore, since $\kappa\neq0$ we can choose $0<\varepsilon\ll1$ sufficiently small such that $\kappa-\varepsilon$ and $\kappa+\varepsilon$ have the same sign. Finally, by passing to the limit as $t\rightarrow\infty$ on inequality \eqref{bound11}, we obtain that $v_i$ blows-up and $l_i=\infty$, which is contradiction.
	This concludes the proof of the claim.
	
	\noindent{\bf Case 2:} $l_{i}=\infty$. 
	
	\noindent This case is more delicate, and it requires a suitable choice of test functions from \cite{MR1879326}. 
	More precisely, let $\phi_0\in C^{\infty}({[0,\infty]})$ be a nonnegative function satisfying $\phi_0>0$ in $[0,2)$, 
	\begin{equation*}
	\phi_0(z)=
	\begin{cases}
	1,&\ {\rm for} \ 0\leqslant z\leqslant1,\\
	0,&\ {\rm for} \ z\geqslant2,
	\end{cases}
	\end{equation*}
	and for $j\in \{1,2,3,4\}$, let us fix the positive constants
	\begin{equation}\label{miti-pokh}
	M_j:=\int_{0}^{2}\frac{|\phi_0^{(j)}(z)|}{|\phi_0(z)|}\, \ud z.
	\end{equation}
	Using the contradiction assumption, we may assume that there exists $T_i>0$ such that for $t>T_i$, it follows
	\begin{equation}\label{blow1}
	v_i^{(4)}(t)-K_2v_i^{(2)}(t)=\hat{c}(n)|\mathcal{V}(t)|^{\frac{8}{n-4}}v_{i}(t)-K_0v_{i}(t)\geqslant v_i(t)^{\frac{n+4}{n-4}}-K_0v_{i}(t)\geqslant \frac{c(n)}{2}v_{i}(t)^{\frac{n+4}{n-4}}
	\end{equation}
	and
	\begin{equation}\label{blow2}
	v_i^{(3)}(t)-K_2v^{(1)}(t)=\frac{1}{2}\int_{T_i}^{t}v_{i}(\tau)^{\frac{n+4}{n-4}}\ud\tau+C_1(T_i).
	\end{equation}
	Besides, as a consequence of \eqref{blow2}, we can find $T_i^{*}>T_i$ satisfying $v_i^{(3)}(T_i^{*})-K_2v^{(1)}(T_i^{*}):=\upsilon>0$.
	Furthermore, since \eqref{vectfowlersystem} is autonomous, we may suppose without loss of generality that $T_i^{*}=0$.
	Then, multiplying inequality \eqref{blow1} by  $\phi(t)=\phi_0(\tau/t)$, and by integrating, we find
	\begin{equation*}
	\int_{0}^{T'}v_i^{(4)}(\tau)\phi(\tau)\ud\tau-K_2\int_{0}^{T'}v_i^{(2)}(\tau)\phi(\tau)\ud\tau\geqslant\frac{1}{2}\int_{0}^{T'}v_{i}(\tau)\ud\tau,
	\end{equation*}
	where $T'=2T$. Moreover, integration by parts combined with $\phi^{(j)}(T')=0$ for $j=0,1,2,3$ implies 
	\begin{equation}\label{blow4}
	\int_{0}^{T'}v_i(\tau)\phi^{(4)}(\tau)v_i(\tau)\ud\tau-K_2\int_{0}^{T'}v_i(\tau)\phi^{(2)}(\tau)\ud\tau\geqslant\frac{c(n)}{2}\int_{0}^{T'}v_{i}(\tau)^{\frac{n+4}{n-4}}\ud\tau+\upsilon.
	\end{equation}
	On the other hand, applying the Young inequality on the right-hand side of \eqref{blow4}, it follows
	\begin{equation}\label{blow5}
	v_i(\tau)|\phi^{(j)}(\tau)|=\varepsilon v_i^{\frac{n+4}{n-4}}(\tau)\phi(\tau)+C_{\varepsilon}\frac{|\phi^{(j)}(\tau)|^{\frac{n+4}{8}}}{\phi(\tau)^{\frac{n-4}{8}}}.
	\end{equation}
	Hence, combining \eqref{blow5} and \eqref{blow4}, we have that for $0<\varepsilon\ll1$ sufficiently small, it follows that there exists $\widetilde{C}_1>0$ satisfying
	\begin{equation*}
	\widetilde{C}_1\int_{0}^{T'}\left[\frac{|\phi^{(4)}(\tau)|^{\frac{n+4}{8}}}{\phi(\tau)^{\frac{n-4}{8}}}+\frac{|\phi^{(2)}(\tau)|^{\frac{n+4}{8}}}{\phi(\tau)^{\frac{n-4}{8}}}\right]\ud\tau\geqslant\frac{c(n)}{4}\int_{0}^{T'}v_{i}(\tau)^{\frac{n+4}{n-4}}\ud\tau+\upsilon.
	\end{equation*}
	Now by \eqref{miti-pokh}, one can find $\widetilde{C}_2>0$ such that
	\begin{equation}\label{blow7}
	\widetilde{C}_2\left(M_4T^{-\frac{n+2}{2}}-M_2T^{-\frac{n}{4}}\right)\geqslant\frac{c(n)}{4}\int_{0}^{T}v_{i}(\tau)^{\frac{n+4}{n-4}}\ud\tau.
	\end{equation}
	Therefore, passing to the limit in \eqref{blow7} the left-hand side converges, whereas the right-hand side blows-up; this is a contradiction. 
	
	For proving of the second part, let us notice that 
	\begin{equation*}
	\lim_{t\rightarrow\infty}\mathcal{P}_{\rm cyl}(t,\mathcal{V})=\omega_{n-1}\left(\frac{K_0}{2}|l^{*}|^{2}-\hat{c}(n)|l^{*}|^{\frac{2n}{n-4}}\right)\geqslant0,
	\end{equation*}
	which implies $l^{*}=0$ and $\mathcal{P}_{\rm cyl}(\mathcal{V})=0$.
\end{proof}

The next lemma shows that if a component solution to \eqref{oursystem} blows-up, then it shall be in finite time. In this fashion, we provide an accurate higher order asymptotic behavior for singular solutions $\mathcal{V}$ of \eqref{vectfowlersystem}, namely, $\bigcup_{j=1}^{\infty}\mathcal{A}\left(\mathcal{V}^{(j)}\right)=\{0\}$.

\begin{lemma}\label{blow-up}
	Let $\mathcal{V}$ be a nonnegative solution to \eqref{vectfowlersystem} such that $\lim_{t\rightarrow\pm\infty}v_i(t)\in \mathcal{A}$ for all $i\in I$. Then, for any $j\geqslant1$, we have that $\displaystyle\lim_{t\rightarrow\pm\infty}v_{i}^{(j)}(t)=0$.
\end{lemma}

\begin{proof}
	As before, we only consider the case $t\rightarrow\infty$. 
	Since $\mathcal{A}=\{0,l^{*}\}$ we must divide our approach into two cases:
	
	\noindent{\bf Case 1:} $\lim_{t\rightarrow\pm\infty}v_i(t)=0$.
	
	\noindent For each ordinary derivative case $j=1,2,3,4$, we construct one step. 
	When $j\geqslant5$, the proof follows directly from the previous cases, and it is omitted. 
	We start by $j=2$,
	
	\noindent{\bf Step 1:} $\mathcal{A}(v_i^{(2)})=0$.\\
	By assumption $v_i(t)<l^{*}$ for $t\gg1$ large, one has
	\begin{equation*}
	v_i^{(4)}-K_2v_i^{(2)}=\left(c(n)|\mathcal{V}|^{\frac{8}{n-4}}v_i-K_0v_i\right)<0.
	\end{equation*}
	Defining $B_i(t)=v_i^{(2)}(t)+K_0v_i(t)$, it holds that $B_i^{(2)}(t)<0$ for all $t\in\mathbb{R}$, and thus, $B_i$ is concave near infinity, which implies $\mathcal{A}(B_i)\neq\emptyset$. Hence, there exists $b_0^{*}\in [0,\infty]$ such that $b_0^{*}:=\lim_{t\rightarrow\infty}B_i(t)$ and $b_1^{*}:=\lim_{t\rightarrow\infty}v_i^{(2)}(t)$. 
	Supposing that $b_1^{*}\neq0$, there exist three possibilities: First, if we assume $b_1^{*}=\infty$, then we have that $\lim_{t\rightarrow\infty}v^{(1)}_i(t)=\infty$, which is contradiction with $\lim_{t\rightarrow\infty}v_i(t)=0$. Second, assuming $0<b_1^{*}<\infty$, it follows that $v_i^{(2)}(t)>{b_1^{*}t}/2$ for $t\gg1$ sufficiently large; thus $v_i^{(1)}(t)>{b_1^{*}t}/{4}$, which is also a contradiction with the hypothesis. Third, $b^{*}<0$, then using the same argument as before, we obtain that $v_i^{(1)}(t)\leqslant{b_1^{*}t}/{4}$, leading to the same contradiction.
	Therefore $b_1^{*}=0$, which concludes the proof.
	
	\noindent{\bf Step 2:} $\mathcal{A}(v_i^{(1)})=0$.
	
	\noindent Indeed, for $t\gg1$ large, there exists $\tau\in[t,t+1]$ satisfying $v_i(t+1)-v_i(t)=v_i^{(1)}(t)+\frac{1}{2}v_i^{(2)}(\tau)$, which, by
	taking the limit, and since  $\tau\rightarrow\infty$ if $t\rightarrow\infty$, one gets that $v_i(t+1)\rightarrow0$ and $v_i(t)\rightarrow0$, which provides $\lim_{\tau\rightarrow\infty}v_i^{(2)}(\xi)\rightarrow0$. 
	Consequently, one has that $v_i^{(1)}(t)\rightarrow0$.
	
	\noindent{\bf Step 3:} $\mathcal{A}(v_i^{(3)})=0$.
	
	\noindent Since $H_{i}$ is concave for large $t\gg1$ and $B_i(t)\rightarrow\infty$ as $t\rightarrow\infty$, we find $\lim_{t\rightarrow\infty}B_i^{(1)}(t)=0$. Consequently, $v^{(3)}_i(t)\rightarrow\infty$ as $t\rightarrow\infty$.
	
	\noindent{\bf Step 4:} $\mathcal{A}(v_i^{(4)})=0$.
	
	\noindent By equation \eqref{vectfowlersystem} and by Step 1, we observe that $v^{(4)}_i(t)\rightarrow\infty$ as $t\rightarrow\infty$. 
	
	As a combination of Step 1--4, we finish the proof of Case 1. 
	
	The second case has an additional difficulty. 
	Precisely, since $v_i(t)\rightarrow l^{*}$ as $t\rightarrow\infty$ for sufficiently large $t\gg1$, there exist two possibilities: either $v_i$ is eventually decreasing or $v_i$ is eventually increasing. 
	In both situations, the proofs are similar; thus, we only present the first one.
	
	\noindent{\bf Case 2:} $\lim_{t\rightarrow\infty}v_i(t)=l^{*}$.
	
	Here we proceed as before.
	
	\noindent{\bf Step 1:} $\mathcal{A}(v_i^{(2)})=0$.
	
	\noindent Since we are considering $v_i$ is eventually decreasing, there exists a large $T_i\gg1$ such that $v_i(t)>l^{*}$ for $t>T_i$ and we get that $v_i^{(4)}-K_2v_i^{(2)}=\left(c(n)|\mathcal{V}|^{\frac{8}{n-4}}v_i-K_0v_i\right)\geqslant0$. In this case, $B_i$ is convex for sufficiently large $t\gg1$. Hence, $\mathcal{A}(B_i)\neq\emptyset$ and there exists $b_0^{*}=\lim_{t\rightarrow\infty}B_i(t)$. Since $v_i(t)\rightarrow l^{*}$ as $t\rightarrow\infty$, we get that $\lim_{t\rightarrow\infty}v_i^{(2)}(t)=b_1^{*}$, where $b_1^{*}=b_0^{*}-K_2l^{*}$. 
	Now repeating the same procedure as before, we obtain that $b_1^{*}=0$ and thus
	$\lim_{t\rightarrow\infty}B_i(t)=K_2l^{*}$, which yields $\mathcal{A}(v_i^{(2)})=0$. 
	
	The remaining steps of the proof follow similarly to Claim 1, and so the proof of the lemma is finished.
\end{proof}

Before we continue our analysis, it is essential to show that any solution to \eqref{vectfowlersystem} is bounded, which is the content of the following lemma. 

\begin{lemma}\label{fbounded}
	Let $\mathcal{V}$ be a nonnegative solution to \eqref{vectfowlersystem}. 
	Then, $v_i(t)<l^{*}$ for all $i\in I$. In particular, $|\mathcal{V}|$ is bounded.
\end{lemma}

\begin{proof} For $i\in I$, let us define the set $Z_i=\left\{t\geqslant0 : v_i^{(1)}(t)=0\right\}$. We divide the proof of the lemma into two cases:
	
	\noindent{\bf Case 1:} $Z_i$ is bounded.
	
	\noindent In this case we have that $v_i$ is monotone for large $t\gg1$ and $\mathcal{A}(v_i)\neq\emptyset$. Therefore, using Lemma~\ref{asymptotics} we obtain that $v_i$ bounded by $l^{*}$ for $t\gg1$ sufficiently large.
	
	\noindent{\bf Case 2:} $Z_i$ is unbounded.
	
	\noindent Fixing $H>0$, we define $F(\tau)=\hat{c}(n)|\tau|^{2^{**}}-\frac{1}{2}|\tau|^2$,
	which satisfies $\lim_{\tau\rightarrow\infty}F(\tau)=\infty$. 
	Therefore, there exists $R_i>|v_i(0)|$ such that $F(\tau)>H$ for $\tau>R_i$.
	
	\noindent{\bf Claim 1:} $|v_i|<R_i$ on $[0,\infty)$.
	
	\noindent Supposing by contradiction that $M_{R_i}=\{t\geqslant0 : |v_i(t)|\geqslant R_i\}$ is non-empty, we can define $t_i^{*}=\inf_{M_{R_i}}v_i$, which is strictly positive by the choice of $R_i$. 
	Thus, we obtain that $v_i(t_i^{*})=R_i$ and also $v_i^{(1)}(t_i^{*})\geqslant0$. In addition, since $Z_i$ is unbounded, we have that $Z_{i}\cap[t_i^{*},\infty)\neq\emptyset$. 
	Therefore, considering $T_i^{*}=\inf_{Z_{i}\cap[t_i^{*},\infty)}v_i$. 
	Hence, a combination of $v^{(1)}(T_i^{*})=0$ and Proposition~\ref{singregularity} implies that $v_i^{(1)}(t)\geqslant0$ for all $t\in[t_i^{*},T_i^{*}]$. 
	Eventually, we conclude that $v_i(T^{*})>R_i$ and $\mathcal{H}(T_i^{*},\mathcal{V})=\frac{1}{2}|\mathcal{V}^{(2)}(T_i^{*})|^{2}+F(|\mathcal{V}(T_i^{*})|)>H$, which is a contradiction with \eqref{conservation}. 
	To complete the proof lemma, one can check that $R_i=l^{*}$ for all $i\in I$.
\end{proof}

\begin{lemma}\label{signal}
	Let $\mathcal{V}$ be a nonnegative solution to \eqref{vectfowlersystem}. 
	Then, it follows that $v^{(1)}_i(t)<\gamma v_i(t)$
	for all $i\in I$ and $t\in\mathbb{R}$, where we recall that $\gamma=\frac{n-4}{2}$ is the Fowler rescaling exponent. 
\end{lemma}

\begin{proof}
	Let us define
	\begin{equation*}
	\widetilde{\gamma}=\sqrt{\frac{K_2}{2}-\sqrt{\frac{K_2^2}{4}-K_0}}.
	\end{equation*}
	Then, by a direct computation, we get that $\widetilde{\gamma}=\gamma$. 
	Setting
	\begin{equation*}
	\lambda_1=\frac{K_2}{2}-\sqrt{\frac{K_2^2}{4}-K_0} \quad {\rm and} \quad \lambda_2=\frac{K_2}{2}+\sqrt{\frac{K_2^2}{4}-K_0},
	\end{equation*}
	we have that $   \lambda_1+\lambda_2=K_2$ and $\lambda_1\lambda_2=K_0$.
	Defining the auxiliary function $\phi_i(t)=v^{(2)}_i-\lambda_2v_i(t)$, we observe that $\phi_i^{(2)}-\lambda_2\phi_i=|\mathcal{V}|^{\frac{8}{n-4}}v_i$ and $-\phi_i^{(2)}+\lambda_2\phi_i\leqslant0$. 
	Hence, since $\mathcal{V}$ is a nonnegative solution to \eqref{fowler4order} by the strong maximum principle, we get that $\phi_i<0$, which implies that $w_i=v_i^{(1)}/v_i$ satisfies 
	\begin{equation}\label{sig1}
	w_i^{(1)}=-w_i+\lambda_1+\frac{\phi_i}{v_i} \quad {\rm and} \quad \frac{v_i^{(2)}}{v_i}=\lambda_1+\frac{\phi_i}{v_i}.
	\end{equation}
	Moreover, by Lemma~\ref{asymptotics}, there exists $t_0\in\mathbb{R}$ such that $v^{(1)}_i(t_0)=0$, which provides $w_i(t_0)=0$. 
	Setting $M:=\left\{t>t_0 : w_i(t)\geqslant\sqrt{\lambda_1}\right\}$, the proof of the lemma is reduced to the next claim.
	
	\noindent{\bf Claim 1:} $M=\emptyset$.
	
	\noindent Indeed, supposing the claim is not true, we set $t_1=\inf M$. 
	Notice that $t_1>t_0$, 
	$w_i^{(1)}(t_1)\geqslant0$ and $w_i(t_1)=\sqrt{\lambda_1}$. 
	On the other hand, by \eqref{sig1}, we obtain that $w_i^{(1)}(t_1)=\frac{\phi_i(t_1)}{v_i(t_1)}<0$, a contradiction. This finishes the proof of the claim.
\end{proof}

As an application of Lemma~\ref{signal}, we complete the proof of Proposition \ref{symmetry}, which states that any component of $\mathcal{U}$ is radially monotonically decreasing.

\begin{corollary}
	Let $\mathcal{U}$ be a nonnegative singular solution to \eqref{oursystem}. 
	Then, ${\partial_r u_i}(r)<0$ for all $r>0$ and $i\in I_{+}$.
\end{corollary}

\begin{proof}
	By a direct computation, we have that ${\partial_r u_i}(r)=-r^{\gamma-1}\left[v_i^{(1)}(t)-\gamma v_i(t)\right]$. Then, the proof of the corollary is a consequence of Lemma~\ref{signal}.
\end{proof}
\subsection{Removable singularity classification}
After establishing the previous lemmas concerning the asymptotic behavior of global solutions to the ODE system \eqref{vectfowlersystem}, we can prove the main results of the section, namely, the removable-singularity classification and the non-existence of semi-singular solutions to \eqref{oursystem}. 
These results will be employed in the proof of Theorem~\ref{theorem2}. More precisely, we show that the Pohozaev invariant of any solution is always nonpositive, and it is zero, if, and only if, the origin is a non-removable, otherwise, for singular solutions to \eqref{oursystem} this invariant is always negative.

To show the removable singularity theorem, we need to define some auxiliary functions. For $i\in I$, let us set $\varphi_i:\mathbb{R}\rightarrow\mathbb{R}$ given by
\begin{equation*}
\varphi_i(t)=v_{i}^{(3)}(t)v_{i}^{(1)}(t)-\frac{1}{2}|v_{i}^{(2)}(t)|^2-\frac{K_2}{2}|v_{i}^{(1)}(t)|^2+\frac{K_0}{2}|v_{i}(t)|^2-\hat{c}(n)|v_{i}(t)|^{2^{**}}.
\end{equation*}

\begin{remark}\label{auxfunction}
	By Lemma~\ref{blow-up}, we observe that
	\begin{equation*}
	\varphi^{(1)}_i(t)=c(n)\left(|\mathcal{V}(t)|^{2^{**}-2}-|v_i(t)|^{2^{**}-2}\right)v_i(t)v^{(1)}_i(t).
	\end{equation*}
	Since $|\mathcal{V}|\geqslant|v_i|$, we have that $\sgn (\varphi^{(1)}_i)=\sgn(v^{(1)}_i)$. 
	In other terms, the monotonicity of $\varphi_i$ is the same of component function $v_i$. 
	Moreover, it holds that
	$\sum_{i=1}^{p}\varphi_i(t)=-H$.
\end{remark}

\begin{proposition}\label{Pohozaev}
	Let $\mathcal{U}$ be a nonnegative singular solution to \eqref{oursystem}. Then, $\mathcal{P}_{\rm sph}(\mathcal{U})\leqslant0$ and $\mathcal{P}_{\rm sph}(\mathcal{U})=0$, if, and only if, $\mathcal{U}\in C^{4,\zeta}(\mathbb{R}^{n},\mathbb{R}^p)$, for some $\zeta\in(0,1)$.
\end{proposition}

\begin{proof} Let us divide the proof into two claims as follows. The first one is concerned with the sign of the Pohozaev invariant. Namely, we show it is always nonpositive.
	
	\noindent{\bf Claim 1:} If $\mathcal{P}_{\rm sph}(\mathcal{U})\geqslant0$, then $\mathcal{P}_{\rm sph}(\mathcal{U})=0$.
	
	\noindent Indeed, let us define the sum function $v_{\Sigma}:\mathbb{R}\rightarrow\mathbb{R}$ given by $v_{\Sigma}(t)=\sum_{i=1}^pv_i(t)$. 
	Hence, by Lemma~\ref{asymptotics}, for any $v_i$ there exists a sufficient large $\hat{t}_i\gg1$ such that $v^{(1)}_i(\hat{t}_i)=0$. Furthermore, by Lemma~\ref{blow-up} for any $i\in I$, we can find a sufficiently large $t_i\geqslant \hat{t}_i\gg1$ such that $v^{(1)}_i(t_i)<0$ for all $t>t_i$. 
	Then, choosing $t_*>\max_{i\in I}\{t_i\}$, we have that $v^{(1)}_{\Sigma}(t)<0$ for $t>t_*$, which implies $\lim_{t\rightarrow\infty}v_{i}(t)=0$. Consequently, by Lemma~\ref{signal}, we conclude that $\mathcal{P}_{\rm sph}(\mathcal{U})=0$.
	
	In the next claim, we use some arguments from \cite[Lemma~2.4]{MR4094467} to show that solutions with zero Pohozaev invariant have a removable singularity at the origin.
	
	\noindent{\bf Claim 2:} If $\mathcal{P}_{\rm sph}(\mathcal{U})=0$, then $\mathcal{U}\in C^{4,\zeta}(\mathbb{R}^{n},\mathbb{R}^p)$, for some $\zeta\in(0,1)$.
	
	\noindent In fact, note that $v_{\Sigma}$ satisfies
	\begin{equation}\label{hurry}
	v^{(4)}_{\Sigma}-K_2v^{(2)}_{\Sigma}+K_0v_{\Sigma}=c(n)|\mathcal{V}|^{2^{**}-2}v_{\Sigma}.
	\end{equation}
	Setting $\widetilde{f}(\mathcal{V})=c(n)|\mathcal{V}|^{2^{**}-2}v_{\Sigma}$, since $v_i(t)\rightarrow0$ as $t\rightarrow\pm\infty$, it follows that $\lim_{t\rightarrow\infty}\widetilde{f}(\mathcal{V}(t))=0$. Then, we define $\tau=-t$ and $\widetilde{v}_{\Sigma}(\tau)=v_{\Sigma}(t)$, which implies that $\widetilde{v}_{\Sigma}$ also satisfies \eqref{hurry}. Moreover, $\lim_{t\rightarrow-\infty}v_{\Sigma}(t)=\lim_{\tau\rightarrow\infty}\widetilde{v}_{\Sigma}(\tau)=0$
	and also
	\begin{equation}\label{pokho1}
	\lim_{\tau\rightarrow\infty}\widetilde{f}(\widetilde{\mathcal{V}}(\tau))=0.
	\end{equation}
	Consequently, by ODE theory (see for instance \cite{MR0171038,MR3158844}), we can find sufficiently large $T\gg1$ satisfying
	\begin{align*}
	\widetilde{v}_{\Sigma}(\tau)&=A_1e^{\lambda_1t}+A_2e^{\lambda_2t}+A_3e^{\lambda_3t}+A_4e^{\lambda_4t}&\\
	&\quad+B_1\int_{T}^{t}e^{\lambda_1(\tau-t)}\widetilde{f}(\widetilde{\mathcal{V}}(t))\ud t+B_2\int_{T}^{t}e^{\lambda_2(\tau-t)}\widetilde{f}(\widetilde{\mathcal{V}}(t))\ud t&\\
	&\quad-B_3\int_{\tau}^{\infty}e^{\lambda_3(\tau-t)}\widetilde{f}(\widetilde{\mathcal{V}}(t))\ud t-B_4\int_{\tau}^{\infty}e^{\lambda_4(\tau-t)}\widetilde{f}(\widetilde{\mathcal{V}}(t))\ud t,&
	\end{align*}
	where $A_1, A_2, A_3, A_4$ are constants depending on $T$, $B_1, B_2, B_3,B_4$ are constants not depending on $T$, and 
	\begin{equation*}
	\lambda_1=-\frac{n}{2}, \ \lambda_2=-\frac{n-4}{2}, \ \lambda_3=\frac{n}{2} \ {\rm and} \ \lambda_4=\frac{n-4}{2}
	\end{equation*} 
	are the solutions to the characteristic equation $\lambda^{4}-K_2\lambda^2+K_0\lambda=0$. In addition, by \eqref{pokho1} we obtain that $A_3=A_4=0$. Hence, we use the same ideas in \cite[Theorem~3.1]{MR3315584} to arrive at
	\begin{equation*}
	\widetilde{v}_{\Sigma}(\tau)=\mathcal{O}(e^{-\frac{n-4}{2}\tau}) \quad {\rm as} \quad \tau\rightarrow\infty
	\quad {\rm or} \quad {v}_{\Sigma}(t)=\mathcal{O}(e^{\frac{n-4}{2}t}) \quad {\rm as} \quad t\rightarrow-\infty.
	\end{equation*}
	Eventually, undoing the cylindrical transformation, we have that ${u}_{\Sigma}(r)=\mathcal{O}(1)$ as $r\rightarrow0$, which finishes the proof of the claim. 
	
	Therefore, using the last claim, we get $u_{\Sigma}$ is uniformly bounded, which implies $u_i\in C^{0}(\mathbb{R}^n)$ for all $i\in I$. 
	Finally, standard elliptic regularity theory provides that $\mathcal{U}\in C^{4,\zeta}(\mathbb{R}^n,\mathbb{R}^p)$ for some $\zeta\in(0,1)$ and for all $i\in I$; this concludes the proof of the proposition.
\end{proof}

\begin{proposition}\label{fully-singular}
Let $\mathcal{U}$ be a nonnegative singular  solution to \eqref{oursystem}. If $\mathcal{P}_{\rm sph}(\mathcal{U})<0$, then $\mathcal{U}$ is fully-singular.
\end{proposition}

\begin{proof} Suppose by contradiction $\mathcal{U}$ is semi-singular, that is, there exists some $i_0\in I\setminus I_{\infty}$. We may suppose without loss of generality $\{i_0\}=I\setminus I_{\infty}$, which yields
\begin{equation}\label{semi-singular}
\displaystyle\lim_{\substack{r\rightarrow 0\\ i\neq i_0}}u_i(r)=\infty \quad {\rm and} \quad \displaystyle\liminf_{r\rightarrow 0}u_{i_0}(r)=C_{i_0}<\infty.
\end{equation}

\noindent{\bf Claim 1:} $\lim_{t\rightarrow\infty}v_{i_0}(t)=\infty$.

\noindent Indeed, using Lemma~\ref{signal}, we have that $\gamma^{-1}|v_{i_0}^{(1)}(t)|\leqslant v_{i_0}(t)\leqslant C_{i}e^{-\gamma t}$ for all $i\in I\setminus\{i_0\}$, which provides $\varphi_{i_0}(t)\rightarrow0$ as $t\rightarrow\infty$. 
Hence, since $P<0$, we get that $H<0$, which combined with Remark~\ref{auxfunction} yields $\sum_{\substack{{i=1}\\ i\neq i_0}}^{p}\varphi_i(t)=-H$. 
Let us divide the rest of the proof into two steps:

\noindent{\bf Step 1:} For each $i\in I\setminus\{i_0\}$, there exists $C_i>0$ such that $u_i(r)\geqslant C_ir^{-\gamma}$ for all $r\in(0,1]$.

\noindent First, it is equivalent to $\inf_{t\geqslant0}v_i(t)\geqslant C_i$ in cylindrical coordinates. 
Assume by contradiction that it does not hold. 
Then, there exists $\{t_k\}_{k\in\mathbb{N}}\subset(0,\infty)$ such that $t_k\rightarrow\infty$ and $v_i(t_k)\rightarrow0$ as $k\rightarrow\infty$. 
Moreover, using Lemma~\ref{signal} for all $i\in I$ one obtains $0\leqslant\gamma^{-1}|v_{i}^{(1)}(t_k)|\leqslant v_{i}(t_k)\rightarrow0$, which yields that $\varphi_i(t_k)\rightarrow0$. 
This is is a contradiction, and the proof of Step 1 is finished.

\noindent{\bf Step 2:} There exists $\varrho\in C^{\infty}(\mathbb{R}\setminus\{0\})$ such that $\lim_{r\rightarrow0}\varrho(r)=\infty$ and 
\begin{equation*}
u_{i_0}(r)\geqslant \varrho(r) \quad {\rm for all} \quad r\in(0,1].
\end{equation*}

\noindent First, it is easy to check that there exists $C_0>0$ such that $u_{i_0}(r)\geqslant C_0$ for all $r\in(0,1]$. 
Second, writing the Laplacian in spherical coordinates, we have
\begin{equation*}
r^{1-n}{\partial_r}\left[r^{n-1}{\partial_r \Delta u_{i_0}}(r)\right]=c(n)|\mathcal{U}|^{2^{**}-2}u_{i_0}.
\end{equation*}
Now use the estimates in Step 1 to obtain,
\begin{equation*}
{\partial_r}\left[r^{n-1}\partial_r \Delta u_{i_0}(r)\right]\geqslant c_0r^{n-5},
\end{equation*}
which, by integrating, implies
\begin{equation*}
r^{n-1}\partial_r \Delta u_{i_0}(r)\geqslant c_1r^{n-4}+c_2.
\end{equation*}
By proceeding as before, we get
\begin{equation*}
\Delta u_{i_0}(r)\geqslant c_1r^{-3}+c_2r^{1-n}.
\end{equation*}
Therefore, by repeating the same procedure, we can find $c_1,c_2,c_3,c_4\in\mathbb{R}$ satisfying
\begin{equation*}
u_{i_0}(r)\geqslant c_1r^{-1}+c_2r^{1-n}+c_3r^{-n}+c_4,
\end{equation*}
which concludes the proof of Step 2.

Eventually, passing to the limit as $r\rightarrow0$ in Step 2, we obtain that $u_{i_0}$ blows-up at the origin.
Hence, Claim 1 holds, which is a contradiction with \eqref{semi-singular}. 
Therefore, semi-singular solutions cannot exist, and the proposition is proved.
\end{proof}

\begin{remark}
We highlight that Proposition~\ref{fully-singular} is a surprising result since, for the type of singular system considered in \cite{MR3394387}, it is only possible to obtain the same conclusion with some restriction on the dimension. This better behavior is due to the symmetries that the Gross--Pitaevskii nonlinearity enjoys.
\end{remark}

\begin{corollary}\label{stronglypositive}
Let $\mathcal{U}$ be a nonnegative singular solution to \eqref{oursystem}. Then, $\mathcal{U}$ is strongly positive.
\end{corollary}

\begin{proof} We already know by Proposition~\ref{singweakpos} that $\mathcal{U}$ is weakly positive. Suppose by contradiction that $\mathcal{U}$ is not strongly positive. Then, there exists some $i_0\in I_0$, that is, $u_{i_0}\equiv0$ and so non-singular at the origin. Thus, by Proposition~\ref{fully-singular} all the other components must also be non-singular at the origin. Therefore, $I_{\infty}=\emptyset$, which is contradiction since $\mathcal{U}$ is a singular solution to \eqref{oursystem}.
\end{proof}

\subsection{Proof of Theorem~\ref{theorem2}}
Finally, we have conditions to connect the information we have obtained to prove our classification result. 
Our idea is to apply the analysis of the Pohozaev invariant and ODE methods together with Theorem~\ref{theorem1}, and Propositions~\ref{Pohozaev} and \ref{fully-singular}, which can be summarized as follows
\begin{theoremtio}\label{theorem2'}
{\it 
	Let $\mathcal{U}$ be a nonnegative solution to \eqref{oursystem}. 
	There exist only two possibilities for the sign of the Pohozaev invariant:\\
	\noindent {\rm (i)} If $\mathcal{P}_{\rm sph}(\mathcal{U})=0$, then $\mathcal{U}=\Lambda u_{x_0,\mu}$, where $u_{x_0,\mu}$ is given by \eqref{sphericalfunctions} $($spherical solution$)$;\\
	\noindent {\rm (ii)} If $\mathcal{P}_{\rm sph}(\mathcal{U})<0$, then $\mathcal{U}=\Lambda^* u_{a,T}$, where $u_{a,T}$ is given by \eqref{emden-folwersolution} $($Emden--Fowler solution$)$.
}
\end{theoremtio}

\begin{proof}
	(i) It follows directly by Proposition~\ref{Pohozaev} and Theorem~\ref{theorem1}.
	
	\noindent (ii) First, by Corollary~\ref{stronglypositive}, it follows that $I_{+}=I_{\infty}=I$, which makes the quotient functions $q_{ij}=v_i/v_j$ well-defined for all $i,j\in I$.
	Moreover, we show that they are constants. 
	Notice that $v_i$ and $v_j$ satisfy,
	\begin{equation*}
		\begin{cases}
			v_i^{(4)}-K_2v_i^{(2)}+K_0v_i=c(n)|\mathcal{V}|^{2^{**}-2}v_i\\
			v_j^{(4)}-K_2v_j^{(2)}+K_0v_j=c(n)|\mathcal{V}|^{2^{**}-2}v_j,
		\end{cases}
	\end{equation*}
	which provides
	\begin{equation}\label{argentina}
		\left(v_i^{(4)}v_j-v_j^{(4)}v_i\right)=-K_2\left(v_i^{(2)}v_j-v_iv_j^{(2)}\right).
	\end{equation}
	Furthermore, a standard computation yields
	\begin{equation*}
		q_{ij}^{(4)}=\frac{v_i^{(4)}v_j-v_iv_j^{(4)}}{v_j^2}-4{v_j^{(1)}}{v_j^{-1}}q^{(3)}_{ij}-6{v_j^{(2)}}{v_j^{-1}}q^{(2)}_{ij}-4{v_j^{(3)}}{v_j^{-1}}q^{(1)}_{ij},
	\end{equation*}
	which, combined with \eqref{argentina}, implies that the quotient satisfy the following fourth order homogeneous Cauchy problem,
	\begin{align*}
		\begin{cases}
			q_{ij}^{(4)}+4{v_j^{(1)}}{v_j^{-1}}q^{(3)}_{ij}+6{v_j^{(2)}}{v_j^{-1}}q^{(2)}_{ij}+(4{v_j^{(3)}}{v_j^{-1}}+K_2)q^{(1)}_{ij}=0 \quad {\rm in} \quad \mathbb{R}\\
			q_{ij}(0)=a_i/a_j, \quad q_{ij}^{(1)}(0)=q_{ij}^{(2)}(0)=q_{ij}^{(3)}(0)=0.
		\end{cases}
	\end{align*}
	Hence, using Lemma~\ref{blow-up} the Picard--Lindel\"{o}f uniqueness theorem, it follows that $q_{ij}\equiv a_i/a_j$.
	Thus, by the same argument at the end of the proof of Theorem~\ref{hebey-druet}, one can find $\Lambda^*\in\mathbb{S}^{p-1}_{+,*}$ such that $\mathcal{V}(t)=\Lambda^* v_{a,T}(t)$, where $v_{a,T}$ is given by \eqref{emden-folwersolution}.
	By undoing the cylindrical transformation, the theorem is proved. 
\end{proof}

As an application of Theorem~\ref{theorem2'}\textcolor{blue}{'} and Lemma~\ref{fbounded}, we provide a sharp global estimate for the blow-up rate near the origin for singular solutions to \eqref{oursystem}.

\begin{corollary}
	Let $\mathcal{U}$ be a nonnegative singular solution to \eqref{oursystem}. Then, there exist $C_1,C_2>0$ such that 
	\begin{equation*}
		C_1|x|^{-\gamma}\leqslant|\mathcal{U}(x)|\leqslant C_2|x|^{-\gamma} \quad {\rm for \ all} \quad  x\in\mathbb{R}^{n}\setminus\{0\}.
	\end{equation*}
	In other terms, $|\mathcal{U}(x)|=\mathcal{O}(|x|^{-\gamma})$ as $x\rightarrow0$,
	where $\gamma=\frac{n-4}{2}$.
\end{corollary}

\appendix

\section{Some basic proofs}\label{appendixA}
In this appendix, we present the proofs of some elementary results that we have used in our text.
 
First, we show that the Laplacian of a function is invariant under spherical averaging.

\begin{lemma}
If $u\in C^2(\mathbb{R}^n)$, then $\overline{\Delta u}=\Delta\overline{u}$, where $\overline{u}(r):=|\partial B_r|^{-1}\int_{\partial B_r}u(y)\ud\sigma_r(y)$.
\end{lemma}

\begin{proof}
	Using spherical coordinates $r=|x|$ and $\sigma=x|x|^{-1}$, we have 
	\begin{equation}\label{Jm3}
	\Delta u=\partial^{(2)}_ru+{(n-1)}{r^{-1}}\partial^{(1)}_ru+{r^{-2}}\Delta_{\sigma}u,
	\end{equation}
	which, from taking the spherical average, implies 
	\begin{align}\label{Jm1}
	\overline{\Delta u}
	&=\overline{\partial^{(2)}_ru+{(n-1)}{r^{-1}}\partial^{(1)}_ru+{r^{-2}}\Delta_{\sigma}u}=\overline{\partial^{(2)}_ru}+{(n-1)}{r^{-1}}\overline{\partial^{(1)}_ru}+\overline{\Delta_{\sigma}u}.&
	\end{align}
	Now using a change of variables, we can rewrite 
	\begin{equation}\label{Jm5}
	\overline{u}(r):=|\partial B_1|^{-1}\int_{\partial B_1}u(x+ry)\ud\sigma_r(y),
	\end{equation}
	which, by a differentiation under the integration sign, implies $\overline{\partial^{(1)}_ru}=\partial^{(1)}_r\overline{u}$ and $\overline{\partial^{(2)}_ru}=\partial^{(2)}_r\overline{u}$. Moreover, since the angular part is radially invariant, we get that $\overline{\Delta_{\sigma}u}=\Delta_{\sigma}\overline{u}$, which together with \eqref{Jm1}, provides 
	\begin{align}\label{Jm2}
	\overline{\Delta u}
	&=\partial^{(2)}_r\overline{u}+{(n-1)}{r^{-1}}\partial^{(1)}_r\overline{u}+{r^{-2}}\Delta_{\sigma}\overline{u}&
	\end{align}
	
	On the other hand, one needs to prove that $\eqref{Jm3}$ is invariant under the action of the spherical average, that is,
	\begin{equation}\label{Jm4}
	\Delta \overline{u}=\partial^{(2)}_r\overline{u}+{(n-1)}{r^{-1}}\partial^{(1)}_r\overline{u}.
	\end{equation}
	In fact, using the divergence theorem, let us compute
	\begin{align*}
	{\partial^{(1)}_r}\overline{u}(r)&=|\partial B_1|^{-1}\int_{\partial B_1} \frac{1}{r}\sum_{i=1}^{n}{u}_{y_i}(x+ry)y_i\ud\sigma_r(y)\\
	&=|\partial B_1|^{-1}\int_{B_1} \frac{1}{r} \sum_{i=1}^{n}{u}_{y_iy_i}(x+ry)\ud\sigma_r(y)&\\
	&=|\partial B_1|^{-1}\int_{B_1} \frac{1}{r}\sum_{i=1}^{n}{u}_{x_ix_i}(x+ry)\ud\sigma_r(y)&\\
	&=|\partial B_1|^{-1}\int_{B_1} \frac{1}{r}\Delta{u}(x+ry)\ud\sigma_r(y).&
	\end{align*}
	Now using the layer cake integration and change of variables, we get 
	\begin{align*}
	{\partial^{(1)}_r}\overline{u}(r)
	&=|\partial B_1|^{-1}r\Delta\left(\frac{1}{r^{n}}\int_{\partial B_r}{u}(x+y)\ud\sigma_r(y)\right)&\\
	&=|\partial B_r|^{-1}\Delta\left(\int_{0}^{r}\left(\int_{\partial B_\rho(x)} {u}(z)\ud\sigma_r(y)\right)\ud\rho\right),&
	\end{align*}
	which, by comparing with \eqref{Jm5}, implies
	\begin{align*}
	{\partial^{(1)}_r}u(r)
	&=\frac{1}{r^{n-1}}\Delta\left(\int_{0}^{r}\rho^{n-1}\left(|\partial B_1|^{-1}\int_{\partial B_\rho(x)} {u}(z)\ud\sigma_r(y)\right)\ud\rho\right)&\\
	&=\frac{1}{r^{n-1}}\Delta\left(\int_{0}^{r}\rho^{n-1}\overline{u}(\rho)\ud\rho\right).&
	\end{align*}
	Hence, we conclude
	\begin{equation}
	{\partial^{(1)}_r}\left(r^{n-1}{\partial^{(1)}_r}\overline{u}(r)\right)=\Delta(r^{n-1}\overline{u}(r)),
	\end{equation}
	which, by taking the radial derivative again, implies that for all $r>0$,
	\begin{equation*}
	r^{n-1}{\partial_r^{(2)}}\overline{u}+(n-1)r^{n}{\partial_r^{(1)}}\overline{u}=r^{n-1}\Delta\overline{u};
	\end{equation*}
	this finishes the proof of \eqref{Jm4}.
	
	Finally, the proof of the lemma follows combining \eqref{Jm2} and \eqref{Jm4}.
\end{proof}

Next, we need to prove the formula that relates the Laplacian of a function with the Laplacian of its Kelvin transform. 
To simplify our computations, we use curvilinear coordinates based on the inversion about a sphere.

\begin{proof}[Proof of \eqref{laplaciankelvin}]
	Let us define a new coordinate system given by $\xi=\mathcal{I}_{x_0,\mu}(x)$ where $\mathcal{I}_{x_0,\mu}$ is the inversion about $\partial B_\mu(x_0)$ defined previously. Hence, $\xi=(\xi_1,\dots,\xi_n)$ is a system of orthogonal curvilinear coordinates about $x_0\in\mathbb{R}^n$, whose metric tensor is given by  $g_{ik}={|\xi|^{-4}}\delta_{jk}$. This is a consequence of the following identities
	\begin{equation}\label{metrictensor}
	{\partial_{x_k} \xi_j}=|x|^{-2}\left(\delta_{jk}-2\frac{x_jx_k}{|x|^2}\right)
	\quad {\rm and} \quad
	\sum_{l=1}^{n}{\partial_{x_j} \xi_l}{\partial_{x_k} \xi_l}=|x|^{-4}\delta_{jk}.
	\end{equation}
	Now we compute the expression for the Laplacian in this new coordinate system.
	For this, we use the Lam\'{e} coefficients associated to metric tensor \eqref{metrictensor} given by
	\begin{equation}\label{lame}
	h_j=\sqrt{g_{jj}(\xi_j,\xi_j)}={|\xi|^{-2}}.
	\end{equation}
	Using these coefficients, remember the following formula for the Laplacian of a function,
	\begin{equation*}
	\Delta u_{x_0,\mu}(x)=\frac{1}{\prod_{k=1}^{n}h_k}\displaystyle\sum_{j=1}^{n}{\partial_{x_j} \xi_l}\left(\frac{\prod_{k=1}^{n}h_k}{h^2_j}{\partial_{\xi_j} u}\right),
	\end{equation*}
	which combined with \eqref{lame} provides
	\begin{equation}\label{georgia}
	\Delta u_{x_0,\mu}(x)=|\xi|^{2n}\displaystyle\sum_{j=1}^{n}{\partial_{\xi_j}}\left({|\xi|^{4-2n}}{\partial_{\xi_j} u}\right).
	\end{equation}
	On the other hand, notice that
	\begin{align}\label{macedonia}
	|\xi|^{n-2}\displaystyle\sum_{j=1}^{n}{\partial_{\xi_j}}\left({|\xi|^{4-2n}}\partial_{\xi_j}u\right)&=2\displaystyle\sum_{j=1}^{n}{\partial_{\xi_j}}\left({|\xi|^{2-n}}u\right){\partial_{\xi_j} u}+{|\xi|^{2-n}}\displaystyle\sum_{j=1}^{n}{\partial_{\xi_j}^{(2)}u}&\\\nonumber&=\displaystyle\sum_{j=1}^{n}{\partial_{\xi_j}^{(2)}}\left({|\xi|^{2-n}u}\right)-u\displaystyle\sum_{j=1}^{n}{\partial_{\xi_j}^{(2)}}\left({|\xi|^{2-n}}\right)&\\\nonumber
	&=\Delta(|\xi|^{2-n}u)-u\Delta(|\xi|^{2-n}).&
	\end{align}
	Finally, observing that $\Delta(|\xi|^{2-n})=0$, and substituting \eqref{macedonia} into  \eqref{georgia}, we get 
	\begin{align*}
	\Delta u_{x_0,\mu}(x)=|\xi|^{n+2}\Delta(|\xi|^{2-n}u)=K_{x_0,\mu}(x)^{n+2}\Delta (u\circ \mathcal{I}_{x_0,\mu}),
	\end{align*}
	which concludes the proof.
\end{proof}

There is a more conceptual way of proving this fact. 
To this end, we use the transformation law for two conformal metrics (see \cite{MR1333601}).

\begin{proof}[Proof of \eqref{laplaciankelvin}]
	First, by the invariance of the Laplacian under the conformal euclidean group, we may assume without loss of generality $x_0=0$ and $\mu=1$.
	Let us fix the notation $I:=I_{0,1}$.
	Explicitly, we consider the conformal diffeomorphism $\mathcal{I}:\mathbb{R}^{n} \setminus\{0\} \rightarrow \mathbb{R}^{n} \setminus\{0\}$ given by the standard the inversion $\mathcal{I}(x)=x|x|^{-2}$ about the unit sphere.
	
	We define a new metric using the pullback of the standard Euclidean metric, denoted by $\delta_0$, under the diffeomorphism $I$, that is, $\hat{g}=\mathcal{I}^*\delta_0$. 
	Using spherical coordinates $(r,\sigma)\in(0,\infty)\times\mathbb{S}^{n-1}$, where $r=|x|$ and $\sigma=x|x|^{-1}$, a direct computation shows 
	\begin{equation*}
	I^{*}\delta_0=\left(\ud \frac{1}{r}\right)^{2}+\left(\frac{1}{r}\right)^{2}(\ud \sigma)^{2}=\frac{(\ud r)^{2}+r^{2}(\ud \sigma)^{2}}{r^{4}},
	\end{equation*}
	which implies $\mathcal{I}^*\delta_0=|x|^{-4}\delta_0$.
	Using the transformation law for the conformal Laplacian of the metric $\hat{g}$, we find
	\begin{equation*}
	L_{\hat{g}}(u)=\phi^{-\frac{n+2}{n-2}} L_{g}(\phi u),
	\end{equation*}
	where $\phi=r^{2-n}$. However, $\mathcal{I}:\left(\mathbb{R}^{n} \backslash\{0\},\hat{g}\right) \rightarrow\left(\mathbb{R}^{n} \backslash\{0\},\delta_0\right)$ is, by construction, an isometry, then the scalar curvature of the Riemannian metric $\hat{g}$
	is zero.
	Therefore,
	\begin{equation*}
	L_{\hat{g}}(u \circ \mathcal{I})=(\Delta u) \circ \mathcal{I}.
	\end{equation*}
	Eventually, we get
	\begin{equation*}
	(\Delta u) \circ \mathcal{I}=r^{n+2} L_{g}(\phi u\circ \mathcal{I})=r^{n+2} \Delta\left(r^{2-n} u\circ \mathcal{I}\right),
	\end{equation*}
	which finishes the proof.
\end{proof}

The next formula provides an expression relating the bi-Laplacian of a function with the bi-Laplacian of its Kelvin transform. 
The strategy of the proof is to iterate \eqref{laplaciankelvin}.

\begin{proof}[Proof of \eqref{biharmonickelvin}]
	Since the bi-Laplacian is invariant under translations and dilations, we may again suppose without loss of generality that $x_0=0$ and $\mu=1$, thus as before we set $\mathcal{I}(x)=\mathcal{I}_{0,1}(x)$. Then, denoting $\mathcal{I}_{0,1}(x)=x^{*}$ and $u_{x_0,\mu}=\widetilde{u}$, we can write $\widetilde{u}(x)=|x|^{2}|x|^{2-n}u(x^{*})$, which provides
	\begin{align*}
	\Delta\left(|x|^{2}|x|^{2-n}u(x^{*})\right)
	&=\Delta(|x|^{2})|x|^{2-n}u(x^{*}+2\nabla\left(|x|^{2}\right)\nabla\left(|x|^{2-n}u(x^{*})\right)+|x|^{2}\Delta\left(|x|^{2-n}u(x^{*})\right).&
	\end{align*}
	In addition, we get $\Delta(|x|^{2})=2n$, which combined with the last identity, gives us
	\begin{align*}
	\Delta \widetilde{u}(x)=2n|x|^{2-n}u(x^{*})+ 2\nabla(|x|^{2})\nabla\left(|x|^{2-n}u(x^{*})\right)+\Delta\left(|x|^{-n}u(x^{*})\right).
	\end{align*}
	Again applying the Laplacian on the last equality, we find
	\begin{align}\label{ontario}
	\Delta^{2}\widetilde{u}(x)&=\Delta\left(2n|x|^{2-n}u(x^{*})\right)+\Delta\left(2\nabla\left(|x|^2\right)\nabla\left(|x|^{2-n}u(x^{*})\right)\right)+\Delta\Delta\left(|x|^{-n}u(x^{*})\right)&\\\nonumber
	&:=L_1+L_2+L_3&.
	\end{align}
	
	In the sequel, let us estimate each term in \eqref{ontario}.
	
	\noindent (i) For the $L_1$, using \eqref{laplaciankelvin}, we have
	\begin{equation}\label{l1}
	\Delta\left(2n|x|^{2-n}u(x^{*})\right)=
	2n|x|^{-(n+2)}\Delta u(x^{*}).
	\end{equation}
	\noindent (ii) For $L_3$, a direct computation using \eqref{laplaciankelvin}, implies
	\begin{align}\label{l3}
	&\Delta\Delta\left(|x|^{-n}u(x^{*})\right)&\\\nonumber
	&=\Delta\left(|x|^{-2}\right)\Delta\left(|x|^{2-n}u(x^{*})\right)+2\nabla\left(|x|^{-2}\right)\nabla\left(\Delta\left(|x|^{2-n}u(x^{*})\right)\right)+|x|^{-2}\Delta\left(\Delta\left(|x|^{2-n}u(x^{*})\right)\right)&\\\nonumber
	&=8|x|^{-(n+2)}\Delta\left(|x|^{2-n}u(x^{*})\right)-2n|x|^{-(n+2)}\Delta\left(|x|^{2-n}u(x^{*})\right)+4|x|^{-4}x\nabla\left(\Delta\left(|x|^{2-n}u(x^{*})\right)\right)&\\\nonumber
	&+|x|^{-(n+4)}\Delta\left(\Delta\left(|x|^{2-n}u(x^{*})\right)\right),&
	\end{align}
	where we have used $\nabla(|x|^{-2})=-2x|x|^{-4}$ and $\Delta(|x|^{-2})=(8-2n)|x|^{-4}$.
	
	\noindent (iii) For $L_2$, we have 
	\begin{align}\label{paris}
	\Delta\left(2\nabla\left(|x|^2\right)\nabla\left(|x|^{2-n}u(x^{*})\right)\right)
	&=8\nabla x\nabla\left(|x|^2\right)\nabla\nabla \left(|x|^{2-n}u(x^{*})\right)+4x\nabla\left(|x|^{-(n+2)}\Delta u(x^{*})\right).&
	\end{align}
	On the other hand, again by \eqref{laplaciankelvin}, it holds
	\begin{align}\label{lisbon}
	\nabla x\nabla\left(|x|^2\right)\nabla\left(\nabla \left(|x|^{2-n}u(x^{*})\right)\right)=\Delta\left(|x|^{2-n}u(x^{*})\right)=|x|^{-(n+2)}\Delta u(x^{*})
	\end{align}    
	and
	\begin{align}\label{london}
	\nabla\left(|x|^{-(n+2)}\Delta u(x^{*})\right)
	&=4x|x|^{-(n+4)}\Delta u(x^{*})+|x|^{4}\nabla\left(|x|^{2-n}\Delta u(x^{*})\right).&
	\end{align}
	Thus, substituting \eqref{lisbon}, \eqref{london} into \eqref{paris}, it follows
	\begin{align}\label{l2}
	\Delta\left(2\nabla\left(|x|^2\right)\nabla\left(|x|^{2-n}u(x^{*})\right)\right)
	&=-8x|x|^{-(n+4)}\Delta u(x^{*})+4|x|^{4}\nabla\left(|x|^{2-n}\Delta u(x^{*})\right).
	\end{align}
	Eventually, by substituting \eqref{l1}, \eqref{l3}, and \eqref{l2} into \eqref{ontario}, we conclude the proof.
\end{proof}


\end{document}